\documentclass[10pt]{article}
\usepackage{graphicx}

\usepackage{amsmath, amsthm, amssymb}
\usepackage{setspace}
\usepackage{graphicx}
\def\R{{\mathbb R}}
\def\N{{\mathbb N}}

\def\bnu{\boldsymbol\nu}
\def\bpsi{\boldsymbol\psi}
\def\btau{\boldsymbol\tau}
\def\bdeta{\boldsymbol\eta}
\def\bphi{\boldsymbol\varphi}
\def\bsigma{\boldsymbol\sigma}
\newtheorem {definition} {Definition}
  \newtheorem {remark}     {Remark}
  
  \newtheorem {theorem}    {Theorem}
  \newtheorem {lemma}      {Lemma}
  \newtheorem {corollary}  {Corollary}
  \newtheorem {proposition}{Proposition}

  \def\soft#1{\leavevmode\setbox0=\hbox{h}\dimen7=\ht0\advance
    \dimen7 by-1ex\relax\if t#1\relax\rlap{\raise.6\dimen7
    \hbox{\kern.3ex\char'47}}#1\relax\else\if T#1\relax
    \rlap{\raise.5\dimen7\hbox{\kern1.3ex\char'47}}#1\relax
    \else\if d#1\relax\rlap{\raise.5\dimen7\hbox{\kern.9ex
    \char'47}}#1\relax\else\if D#1\relax\rlap{\raise.5\dimen7
    \hbox{\kern1.4ex\char'47}}#1\relax\else\if l#1\relax
    \rlap{\raise.5\dimen7\hbox{\kern.4ex\char'47}}#1\relax
    \else\if L#1\relax\rlap{\raise.5\dimen7\hbox{\kern.7ex
    \char'47}}#1\relax\else\message{accent \string\soft
    \space #1 not defined!}#1\relax\fi\fi\fi\fi\fi\fi} 
  
  \def\ocirc#1{\ifmmode\setbox0=\hbox{$#1$}\dimen0=\ht0
    \advance\dimen0 by1pt\rlap{\hbox to\wd0{\hss\raise\dimen0
    \hbox{\hskip.2em$\scriptscriptstyle\circ$}\hss}}#1\else
    {\accent"17 #1}\fi}

\usepackage[a4paper, total={6in, 9in}]{geometry}

\begin{document}
\title{\sc{Existence of a weak solution to a fluid-elastic structure interaction problem with the Navier slip boundary condition}}

\author{ Boris Muha \thanks{Department of Mathematics,
    University of Zagreb, 10000 Zagreb, Croatia,
    borism@math.hr} \and Sun\v{c}ica \v{C}ani\'{c}\thanks{Department of Mathematics,
    University of Houston, Houston, Texas 77204-3476,
    canic@math.uh.edu} }

\date{}

\maketitle

\begin{abstract}
We study a nonlinear, moving boundary fluid-structure interaction (FSI) problem between an incompressible, viscous Newtonian fluid, modeled by the 2D Navier-Stokes equations, and an elastic structure modeled by the shell or plate equations. The fluid and structure are coupled via the {\em Navier slip boundary condition} and balance of contact forces at the fluid-structure interface. The slip boundary condition might be more realistic than the classical no-slip boundary condition in situations, e.g., when the structure is ``rough'', and in modeling FSI dynamics near, or at a contact. Cardiovascular tissue and cell-seeded tissue constructs, which consist of grooves 
in tissue scaffolds that are lined with cells, are examples of ``rough'' elastic interfaces interacting with an incompressible, viscous fluid. The problem of heart valve closure is an example of a FSI problem with a contact involving elastic interfaces.  We prove the existence of a weak solution to this class of problems by designing a constructive proof based on the time discretization via operator splitting. This is the first existence result for fluid-structure interaction problems involving elastic structures satisfying the Navier slip boundary condition. 
\end{abstract}

\section{Introduction}
We study a nonlinear, moving boundary fluid-structure interaction (FSI) problem between a viscous, incompressible Newtonian fluid,
modeled by the Navier-Stokes equations, and an elastic shell or plate. 
The fluid and structure are coupled through two coupling conditions: the Navier slip boundary condition and the continuity of contact forces
at the fluid-structure interface. The Navier slip boundary condition states that the difference, i.e., the slip between the tangential components of the fluid  and structure 
velocities is proportional to the tangential component of the fluid normal stress evaluated at the fluid-structure interface, while the normal components
of the fluid and structure velocities are continuous. 
The main motivation  for using the Navier slip boundary condition
comes from fluid-structure interaction (FSI) problems involving elastic structures with ``rough'' boundaries, and from studying FSI problems near a contact.

FSI problems involving elastic structures with rough boundaries appear, for example when studying FSI between blood flow and 
cardiovascular tissue, whether natural or bio-artificial, which is lined with cells that are in direct contact with blood flow. 
Bio-artificial vascular tissue constructs (i.e., vascular grafts) often times involve cells
seeded on tissue scaffolds with grooved microstructure, which interacts with blood flow. 
See Figure~\ref{motivation}. To filter out the small scales of the rough fluid domain boundary, 
effective boundary conditions based on the Navier slip condition have been used 
in various applications, see e.g., a review paper by Mikeli\'{c} \cite{MikelicReview}, 
and \cite{Masmoudi2010,Bucur2010,JagerMikelic2001}. 
Instead of using the no-slip condition at the groove-scale on the rough boundary, the Navier slip condition is applied at the smooth boundary instead.
\begin{figure}[ht]
\centering{
\includegraphics[scale=0.3]{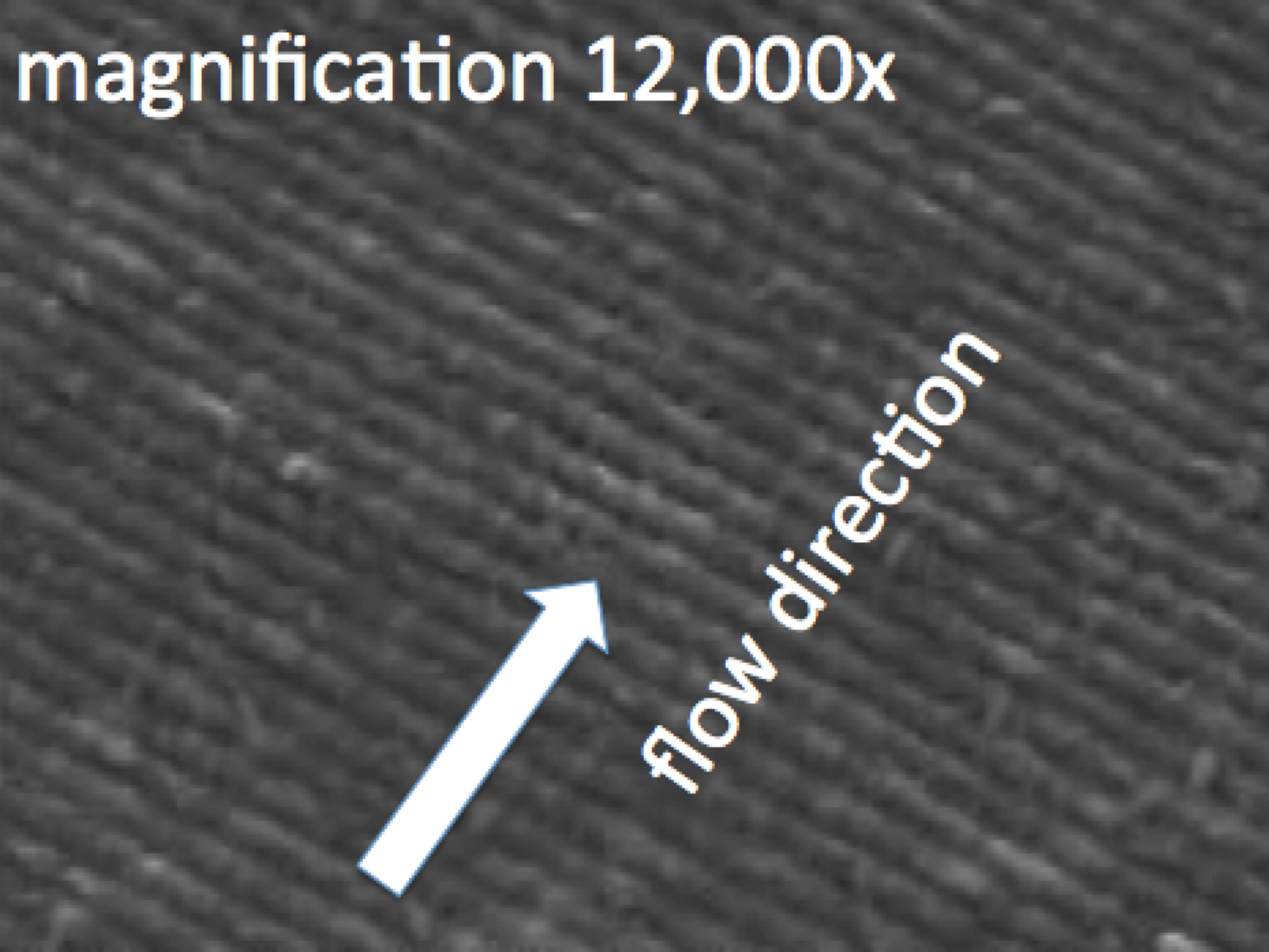}
\hskip 0.3in
\includegraphics[scale=0.42]{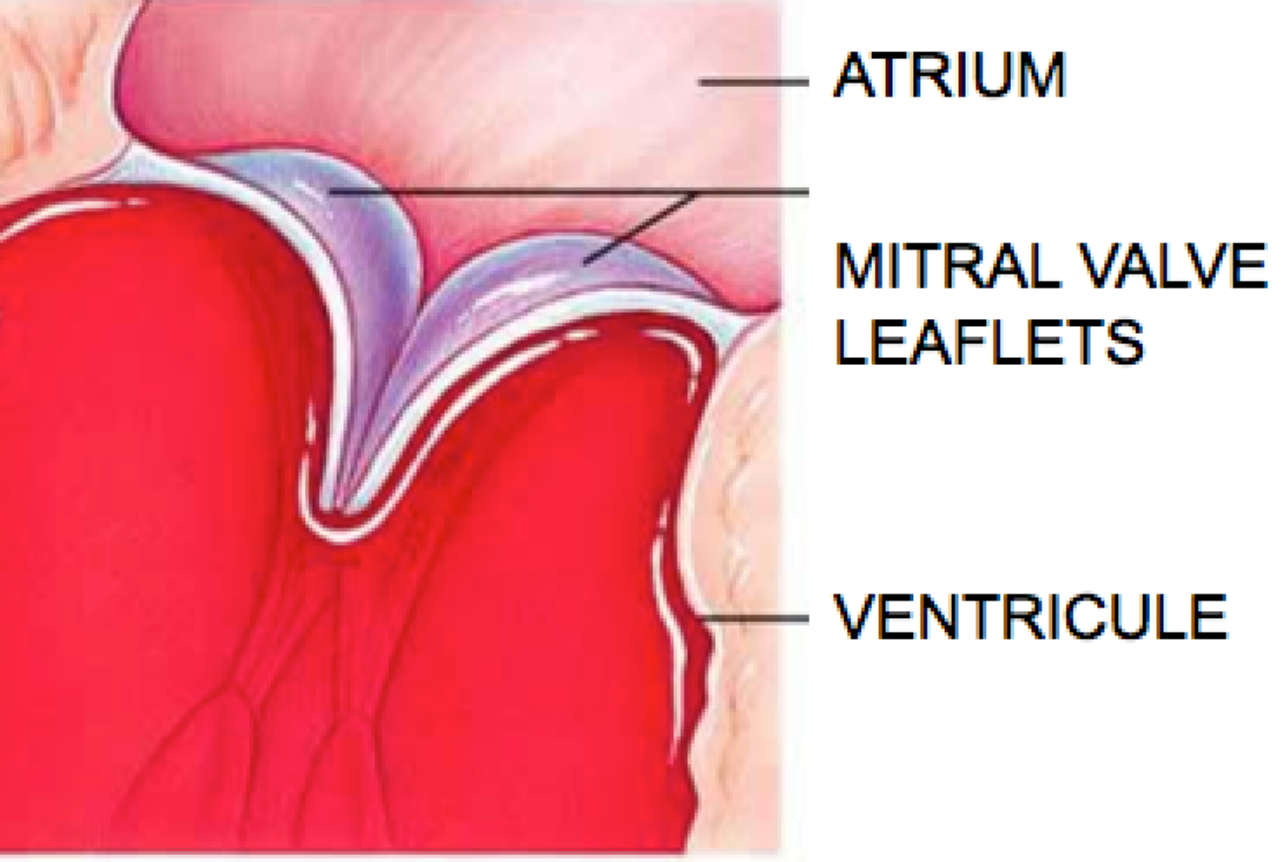}
}
\caption{\small{Two examples of fluid-elastic structure interaction problems where the Navier slip condition may be more appropriate. 
Left: An image of a nano-patterned film surface for vascular tissue engineering  \cite{Zorlutuna}
giving rise to a ``rough'' surface interacting with blood flow.
 The grooves in this tissue construct are to be lined with cells, producing a bio-artificial vascular graft for a replacement of diseased arteries. 
 The grooves are perpendicular to the direction of blood flow. 
 Right: A sketch of a closed mitral valve interacting with blood flow. Modeling the contact between leaflets with the no-slip condition 
gives rise to leaky valves due to the no-slip condition paradox in collision of smooth bodies \cite{Berlyand,HilTak2,HilTak,Starovoitov}.}}
\label{motivation}
\end{figure}

Another motivation for using the Navier slip boundary condition comes from studying problems near  or at a contact (or collision). It has been shown recently
that the no-slip condition is not a realistic physical  condition to model contact between smooth rigid bodies
immersed in an incompressible fluid 
\cite{Berlyand,HilTak2,HilTak,Starovoitov}. 
It was shown that two smooth rigid bodies cannot touch each other if the no-slip boundary condition is considered.
One solution to this no-collision paradox is to consider bodies with ``non-smooth'' boundaries, in which case collisions can occur \cite{GVHil}. 
The other explanation for the no-collision paradox is that the no-slip boundary condition does not describe near-contact dynamics well, and a new model and/or a different boundary condition, such as for example, the Navier slip boundary condition, need to be employed to model 
 contact between bodies/structures interacting while immersed in an incompressible, viscous fluid \cite{NePen}. 
Examples include applications in cardiovascular sciences, for example,  modeling the closure of  heart valves. 
It is well known that numerical simulation of heart valve closure suffers from the ``numerical'' leakage of blood through a ``closed'' heart valve
whenever the no-slip boundary condition is used. Different kinds of ``gap'' boundary conditions have been used to get around this difficulty, see e.g., \cite{Doyle}. 
 Considering the Navier slip boundary condition near or at the closure would provide a more realistic modeling of the problem. 

 From the mathematical analysis point of view, the fist step in the direction of studying the Navier slip boundary condition near or at a contact
 was made by Neustupa and Penel \cite{NePen} who proved that when the no-slip boundary condition is replaced with the slip boundary condition, collision can occur for a prescribed movement of {\em rigid bodies}. 
 Recently, G\'erard-Varet and Hillairet considered an FSI problem involving a movement of a rigid solid 
 immersed in an incompressible Navier-Stokes flow with the slip boundary condition and proved the existence of a weak solution up to collision \cite{GVHil3}. 
 In a subsequent work \cite{GVHil2} they proved that prescribing the slip boundary condition on both the rigid body boundary and the boundary of the domain, allows 
 collision of the rigid body with the boundary. 
 The existence of a {\sl global} weak solution which permits collision of a ``smooth'' rigid body with a ``smooth'' fluid domain boundary was proved in  \cite{NecasovaSlip}. 
For completeness, we also mention several recent works where the slip boundary condition was considered in various 
existence results  for FSI problems involving rigid bodies and Newtonian fluids \cite{MT3,PlanasSueur,WangFluidSolid}. 
All the above-mentioned works consider FSI between rigid bodies and an incompressible, viscous fluid. 
 To the best of our knowledge there are no existence results for non-linear moving boundary FSI problems involving elastic structures
 satisfying the Navier slip boundary condition.  
 The present work is the {\em  first existence result involving the 
Navier slip boundary condition for a fluid-structure interaction problem
with elastic structures}. 

Classical FSI problems with the no-slip boundary condition have been extensively studied from both the analytical and numerical point of view (see e.g. \cite{FSIforBio,GaldiHandbook,BorSun} and the references within). 
Earlier works have focused on problems in which the coupling between the fluid and structure
was calculated at a fixed fluid domain boundary,
see \cite{Gunzburger}, and  \cite{BarGruLasTuff2,BarGruLasTuff,KukavicaTuffahaZiane},
where an additional nonlinear coupling term was added and calculated at a fixed fluid interface.
A study of well-posedness for FSI problems between an incompressible, viscous fluid and an elastic/viscoelastic structure satisfying the no-slip boundary condition,
with the coupling evaluated at a moving interface, started with the result of daVeiga  \cite{BdV1}, where existence of a strong solution was obtained locally in time 
for an interaction between a $2D$ fluid and a $1D$ viscoelastic string, assuming periodic boundary conditions.
This result was extended by Lequeurre in \cite{Leq11,Leq13}, where the existence of a unique, local in time, strong solution
for any data, and the existence of a global strong solution for small data, was proved in the case when the structure was modeled as a clamped viscoelastic beam. 

D.~Coutand and S.~Shkoller proved existence, locally in time, of a unique, regular solution for
an interaction between a viscous, incompressible fluid in $3D$ and a $3D$ structure, immersed in the fluid,
where the structure was modeled by the equations of linear  elasticity satisfying no-slip at the interface \cite{CSS1}. 
In the case when the structure (solid) is modeled by a linear wave equation, I. Kukavica et al. proved the existence, locally in time, of a strong solution, 
assuming lower regularity  for the initial data \cite{Kuk,Kuk2,IgnatovaKukavica}. 
A similar result for compressible flows can be found in \cite{KukavicaNSLame}. 
In \cite{raymond2013fluid} Raymod et al. considered 
a FSI problem between a linear elastic solid immersed in an incompressible viscous fluid, 
and proved the existence and uniqueness of a strong solution.
All the above mentioned existence results for strong solutions are local in time.
Recently, in \cite{ignatova2014well} a global existence result for small data was obtained for a similar  moving boundary FSI problem 
but with additional interface and structure damping terms.

In the context of weak solutions incorporating the no-slip condition, 
the following results have been obtained. 
Existence of a weak solution for a FSI problem between a $3D$ incompressible, viscous fluid 
and a $2D$ viscoelastic plate was shown by Chambolle et al. in \cite{CDEM}, while
Grandmont improved this result in \cite{CG}  to hold for
a $2D$ elastic plate. These results were extended to a more general geometry in \cite{LenRuz},
and to a non-Newtonian shear dependent fluid in \cite{LukacovaCMAME}.
In these works existence of a weak solution was proved for as long as the elastic boundary does not touch "the bottom" (rigid)
portion of the fluid domain boundary. 

Muha and \v{C}ani\'{c}  recently proved the existence of  a weak solution to a class of FSI problems modeling the flow
of an incompressible, viscous, Newtonian fluid flowing through a 2D cylinder whose lateral wall was modeled by 
either the linearly viscoelastic, or by the linearly elastic Koiter shell equations \cite{BorSun}, assuming nonlinear coupling at
the deformed fluid-structure interface. These results were extended by the same authors to a 3D FSI problem involving a cylindrical  Koiter shell \cite{muha2013nonlinear},
and to a semi-linear cylindrical Koiter shell \cite{BorSunNonLinear}. The main novelty in these works was a design of a constructive existence proof
based on the Lie operator splitting scheme, which has been used in numerical simulation of several FSI problems  \cite{GioSun,MarSun,BorSun,Martina_Biot,Lukacova,LukacovaCMAME}, and has proven to be a robust method for a design of constructive existence proofs
for an entire class of FSI problems.

In the present work a further non-trivial extension of the Lie operator splitting scheme is introduced to deal with the Navier slip boundary condition and
with the non-zero longitudinal displacement of the structure. Dealing with the Navier slip condition and non-zero longitudinal displacement introduces 
several mathematical difficulties. In contrast with the no-slip boundary condition which ``transmits'' the regularizing mechanism 
by the viscous fluid dissipation
onto the fluid-structure interface, in the Navier slip condition the tangential components of the fluid and structure velocities are no longer continuous, and thus information is lost in the tangential direction.
As a result, new compactness arguments had to be designed in the existence proof to control the tangential velocity components at the interface.
The compactness arguments are based on Simon's characterization of compactness in $L^2(0,T;B)$ spaces \cite{Simon},  
and on interpolation of the classical Sobolev spaces with real exponents $H^s$ (or alternatively Nikolskii spaces $N^{s,p} \cite{Simon2})$. This is new. If we had worked with continuous energy estimate, we would have been able
to obtain structure regularity in the standard space $L^\infty(0,T;H^2(\Gamma))$. 
The time-discretization via operator splitting, however, enabled us to obtain an additional estimate in time
that is due to the dissipative term in the backward Euler approximation of the time derivative of the structure velocity. 
The uniform boundedness of this term enabled us to obtain structure regularity in $H^s(0,T;H^2(\Gamma))$, $s < 1/2$.
This was crucial for the existence proof. This approach in general brings new information about the time-behavior of weak solutions
to the elastic structure problems.

Furthermore, to deal with the non-zero longitudinal displacement and keep the behavior of fluid-structure interface ``under control'', 
we had to consider higher-order terms in the structure model
given by the bending rigidity of {\em shells}. 
The linearly elastic membrane model was not tractable. 
Due to the non-zero longitudinal displacement additional nonlinearities appear in the problem
that track the  geometric quantities such as the fluid-structure interface surface measure, 
the interface tangent and normal, and the change in the moving fluid domain measure (given by the Jacobian of the ALE mapping
mapping the moving domain onto a fixed, reference domain). 
These now appear explicitly in the weak formulation of the problem, and cause various difficulties 
 in the analysis. This is one of the reasons why our existence result is local in time, i.e., it holds for the time interval $(0,T)$ for
 which we can guarantee that the fluid domain does not degenerate
 in the sense that the ALE mapping remains injective in time
 as the fluid domain moves, and the Jacobian of the ALE mapping remains strictly positive, see Figure~\ref{degeneration}. 
 Therefore, in this manuscript we prove the existence, locally in time,  of a weak solution, 
 to a nonlinear moving-boundary problem between an incompressible, viscous Newtonian fluid and an elastic shell or plate, satisfying
 the Navier slip condition at the fluid-structure interface, and balance of forces at the fluid-structure interface.

\section{Problem description}
We study the flow of an incompressible, viscous fluid through a 2D fluid domain whose boundary contains an elastic,
thin structure. 
The fluid and structure are fully coupled through two coupling conditions:  the Navier slip boundary condition,
and the dynamic coupling condition describing the balance of forces at the elastic structure interface. 
The flow is driven by the data, which includes the case of the 
 time-dependent dynamic pressure data prescribed at the ``inlet and outlet'' portion of the fixed boundary, denoted in 
 Figure~\ref{fig:domain} as $\Gamma_1$ and $\Gamma_3$. 
The reference fluid domain, denoted by $\Omega \subset \R^2$, is  considered to be a polygon with angles less than of equal to $\pi$,
see Figure~\ref{fig:domain}. 
If we denote by $\Gamma_i, i = 0,\dots,m$ the faces of $\Omega$, without the loss of generality we 
can assume that $\Gamma_0$ is compliant, and $\Gamma_1,\dots,m$ is the rigid portion. We denote 
the rigid portion by ${\Sigma = \bigcup_{i=1}^m\overline{\Gamma_i}}$, and the compliant portion by $\Gamma=\Gamma_0 = (0,L)$.
\begin{figure}[ht]
\centering{
\includegraphics[scale=0.4]{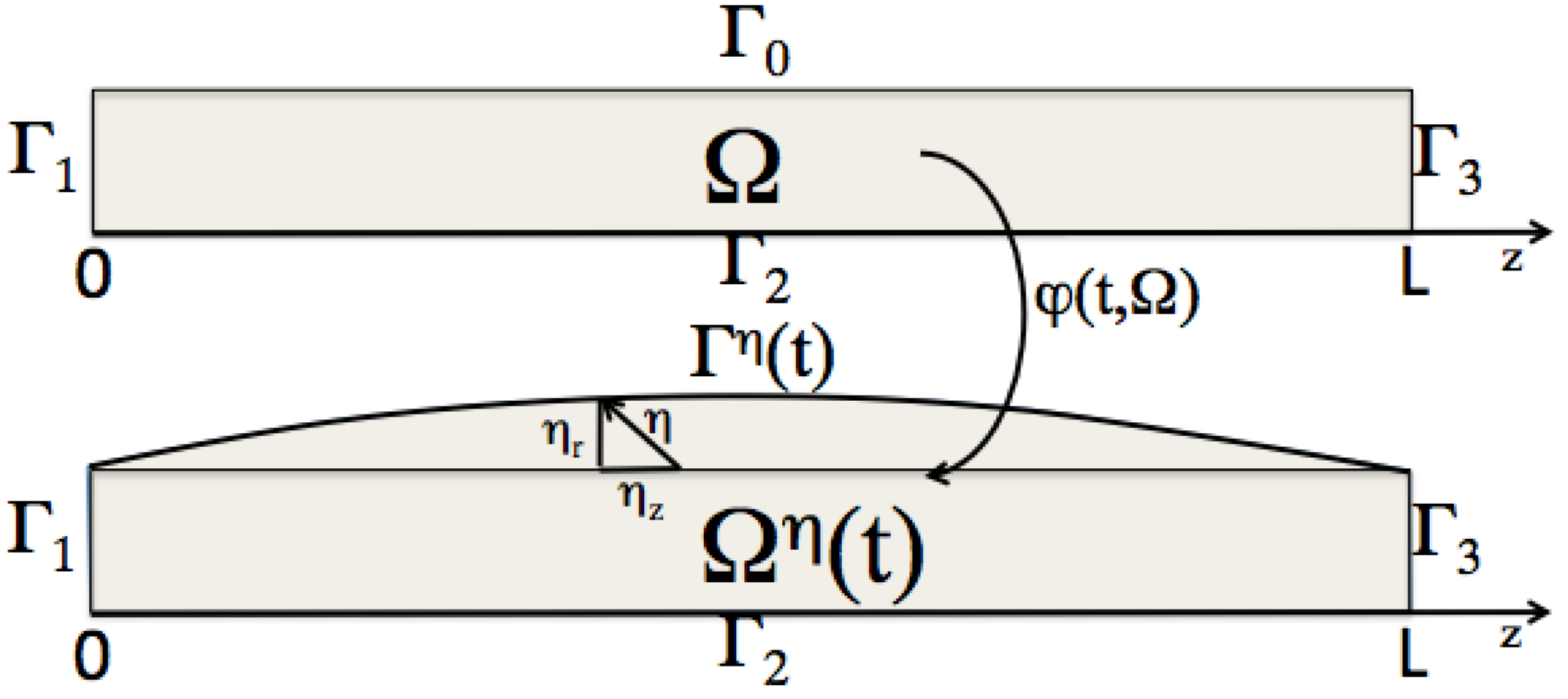}
}
\caption{An example of fluid and structure domains.}
\label{fig:domain}
\end{figure}

As the fluid flows through the compliant domain, the elastic part of the boundary deforms, giving rise to a time-dependent fluid domain
which is not known {\sl a priori}. 
We denote by $\bphi(t,.):\Omega\to \R^2$, $t\in [0,T)$ the time-dependent deformation of the fluid domain determined
by the interaction between fluid flow and the elastic part of the fluid domain boundary. 
We will be assuming that $\bphi$ is such that the rigid portion $\Sigma$ of the fluid domain remains fixed,
and that  the fluid domain does not degenerate in the sense that
 the elastic structure does not touch any part of the boundary during deformation. 
More precisely, we will be assuming that $\bphi$ is a $C^1$ diffeomorphism such that 
$$\bphi_{|\Sigma}={\bf id} \ {\rm and}\ {\rm det}\nabla\bphi(t,{\bf x})>0, \ (t,{\bf x})\in [0,T)\times\overline{\Omega}.$$
We denote the displacement of the elastic part of the boundary $\Gamma_0$ by 
$\bdeta(t,{\bf x})=\bphi(t,{\bf x})-{\bf x},\;{\bf x}\in \Gamma_0$. 
Since $\Gamma_0$ can be identified by the interval $(0,L)$, as mentioned above, $\bdeta$  is defined
as a mapping
$\bdeta:[0,T)\times [0,L]\to \R^2$, with
\begin{equation}\label{clamped}
\bdeta(t,z)=(\eta_z(t,z),\eta_r(t,z)),\; z\in [0,L],\; \eta(0)=\partial_z\eta(0)=\eta(L)=\partial_z \eta(L)=0.
\end{equation}
Here $\eta_z$ and $\eta_r$ denote the tangential and normal components of displacement with respect to the reference configuration $\Gamma_0$, respectively,
and the last set of conditions in \eqref{clamped} state that the elastic structure is clamped at the points at which it meets the 
rigid portion of the boundary $\Sigma$. 
We denote by $\Omega^{\eta}(t)=\bphi(t,\Omega)$ the deformed fluid domain at time $t$, 
and by $\Gamma^{\eta}(t)=\bphi(t,\Gamma_0)$ the corresponding deformed elastic part of the boundary of $\Omega^{\eta}(t)$. 
We chose to include $\bdeta$ as a superscript in the notation for the deformed fluid domain and for the deformed elastic structure to emphasize that 
they both depend on one of the unknowns in the problem, which is the structure displacement $\bdeta$.
The following notation will be useful in subsequent calculations.
The surface element of the deformed structure will be denoted by:
$$
d\Gamma^{\eta}=\sqrt{(1+\partial_z\eta_z(t,z))^2+(\partial_r\eta_r(t,z))^2}dz=S^{\eta}(t,z)dz,
$$
the tangent vector to the deformed structure will be denoted by ${\btau}^{\eta}(t,z)=\partial_z\bphi^{\eta}(t,z)$, 
and the outer unit normal 
on $\Gamma^\eta(t)$ at point $\bphi(t,z)$ will be denoted by $\bnu^{\eta}(t,z)$.

\if 1 = 0
\begin{remark}[About the domain] TODO Rewrite this*******
We can also look at the more general domain, for example a non-convex curvilinear polygon. The only thing where we need this regularity in our analysis is the regularity of the solution to the Dirichlet boundary value problem for the Laplace equation (see problem~\eqref{ALEDef}). However, since this is not in the focus of our study, in order to simplify the presentation we chose not to go the optimal results, but to work with the results from classical monograph of Grisvard~\cite{Grisvard2}.
\end{remark}
\fi

{\bf The fluid.} The fluid flow is governed by the Navier-Stokes equations for an incompressible, viscous fluid defined
on the family of time-dependent domains $\Omega^{\eta}(t)$:
\begin{equation}
\left .
\begin{array}{rcl}
\rho_F (\partial_t{\bf u}+{\bf u}\cdot\nabla{\bf u})&=&\nabla\cdot\bsigma,
\\
\nabla\cdot{\bf u}&=&0,
\end{array}
\right \}\ {\rm in}\ \Omega^{\eta}(t),\ t\in (0,T),
\label{NS}
\end{equation}
where $\rho_F$ denotes the fluid density, ${\bf u}$ is the fluid velocity,
$\boldsymbol\sigma=-p{\bf I}+2\mu{\bf D}({\bf u})$ is the fluid Cauchy stress tensor, $p$ is the fluid pressure,
$\mu$ is the kinematic viscosity coefficient, and ${\bf D}({\bf u})=\frac 1 2(\nabla{\bf u}+\nabla^{\tau}{\bf u})$ is the symmetrized gradient of ${\bf u}$.

The approach presented in this manuscript can handle different types of boundary conditions prescribed on the rigid boundary $\Sigma$. 
More precisely, on each face $\Gamma_i$ of the {\em rigid boundary} $\Sigma$ we prescribe one of the following four types of boundary conditions:
\begin{enumerate}
\item Dynamic pressure: $p+\frac{\rho_F}{2}|{\bf u}|^2=P_{i}, {\bf u}\cdot\btau=0$, prescribed on $\Gamma_i, i \in I$,
\item Velocity (no-slip): ${\bf u}={\bf 0}$, prescribed on $\Gamma_i, i \in II$,
\item The Navier slip boundary condition: ${\bf u}\cdot\bnu=0, {\bf u}\cdot\btau+\alpha_i\boldsymbol\sigma\bnu\cdot\btau=0$,
prescribed on $\Gamma_i, i \in III$,
\item The symmetry boundary condition: ${\bf u}\cdot\bnu=0,\; \partial_\nu u_{\tau}=0$, prescribed on $\Gamma_i, i\in IV$,
\end{enumerate}
where $I$, $II$, $III$ and $IV$ denote the subsets of the set of indices $\{1,\dots,m\}$ such that 
the boundary condition of type $1$, $2$, $3$ or $4$ 
is satisfied. 
We note that non-homogeneous boundary conditions can also be handled with additional care.

The problem is supplemented with the initial condition:
\begin{equation}\label{NSIC}
\begin{array}{c}
{\bf u}(0,.)={\bf u}_0.
\end{array}
\end{equation}

To close the problem, it remains to specify the boundary conditions on the elastic part of the boundary $\Gamma^\eta(t)$.
For this purpose we introduce the elastodynamics equations modeling the motion of the elastic structure, and the two-way coupling
between the structure and the fluid motion.

{\bf The structure.} The elasto-dynamics of {\bf thin structure} $\Gamma^\eta(t)$ 
will be given in terms of displacement 
with respect to the reference configuration $\Gamma=\Gamma_0=(0,L)$ (Lagrangian formulation). 
To include different shell models, we formulate the elastodynamics problem in terms of 
 a general  {\em continuous, self-adjoint, coercive, linear operator} ${\cal L}_e$, defined on $H_0^2(0,L)$, 
for which there exists a constant $c > 0$ such that 
\begin{equation}\label{coercivity}
\left<{\cal{L}}_e \bdeta,\bdeta \right> \ge c  \| \bdeta \|^2_{H_0^2(\Gamma)}, \quad \forall \bdeta \in H_0^2(\Gamma),
\end{equation}
where $\left< \cdot,\cdot \right>$ is the duality pairing between $H_0^2$ and $H^{-2}$.
The structure elastodynamics problem is then given by:
\begin{eqnarray}
\rho_{S} h  \partial_{tt} \bdeta = -{\cal L}_e \bdeta + {\bf f}, \quad & z\in \Gamma=(0,L),\ t\in (0,T),
\label{Koiter}
\\
\bdeta(t,0)=\partial_z\bdeta(t,0)=\bdeta(t,L)=\partial_z\bdeta(t,L)=0,\; & t\in (0,T),
\label{KoiterBC}
\\
\bdeta(0,z)=\bdeta_0,\; \partial_t\bdeta(0,z)={\bf v}_0,; & z\in \Gamma=(0,L).
\nonumber
\end{eqnarray}
where $\rho_S$ is the structure density, $h$ the elastic shell thickness, ${\bf f}$ is linear force density acting on the shell, $\bdeta=(\eta_z,\eta_r)$ is the shell displacement, and $\bdeta_0$ and ${\bf v}_0$ are the initial structure displacement and the initial structure velocity, respectively. 


The fluid and structure equations are coupled via the following two sets of {\bf coupling conditions}:
\begin{itemize}
\item {\bf The\  kinematic\  coupling\ condition (Navier slip condition):}
\begin{equation}\label{Coupling1a}
\begin{array}{l}
{\rm Continuity \ of \ normal\ velocity\;on\;}\Gamma^{\eta}(t):\\
\displaystyle{\partial_t\bdeta(t,z)\cdot\bnu^{\eta}(t,z)} = \boldsymbol{u}(\bphi(t,z))\cdot\bnu^{\eta}(t,z), 
\quad
(t,z) \in (0,T)\times\Gamma,
\\  \\
 {\rm The\  slip\ condition \ between\ the \ fluid\ and\ thin\ structure\;on\;}\Gamma^{\eta}(t):\\
\displaystyle{(\partial_t\bdeta(t,z)-{\bf u}(\bphi(t,z)))\cdot\btau^{\eta}(t,z)}
\\
\displaystyle{=\alpha\bsigma\big (\phi(t,z)\big )\bnu^{\eta}(t,z)\cdot\btau}^{\eta}(t,z),\quad (t,z)\in (0,T)\times\Gamma.
\end{array}
\end{equation}
\item {\bf The\  dynamic\  coupling\ condition:}
\begin{equation}
\label{Coupling1b}
\rho_{S} h \partial_{tt}\bdeta(t,z) =  -{\mathcal L}_e\bdeta(t,z) -S^{\eta}(t,z)\bsigma\big (\bphi(t,z)\big )\bnu^{\eta}(t,z),\; (t,z)\in (0,T)\times\Gamma,
\end{equation}
stating that the structure interface elastodynamics is driven by the jump in the normal stress across the interface, where we have assumed,
without the loss of generality, 
that the normal stress on the outside of the structure is equal to zero. The term $S^\eta$,
which multiplies the normal fluid stress $\bsigma \bnu^\eta$, is the Jacobian of the transformation between the Eulerian 
and Lagrangian formulations
of the fluid and structure problems, respectively.
\end{itemize}
Notice that there is no pressure contribution in the slip condition \eqref{Coupling1a}.
The pressure contributes only through the dynamic coupling condition \eqref{Coupling1b}.

In summary, we study the following problem.

\vskip 0.2in

\noindent
\framebox{
\parbox{6.0in}{\hskip 0.15in 

\noindent 
Find $({\bf u},\bdeta)$ such that the following holds:

\noindent
{\bf The fluid equations:}
\begin{equation}
\left .
\begin{array}{rcl}
\rho_F (\partial_t{\bf u}+{\bf u}\cdot\nabla{\bf u})&=&\nabla\cdot\bsigma,
\\
\nabla\cdot{\bf u}&=&0,
\end{array}
\right \}\ {\rm in}\ \Omega^{\eta}(t),\ t\in (0,T);
\label{NSP}
\end{equation}
{\bf The elastic structure (boundary conditions on $(0,T)\times\Gamma$):}
\begin{eqnarray}
\rho_{S} h \partial_{tt}\bdeta(t,z) +{\mathcal L}_e\bdeta(t,z) &=& -S^{\eta}(t,z)\bsigma\big (\bphi(t,z)\big )\bnu^{\eta}(t,z),
\label{ElasticP}
\\
\displaystyle{\partial_t\bdeta(t,z)\cdot\bnu^{\eta}(t,z)} &=& \boldsymbol{u}(\bphi(t,z))\cdot\bnu^{\eta}(t,z),
\\
\displaystyle{(\partial_t\bdeta(t,z)-{\bf u}(\bphi(t,z))\cdot\btau^{\eta}(t,z)} &=& \alpha\bsigma\big (\bphi(t,z)\big )\bnu^{\eta}(t,z)\cdot\btau^{\eta}(t,z),
\label{SlipP}
\end{eqnarray}
with
\begin{equation*}
\bdeta(t,0)=\partial_z\bdeta(t,0)=\bdeta(t,L)=\partial_z\bdeta(t,L)=0,\; t\in (0,T);
\end{equation*}
{\bf Boundary conditions on $\Sigma$:}
\begin{equation}
\begin{array}{rl}
p+\frac{\rho_F}{2}|{\bf u}|^2 =P_{i},\; {\bf u}\cdot\btau=0\; &{\rm on}\;\Gamma_i,\;i\in I,\\
{\bf u}={\bf 0}\; &{\rm on}\;\Gamma_i,\;i\in II,\\
{\bf u}\cdot\bnu =0, {\bf u}\cdot\btau + \alpha_i\boldsymbol\sigma\bnu\cdot\btau=0\; &{\rm on}\;\Gamma_i,\;i\in III,\\
{\bf u}\cdot{\bf n}=0,\; \partial_\nu u_{\tau}=0\; &{\rm on}\;\Gamma_i,\;i\in IV;
\end{array}
\label{ProblemBC}
\end{equation}
with $u_\tau$ denoting the tangental component of velocity $\bf u$.
\vskip 0.1in
\noindent
{\bf Initial conditions:}
\begin{equation}
{\bf u}(0,.)={\bf u}_0,\;\bdeta(0,.)=\bdeta_0,\; \partial_t\bdeta(0,.)={\bf v}_0.
\label{ProblemIC}
\end{equation}
}}
\vskip 0.1in
The initial data must satisfy the following {\bf compatibility conditions}:
\begin{itemize}
\item  The initial fluid velocity must satisfy:
\begin{equation}\label{FluidCC}
\begin{array}{ll}
{\bf u}_0\in L^2(\Omega_0)^2,\; \nabla\cdot{\bf u}_0=0,\; & {\rm in}\; \Omega^0,\\ 
{\bf u}_0\cdot\nu=0,\; &{\rm on}\;\Gamma_i,\;   i\in II\cup III\cup IV,\\
{\bf u}_0\cdot\bnu_0={\bf v}_0\cdot\bnu_0, \quad & {\rm on}\; \Gamma^0,
\end{array}
\end{equation}
where $\Omega^0=\Omega^{\eta}(0)$, $\Gamma^0=\Gamma^{\eta}(0)$, $\bnu_0=\bnu^{\eta}(0,.)$. 
\item The initial domain must be such that  there exists a diffeomorphism $\bphi^0\in C^1(\overline{\Omega})$
such that 
\begin{equation}\label{StructureCC}
\begin{array}{c}
\bphi^0(\Omega)=\Omega^0,\;\det\nabla\bphi^0>0,
(\bphi-{\bf I})_{|\Gamma}=\bdeta_0,
\end{array}
\end{equation}
and the initial displacement $\bdeta_0$ is such that 
\begin{equation}\label{eta0}
\|\bdeta_0\|_{H^{11/6}} \le c, \ {\rm where}\ c\ {\rm is \ small}.
\end{equation}
\end{itemize}
We aim at proving the existence of a weak solution to this nonlinear moving boundary problem.

Before we continue, we note that condition \eqref{eta0} on the smallness of the $H^{11/6}$ norm of $\bdeta_0$ is somewhat artificial,
and is stated for technical 
purposes. 
This condition simplifies the analysis presented in this paper
which uses Grisvards's regularity results for elliptic problems on polygonal domains $\Omega$,
which includes our reference domain. 
With some additional technicalities, we could have obtained the same existence result by 
considering the reference domain to be the initial configuration of the fluid domain, 
which may not be a polygon.
In that case we would not need condition \eqref{eta0}, but the existence proof would become more technical. 

\section{Weak formulation}
\subsection{A Formal Energy Inequality}
To motivate the solution spaces for the weak solution 
of problem \eqref{NSP}-\eqref{ProblemIC} we present here 
a preliminary version of the formal energy estimate,
which shows that  for any smooth solution of problem \eqref{NSP}-\eqref{ProblemIC}
the total energy of the problem is  bounded by the data of the problem.
The formal energy estimate is derived in a standard way, 
by multiplying equations~\eqref{NS} and~\eqref{Koiter} by a solution $\bf u$ and $\bdeta$, respectively, and integrating by parts.
 The coupling conditions \eqref{ElasticP} and \eqref{SlipP} are used at the boundary $\Gamma^\eta(t)$, 
 and boundary conditions \eqref{ProblemBC} are used on the fixed portion of the boundary $\Sigma=\cup \Gamma_i$. 
 We will use $u_\tau$ and $\eta_\tau$ to denote the tangential components of 
 the trace of fluid velocity and of displacement at the moving boundary, respectively. 
We obtain that any smooth solution $({\bf u},\bdeta)$ of problem~\eqref{NSP}-\eqref{ProblemIC} satisfies the following energy estimate:
\begin{equation}\label{EI_preliminary}
\begin{array}{c}
\displaystyle{\frac{1}{2}\frac{d}{dt}\big (\rho_F\|{\bf u}\|^2_{L^2(\Omega^\eta(t))}+\rho_S h\|\partial_t\bdeta\|^2_{L^2(\Gamma)}+c\|\bdeta\|^2_{H^2(\Gamma)}\big )+\mu\|{\bf D}({\bf u})\|^2_{L^2(\Omega^{\eta}(t))}}\\ \\
\displaystyle{+\frac{1}{\alpha}\|u_{\tau}-\partial_t\eta_{\tau}\|^2_{L^2(\Gamma^\eta(t))}+
\sum_{i\in III} \frac{1}{\alpha_i} \|u_{\tau}\|^2_{L^2(\Gamma_i)}
\leq {\bf C},}
\end{array}
\end{equation}
where ${\bf C}$ depends on the initial and boundary data,
and constant $c$ in front of the $H^2$-norm of $\bdeta$ is associated with the coercivity of the structure operator ${\cal{L}}_e$,
see equation \eqref{coercivity}.

Before we can define weak solutions to problem \eqref{NSP}-\eqref{ProblemIC} we notice that
one of the main difficulties associated with studying problem \eqref{NSP}-\eqref{ProblemIC} 
is the moving fluid domain, which is not known {\sl a priori}.
To deal with this difficulty a couple of approaches have been proposed in the literature.  
One approach is to reformulate the problem in Lagrangian coordinates (see e.g. \cite{CSS1,Kuk}). 
Unfortunately, since problem~\eqref{NSP}-\eqref{ProblemIC} is given on a fixed, control volume, with the ``inlet'' and ``outlet'' boundary data, 
Lagrangian coordinates cannot be used.  
In this manuscript we adopt the second classical approach and use the so-called Arbitrary Lagrangian Eulerian (ALE) mapping (see e.g. \cite{MarSun,donea1983arbitrary,Quarteroni2000}) 
to transform problem~\eqref{NSP}-\eqref{ProblemIC} to the fixed reference domain $\Omega$. 
This will introduce additional nonlinearities in the problem, which will depend on the ALE mapping.
In the next section we construct the appropriate ALE mapping and study its regularity properties, 
which we will need in the proof of the main theorem. 

\subsection{Construction and regularity of the ALE mappping}\label{ALEConstruction}
{\bf Definition of the ALE mapping.}
Motivated by the energy inequality~\eqref{EI_preliminary} we assume that displacement $\bdeta$ satisfies
\begin{equation}\label{eta_regularity}
\bdeta\in L^{\infty}(0,T;H_0^2(\Gamma)^2\cap W^{1,\infty}(0,T;L^2(\Gamma))^2\hookrightarrow C^{0,1-\beta}(0,T;H^{2\beta}_0(\Gamma))^2,
\end{equation}
for $0 < \beta < 1$. 
The inclusion above is a direct consequence of the standard Hilbert interpolation inequalities (see e.g. \cite{LionsMagenes}). 
We denote the corresponding deformation of the elastic boundary by $\bphi^\eta$, i.e.
$$
\bphi^\eta(t,z)={\bf id}+\bdeta(t,z),\quad (t,z)\in [0,T]\times\Gamma.
$$
We consider a family of ALE mappings ${\bf A}_{\eta}$ parameterized by $\eta$,
$$
{\bf A}_\eta (t) : \Omega \to \Omega^\eta(t),
$$
defined
for each $\bdeta$ as a harmonic extension of deformation $\bphi^{\eta}$, 
i.e. ${\bf A}_{\eta}(t)$ is defined as the solution of the following boundary value problem defined on the reference domain $\Omega$:
\begin{equation}\label{ALEDef}
\begin{array}{rcl}
\Delta {\bf A}_{\eta}(t,\cdot)&=&0\;{\rm in}\;{\Omega},\;
\\
{\bf A}_{\eta}(t)_{|\Gamma}&=&\bphi^{\eta}(t,.),\; \\
{\bf A}_{\eta}(t)_{|\Sigma} &=& {\bf id}.
\end{array}
\end{equation}

{\bf Regularity of the ALE mapping.} Since domain $\Omega$ is a polyhedral domain with maximal angle $\pi$, we can apply 
Theorem 5.1.3.1 from Grisvard \cite{Grisvard2}, p. 261, to obtain the following 
regularity of the ALE mapping:
$$
\|{\bf A}_{\eta}(t)\|_{W^{2,3}(\Omega)}\leq C\|\bdeta(t,.)\|_{W^{5/3,3}(\Gamma)}.
$$
We can further estimate the right hand-side by the Sobolev Embedding Theorem 
\begin{equation}\label{eta_reg}
\|\bdeta(t,.)\|_{W^{5/3,3}(\Gamma)}\leq C\|\bdeta\|_{H^s(\Gamma)}, s\geq 11/6,
\end{equation}
and so
\begin{equation}\label{ALEreg1}
\|{\bf A}_{\eta}(t)\|_{W^{2,3}(\Omega)}\leq C\|\bdeta(t,.)\|_{W^{5/3,3}(\Gamma)}\leq \tilde{C} \|\bdeta(t,.)\|_{H^s(\Gamma)},\; s\geq 11/6.
\end{equation}
By using the Sobolev Embedding Theorem to estimate the $W^{2,3}$-norm of ${\bf A}_\eta(t)$ we obtain that
 ${\bf A}_{\eta}$ has a H\"older continuous derivative, namely
\begin{equation}\label{ALEreg2}
\|{\bf A}_{\eta}(t)\|_{C^{1,1/3}(\Omega)}\leq \tilde{C} \|\bdeta(t,.)\|_{H^s(\Gamma)},\; s\geq 11/6.
\end{equation}
Finally, we notice that time $t$ is only a parameter in the linear problem~\eqref{ALEDef},
thus the regularity properties of  ${\bf A}_{\eta}$ with respect to time are the same as 
the regularity properties of $\bdeta$. Now, since $\bdeta \in C^{0,1/12}(0,T;H_0^{11/6}(0,L))^2$,
as shown in \eqref{eta_regularity} for $\beta = 11/6$, and from the inequality \eqref{ALEreg2} with $s = 11/6$,
we have
\begin{eqnarray}
  &1.& {\bf A}_{\eta}\in C^{0,1/12}(0,T;C^{1,1/3}(\Omega))^2,\  {\rm and}
  \label{ALEreg4}
  \\
  &2.&  \displaystyle{\|{\bf A}_{\eta}\|_{C^{0,1/12}(0,T;C^{1,1/3}(\Omega))}\leq C \|\bdeta\|_{C^{0,1/12}(0,T;H^{11/6}(\Gamma))}.}
\label{ALEreg3}
\end{eqnarray}
This regularity result is not optimal, but it is sufficient for the remainder of the proof.

{\bf The ALE velocity.} The ALE velocity is defined by 
\begin{equation}\label{ALEVelocity}
{\bf w}^{\eta}=\frac{d}{dt}{\bf A}_{\eta}.
\end{equation}
From the regularity of $\bdeta$ in \eqref{eta_regularity}
we see that ${\bf w}^{\eta} \in L^2(0,T;H^{1/2}(\Omega))^2$ with
\begin{equation}\label{DomVelReg}
\|{\bf w}^{\eta}\|_{L^{\infty}(0,T;H^{1/2}(\Omega))}\leq C \|\partial_t\bdeta\|_{L^{\infty}(0,T;L^2(\Gamma))}.
\end{equation}
Furthermore, the following estimate holds:
\begin{equation}\label{DomVelRegBis}
\|{\bf w}^{\eta}\|_{H^1(0,T;H^{s+1/2}(\Omega))}\leq C \|\partial_t\bdeta\|_{L^{2}(0,T;H^s(\Gamma))},\; s<1.
\end{equation}
We shall see later in Proposition~\ref{AdditionalReg} that the right hand-side of this inequality is, indeed, bounded.
More precisely, we will show that $\bdeta\in H^1(0,T;H^s(\Gamma))$, $s<1$.

{\bf The Jacobian of the ALE mapping.} Let us now consider the Jacobian of the ALE mapping 
\begin{equation}
\label{Jacobian}
J^{\eta}(t,.)=\det \nabla {\bf A}_{\eta}(t).
\end{equation}
From the regularity property  \eqref{ALEreg4} of ${\bf A}_{\eta}$ we have
$$
J^{\eta}\in C^{0,1/12}(0,T;C^{0,1/3}(\Omega)).
$$
Now, from the compatibility condition~\eqref{StructureCC} we have that the Jacobian at $t = 0$ satisfies
$J(0,{\bf x})\geq C>0$, ${\bf x}\in\Omega$. 
By the continuity of the Jacobian as a function of time,
this implies the existence of a time interval $(0,T')$ such that $J^{\eta}$ is strictly positive on $(0,T')$, i.e. we have:
\begin{equation}\label{PositivityJ}
J^{\eta}\geq C>0,\quad{\rm on}\; (0,T')\times\Omega.
\end{equation}

{\bf Injectivity of $A_\eta$.} In order to be able to use the above-constructed ALE mapping to transform the problem onto the fixed reference domain, it remains to show that $A_{\eta}$ is an injection. A sufficient condition for the injectivity of $\bdeta$ is given by the following proposition.

\begin{proposition}\label{ALEinjectivity}
Let $\bdeta\in L^{\infty}(0,T;H^2(\Gamma))\cap W^{1,\infty}(0,L;L^2(\Gamma))$ be such that $\bdeta(0,.)=\bdeta_0$, 
and $\bdeta_0$ satisfies conditions~\eqref{StructureCC} and \eqref{eta0}. Then there exists a $T''>0$ 
such that for every $t\in [0,T'']$ the ALE mapping ${\bf A}_{\eta}(t)$ is an injection.
\end{proposition}
\begin{proof}
First we notice that because of the linearity of problem~\eqref{ALEDef} 
and the definition of deformation $\bphi^{\eta}$, we can write the ALE mapping in the following form:
$$
{\bf A}_{\eta}(t)={\bf id}+{\bf B}_{\eta}(t),
$$
where ${\bf B}_{\eta}$ is the solution of the following boundary value problem:
$$
\begin{array}{rcl}
\Delta {\bf B}_{\eta}(t,)&=& 0\;{\rm in}\;{\Omega},\;
\\
{\bf B}_{\eta}(t)_{|\Gamma} &=& \bdeta^{\eta}(t,.),\; \\
{\bf B}_{\eta}(t)_{|\partial\Omega\setminus\Gamma}&=&0.
\end{array}
$$
Now, we see that the regularity of ${\bf B}_{\eta}$ follows in the same way as the regularity of ${\bf A}_{\eta}$ with analogous estimates in terms of $\bdeta$ in the same norms. Therefore, ${\bf B}_{\eta}$ satisfies \eqref{ALEreg3}, 
which implies, among other things, that 
$$
\sup_{\bar\Omega} | \nabla {\bf B}_{\eta} | \le c(\Omega).
$$
We now use Theorem 5.5-1 from \cite{CiarletBook1} (pp. 222), which we state here for completeness:
\begin{theorem}{\bf (Sufficient conditions for preservation of injectivity and orientation \cite{CiarletBook1})}\label{Ciarlet}\
\indent
(A) Let ${\bphi} = {\bf id} + \bpsi: \Omega \subset \R^n \to \R^n$ be a mapping differentiable at a point ${\bf x} \in \Omega$.
Then:
$$
| \nabla \bpsi | < 1 \rightarrow {\rm det} \nabla \bphi > 0.
$$
\indent
(B) Let $\Omega$ be a domain in $\R^n$. There exists a constant $c(\Omega) > 0$ such that any mapping 
${\bphi} = {\bf id} + \bpsi \in C^1(\bar\Omega,R^n)$ satisfying 
\begin{equation}\label{c_Omega}
\sup_{\bar\Omega} | \nabla \bpsi | \le c(\Omega)
\end{equation}
is injective.
\end{theorem}
Statement (B) of the theorem says that every domain has as associated constant $c(\Omega)$ such that whenever 
\eqref{c_Omega} holds, the mapping ${\bphi}$ is injective.

We use this theorem, together with the assumptions \eqref{StructureCC} and \eqref{eta0}, to see that there exists a $T'' > 0$ such that the ALE mapping
${\bf A}_\eta(t)$ is an injection for every $t \in [0,T'']$.

\end{proof}

We now take the minimum between $T'$ and $T''$, and call this new time $T$ again, i.e., 
\begin{equation}\label{T}
T=\min\{T',T''\},
\end{equation}
where $T'$ is determined from the positivity of the Jacobian $J^\eta$, see \eqref{PositivityJ}, and $T''$ is determined from 
the injectivity of ${\bf A}_\eta$, see Proposition~\ref{ALEinjectivity}.
This new time determines the existence time interval for the weak solution. 
Our existence result will be local in time in the sense that the maximum $T$ for which we can show that a solution exists 
is determined by the time at which the fluid domain degenerates in the sense that either  the Jacobian of the ALE mapping
becomes zero, or the ALE mapping ceases to be injective. Examples showing two types of domain degeneration are shown in 
Figure~\ref{degeneration}. The degeneration of the fluid domain
due to the loss of injectivity of ${\bf A}_\eta$ shown in Figure~\ref{degeneration} left can occur because the longitudinal 
displacement is non-zero. The degeneration of the fluid domain shown in Figure~\ref{degeneration} right, associated with 
the loss of injectivity of ${\bf A}_\eta$ and loss of strict positivity of the Jacobian $J^\eta$, can occur even 
if one assumes that the longitudinal displacement of the structure is zero.
\begin{figure}[ht]
\centering{
\includegraphics[scale=0.35]{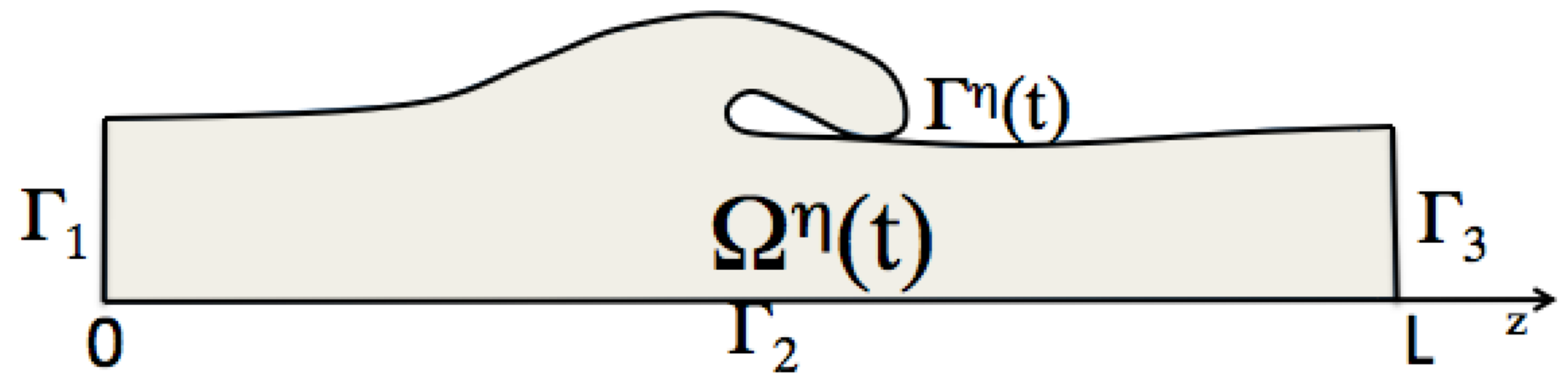}
\includegraphics[scale=0.35]{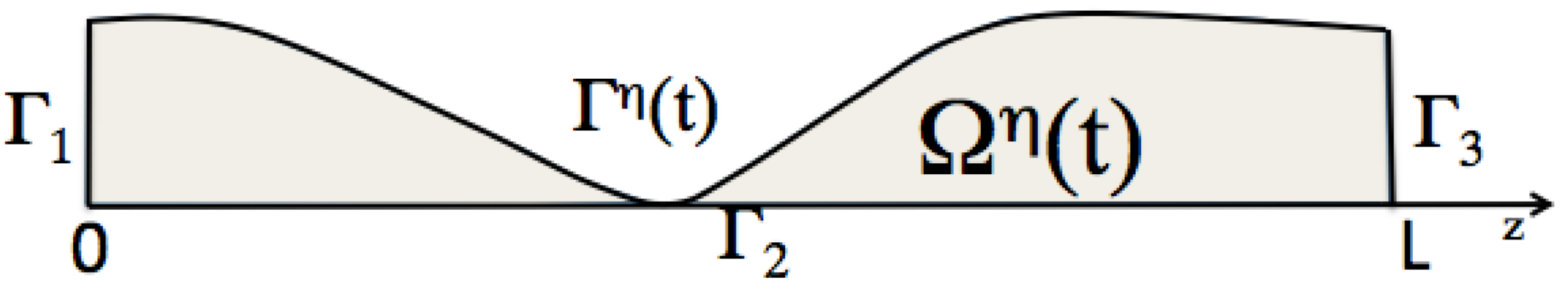}
}
\caption{Two ways a fluid domain can degenerate. Left: loss of injectivity of the ALE mapping ${\bf A}_\eta$. Right: loss of injectivity of
the ALE mapping ${\bf A}_\eta$  and loss strict positivity of the Jacobian $J^\eta$.}
\label{degeneration}
\end{figure}

\subsection{The weak ALE formulation}\label{Sec:Defweak}
As mentioned earlier, we will prove the existence of a weak solution to problem \eqref{NSP}-\eqref{ProblemIC} by mapping
the problem defined on the moving domain $\Omega^\eta(t)$ onto a fixed, reference domain $\Omega$, and study the transformed problem
on $\Omega$. 
For this purpose we map the functions defined on the moving domain $\Omega^\eta(t)$ onto the reference 
domain using the ALE mapping introduced above. We will use super-script $\eta$ to denote that
those functions now depend (implicitly) on $\bdeta$. More precisely,
let ${\bf f}$ be a (scalar or vector) function defined on $\Omega^\eta(t)$. Then  ${\bf f}^{\eta}$, defined on $\Omega$,
is given by:
$$
{\bf f}^{\eta}(t,z,r)={\bf f}\big (t,{\bf A}_{\eta}(t)(z,r)\big )={\bf f}(t,x,y),
$$
where $(x,y)={\bf A}_{\eta}(t)(z,r)\in\Omega^{\eta}(t)$ denotes the coordinates in $\Omega^\eta(t)$.
Furthermore, we define the transformed gradient and divergence operators by
\begin{equation}\label{nablaeta}
\nabla^{\eta}{\bf f}^{\eta}(t):=(\nabla{\bf f}(t))^{\eta}=\nabla{\bf f}^{\eta}(\nabla A_{\eta}(t))^{-1},\;\nabla^{\eta}\cdot{\bf f}^{\eta}={\rm tr}\big (\nabla^{\eta}{\bf f}^{\eta}\big ).
\end{equation}
It will be useful in the remainder of the paper to obtain a relationship between the gradient $\nabla^{\eta}$ and symmetrized gradient
 ${\bf D}^{\eta(t,.)}$ of the fluid velocity ${\bf u}$
defined on $\Omega$. This is typically given via Korn's inequality. 
However, the problem is that the fluid velocity ${\bf u}^\eta(t,z,r)$ for $t \in (0,T)$, which is defined on the fixed domain $\Omega$,
is coming from a family of velocities  ${\bf u}$ defined on domains $\Omega^\eta(t)$ for $t \in (0,T)$, via 
a family of  ALE mappings, all depending on $\bdeta$. We now show that if $\bdeta$ is ``nice enough'', there exists a uniform 
Korn's constant, independent of the family of domains, such that a version of Korn's inequality holds. 
More precisely, the following Lemma holds true. 
\begin{lemma}{\bf (The ``transformed'' Korn's inequality)}\label{Korn}
Let 
\begin{enumerate}
\item $\bdeta\in L^{\infty}(0,T;H^2(\Gamma))\cap W^{1,\infty}(0,L;L^2(\Gamma))$, and
\item $\bdeta(0,.)=\bdeta_0$ where $\bdeta_0$ satisfies conditions \eqref{StructureCC} and \eqref{eta0}.
\end{enumerate}
Then there exists a time $T'>0$ and constants $C, c>0$ depending only on 
$\|\bdeta\|_{L^{\infty}(H^2)\cap W^{1,\infty}(L^2)}$,
 such that for every ${\bf u}^\eta\in H^1(\Omega)^2$ satisfying boundary condition~\eqref{ProblemBC} on $\Sigma$, the following transformed version of Korn's inequality holds:
$$
c\|{\bf D}^{\eta (t,.)}({\bf u})\|_{L^2(\Omega)}\leq\|\nabla^{\eta(t,.)}({\bf u})\|_{L^2(\Omega)}\leq C\|{\bf D}^{\eta(t,.)}({\bf u})\|_{L^2(\Omega)},\;t\in [0,T'].
$$ 
The time $T' > 0$ is determined by the injectivity of the ALE mapping ${\bf A}_\eta$.
\end{lemma}
\proof
From Proposition~\ref{ALEinjectivity} we deduce the existence of a $T'>0$ such that ${\bf A}_{\eta}(t)$ is injective for every $t\in [0,T']$. Now, 
the statement of the Lemma follows from the results in \cite{Velcic} (Lemma 1 and Remark 6) in the same way as in \cite{BorSun}. 
Namely, for each fixed $t \in (0,T')$ we map ${\bf u}$ back to the physical domain $\Omega^\eta(t)$ 
and apply Korn's inequality there in a standard way, using the classical Korn's constant which depends on domain $\Omega^\eta(t)$
defined by $\bdeta$. Due to the regularity of $\bdeta$ given by conditions 1. and 2. in the statement of the Lemma, and 
due to the uniform (in $t$) estimate~\eqref{ALEreg3}, it follows that the set $\{{\bf A}_{\eta}(t):t\in (0,T')\}$ is compact in $W^{1,\infty}(\Omega)^2$, from which the existence of universal Korn constants $c$ and $C$ follows.
\qed

To define the ALE formulation of problem~\eqref{NSP}-\eqref{ProblemIC} we recall the definition of the ALE velocity given in \eqref{ALEVelocity}
and define the ALE derivative as a time derivative evaluated on the fixed reference domain:
\begin{equation}\label{ALEDerivative}
\partial_t{\bf f}_{|\Omega}=\partial_t{\bf f}+({\bf w}^\eta\cdot\nabla){\bf f}.
\end{equation}
Using the ALE mapping we can rewrite the Navier-Stokes equations in the ALE formulation as follows:
\begin{equation}\label{ALENS}
\partial_t{\bf u}_{|\Omega}+({\bf u}-{\bf w}^\eta)\cdot\nabla{\bf u}=\nabla\cdot\bsigma\quad{\rm in}\;\Omega^{\eta}(t).
\end{equation}
Here, the terms $\partial_t{\bf u}_{|\Omega}$ and ${\bf w}^\eta$,
which are originally defined on $\Omega$, 
are  composed with the inverse of the ALE mapping, which maps them back to the moving domain $\Omega^\eta(t)$.

Our goal is to define weak solutions to problem \eqref{NSP}-\eqref{ProblemIC} on the fixed, reference domain $\Omega$. 
The first step is to introduce the necessary function spaces on $\Omega$. 
For this purpose we notice that the incompressibility condition in moving domains $\nabla\cdot{\bf u}=0$ 
transforms into the following condition on $\Omega$: 
$$\nabla^{\eta}\cdot{\bf u}^\eta=0,$$ 
which we will call the transformed divergence-free condition. 

Motivated by the energy inequality \eqref{EI_preliminary} we can now define the function spaces associated with 
weak solutions of problem \eqref{NSP}-\eqref{ProblemIC}.
We define the analogue of the classical function space for the fluid velocity with the transformed divergence-free condition:
$$
{\mathcal V}_F^{\eta}=\{{\bf u}^\eta=(u_z^\eta,u_r^\eta)\in H^1(\Omega)^2:\nabla^{\eta}\cdot{\bf u}^\eta=0,
{\bf u}^\eta\cdot\btau^\eta=0,\; {\rm on}\ \Gamma_i,\;i\in I,\;
$$
$$
{\bf u}^\eta=0\ {\rm on}\ \Gamma_i,\;i\in II,\; {\bf u}^\eta\cdot\bnu^\eta=0,\; {\rm on}\ \Gamma_i,\;i\in III\cup IV \}.
$$
The corresponding space involving time is given by:
\begin{equation}\label{FluidSpace}
{\mathcal W}_F^\eta(0,T)=L^{\infty}(0,T;L^2(\Omega))\cap L^2(0,T;{\cal V}_F^\eta).
\end{equation}
The structure function spaces are classical:
\begin{equation}\label{struc_test}
{\mathcal W}_S(0,T)=W^{1,\infty}(0,T;L^2(0,L))^2\cap L^{\infty}(0,T;H^2_0(\Gamma))^2.
\end{equation}
The solution space for our problem with the {\bf slip boundary condition} must incorporate the continuity of normal velocities:
\begin{equation}\label{W_eta}
{\mathcal W^\eta}(0,T)=\{({\bf u}^\eta,\bdeta)\in {\cal W}_F^\eta(0,T)\times{\cal W}_S(0,T):{\bf u}^\eta_{|\Gamma}\cdot\bnu^{\eta}=\partial_t\bdeta\cdot\bnu^{\eta}\}.
\end{equation}
The corresponding test space is defined by
\begin{equation}\label{Q_eta}
{\mathcal Q^\eta}(0,T)=\{({\bf q}^\eta,\bpsi)\in C^1_c([0,T);{\cal V}_F^\eta\times H_0^2(\Gamma)^2):{\bf q}^\eta_{|\Gamma}\cdot\bnu^{\eta}=\bpsi(t,z)\cdot\bnu^{\eta}\}.
\end{equation}

To obtain the weak formulation on $\Omega$ we first consider our problem defined on  moving domains $\Omega^\eta(t)$.
Take a test function ${\bf q}$ defined on $\Omega^\eta(t)$ for some $\bdeta$, such that the corresponding test function ${\bf q}^\eta$
defined on the fixed domain belongs to the test space ${\mathcal Q^\eta}(0,T)$,
namely,  $({\bf q}^{\eta},\bpsi)\in {\mathcal Q^\eta}(0,T)$. 
Multiply~\eqref{ALENS} by ${\bf q}$, integrate over $\Omega^{\eta}(t)$, and formally integrate by parts. 
We obtain the following. 
For the convective term we have:
\begin{eqnarray*}
\int_{\Omega^{\eta}(t)}(({\bf u}-{\bf w}^\eta)\cdot\nabla){\bf u}\cdot{\bf q}=
\displaystyle{\frac 1 2\int_{\Omega^{\eta}(t)}(({\bf u}-{\bf w}^\eta)\cdot\nabla){\bf u}\cdot{\bf q}-
\frac 1 2\int_{\Omega^{\eta}(t)}(({\bf u}-{\bf w}^\eta)\cdot\nabla){\bf q}\cdot{\bf u}}\\
\displaystyle{+\frac 1 2\int_{\Gamma^{\eta}(t)}({\bf u}-{\bf w}^\eta)\cdot\bnu^{\eta}({\bf u}\cdot{\bf q})d\Gamma^{\eta}+\frac{1}{2}\int_{\Omega^{\eta}(t)}(\nabla\cdot{\bf w}^\eta){\bf u}\cdot{\bf q}-\frac 1 2\sum_{i\in I}\int_{\Gamma_i}|{\bf u}|^2q_\nu}.
\end{eqnarray*}
For  the diffusive part of the Navier-Stokes equations we have:
$$
-\int_{\Omega^{\eta}(t)}(\nabla\cdot\bsigma)\cdot{\bf q}=2\mu\int_{\Omega^{\eta}(t)}{\bf D}({\bf u}):{\bf D}({\bf q})-\int_{\partial\Omega^{\eta}(t)}\bsigma\bnu^{\eta}\cdot{\bf q},
$$
where the second term on the right hand side can be expressed as follows:
\begin{align}
\int_{\partial\Omega^{\eta}(t)}\bsigma\bnu^{\eta}\cdot{\bf q} & = \int_{\Gamma^{\eta}(t)}\Big ((\bsigma\bnu^{\eta}\cdot\bnu^{\eta}){\bf q}\cdot\bnu^{\eta}+(\bsigma\bnu^{\eta}\cdot\btau^{\eta}){\bf q}\cdot\btau^{\eta}\Big )d\Gamma^{\eta}
\nonumber
\\
&+ \sum_{i\in I}\int_{\Gamma_i}pq_\nu
+\sum_{i\in III}\int_{\Gamma_i}(\bsigma\bnu\cdot\btau){\bf q}\cdot\btau
\nonumber
\\
&=\int_{\Gamma^{\eta}(t)}\Big ((\bsigma\bnu^{\eta}\cdot\bnu^{\eta})\bpsi\cdot\bnu^{\eta}+\frac{1}{\alpha}(\partial_t\bdeta-{\bf u})\cdot\btau^{\eta}({\bf q}\cdot\btau^{\eta})\Big )d\Gamma^{\eta}
\nonumber
\\
&+\sum_{i\in I}\int_{\Gamma_i}pq_\nu
-\sum_{i\in III}\int_{\Gamma_i}\frac{1}{\alpha_i}u_\tau q_\tau.
\nonumber
\end{align}
We sum all the integrals, and take into account the coupling conditions \eqref{ElasticP}-\eqref{SlipP} holding along $\Gamma^\eta(t)$,
and boundary conditions \eqref{ProblemBC} holding along $\Sigma$. 
To deal with the non-zero data on $\Gamma_i, i\in I$
we  introduce the ``source term functional'' ${\bf R}$ which will collect the two terms corresponding to dynamic pressure data
defined on $\Gamma_i, i\in I$:
\begin{equation}\label{SourceTermR}
\langle {\bf R}, {\bf q}\rangle := \sum_{i\in I} \int_{\Gamma_i} P_i {\bf q}\cdot {\bnu},\quad 
\forall {\bf q} \in C_c^1([0,T)\times\overline\Omega),
\end{equation}
where $\bnu$ is the unit outward normal to the boundary $\Gamma_i, i \in I$.
Then we integrate the entire expression with respect to time over $(0,T)$ to obtain the following weak ALE
formulation of  problem \eqref{NSP}-\eqref{ProblemIC} defined on $\Omega^\eta(t)$:

\begin{equation}\label{WeakFormulationMovingDomain}
\begin{array}{c}
\displaystyle{\rho_F\int_0^T\int_{\Omega^{\eta}(t)}\left(\partial_t{\bf u}^{\eta}_{|\Omega}\cdot{\bf q}+\frac{1}{2}\big ((({\bf u}-{\bf w}^\eta)\cdot\nabla){\bf u}\cdot{\bf q}-
(({\bf u}-{\bf w}^\eta)\cdot\nabla){\bf q}\cdot{\bf u}\big )+\frac{1}{2}(\nabla\cdot{\bf w}^\eta){\bf u}\cdot{\bf q}\right)}
\\ 
\displaystyle{ + \int_0^T \left\{2\mu\int_{\Omega^{\eta}(t)}{\bf D}({\bf u}):{\bf D}({\bf q})+\sum_{i\in III}\int_{\Gamma_i}\frac{1}{\alpha_i}u_\tau q_\tau
+\frac{1}{\alpha}\int_{\Gamma^{\eta}(t)}(u_{\tau^{\eta}}-\partial_t\eta_{\tau^{\eta}})q_{\tau^{\eta}}d\Gamma^{\eta}(t)\right\}dt} \\
\displaystyle{ +\rho_S h\int_0^T\int_{\Gamma} \partial^2_t\bdeta\bpsi dz dt} 
\displaystyle{+\int_0^T \langle{\mathcal L}_e\bdeta,\bpsi\rangle dt +\frac{1}{\alpha}\int_0^T\int_{\Gamma}(\partial_t\eta_{\tau^{\eta}}-u_{\tau^{\eta}})\phi_{\tau^{\eta}}S^{\eta}dz dt = \int_0^T\langle {\bf R},{\bf q}\rangle},\; 
\end{array}
\end{equation}
for all $({\bf q},\bpsi)$ such that $({\bf q}^{\eta},\bpsi)\in {\mathcal Q}^{\eta}(0,T)$,
where $g_{\tau^{\eta}}={\bf q}\cdot\btau^{\eta}$ and $g_{\nu^{\eta}}={\bf q}\cdot\bnu^{\eta}$ denotes the tangential and normal component of ${\bf q}$, respectively, and the source term functional ${\bf R}$ is defined in \eqref{SourceTermR} to account for the non-zero dynamic pressure boundary 
condition on $\Gamma_i, i\in I$.

We now transform~\eqref{WeakFormulationMovingDomain} to the fixed reference domain via 
the ALE mapping ${\bf A}_{\eta}$. To do that we first compute the integral involving the ALE derivative $\partial_t{\bf u}^\eta_{|\Omega}$:
\begin{align}
\int_0^T\int_{\Omega^{\eta}(t)}\partial_t{\bf u}^{\eta}_{|\Omega}\cdot{\bf q}& =\int_0^T\int_{\Omega}J^{\eta}\partial_t{\bf u}^{\eta}\cdot{\bf q}^{\eta}
\nonumber 
\\
=& -\int_0^T\int_{\Omega} \partial_t  J^{\eta}{\bf u}^{\eta}\cdot{\bf q}^{\eta}-\int_0^T\int_{\Omega}J^{\eta}{\bf u}^{\eta}\cdot\partial_t{\bf q}^{\eta}
-\int_{\Omega}J_0{\bf u}_0\cdot{\bf q}^{\eta}(0,.).
\nonumber
\end{align}
Now, since 
\begin{equation}\label{J_t}
\partial_t J^{\eta}=J^{\eta}\nabla^{\eta}\cdot{\bf w}^{\eta}
\end{equation}
(see e.g. \cite{Gur}, p. 77),
the above expression reads
\begin{align}
\int_0^T\int_{\Omega^{\eta}(t)}\partial_t{\bf u}^{\eta}_{|\Omega}\cdot{\bf q}= 
-\int_0^T \int_{\Omega}  \left\{ J^{\eta} \left( \nabla^{\eta}\cdot{\bf w}^{\eta}\right) \left( {\bf u}^{\eta}\cdot{\bf q}^{\eta}\right) 
-  J^{\eta}{\bf u}^{\eta}\cdot\partial_t{\bf q}^{\eta}\right\}
-\int_{\Omega}J_0{\bf u}_0\cdot{\bf q}^{\eta}(0,.).
\nonumber
\end{align}
Now, we can define the weak solution on the fixed reference domain.

\vskip 0.1in
{\bf NOTATION.} To simplify {notation}, from this point on we will omit the superscript $\eta$ in ${\bf u}^{\eta}$ and ${\bf q}^{\eta}$
since everything will be happening only on the fixed, reference domain, and there will be no place for confusion.

\begin{definition}{\bf (Weak solution) } \label{WeakALE}
We say that $({\bf u},\bdeta)\in{\mathcal W}^{\eta}(0,T)$ is a weak solution to problem \eqref{NSP}-\eqref{ProblemIC} defined on the reference domain $\Omega$, if for every $({\bf q},\bpsi)\in{\mathcal Q}^{\eta}(0,T)$ the following equality holds:
\end{definition}
\begin{equation}\label{WeakALEFormula}
\begin{array}{c}
\displaystyle{\frac{\rho_F}{2}\int_0^T\int_{\Omega}J^{\eta}\Big ((({\bf u}-{\bf w}^{\eta})\cdot\nabla^{\eta}){\bf u}\cdot {\bf q}-(({\bf u}-{\bf w}^{\eta})\cdot\nabla^{\eta}){\bf q}\cdot {\bf u}}
-(\nabla^{\eta}\cdot{\bf w}^{\eta}){\bf q}\cdot {\bf u}\Big)
\\ \\
\displaystyle{
-\rho_F\int_0^T\int_{\Omega}J^{\eta}{\bf u}\cdot\partial_t{\bf q} +\int_0^T2\mu\int_{\Omega} J^\eta {\bf D}^{\eta}({\bf u}):{\bf D}^{\eta}({\bf q})+\sum_{i\in III}\int_0^T \int_{\Gamma_i}\frac{1}{\alpha_i}u_\tau q_\tau}
\\ \\
\displaystyle{+\frac{1}{\alpha}\int_0^T \int_{\Gamma}(u_{\tau^{\eta}}-\partial_t\eta_{\tau^{\eta}})q_{\tau^{\eta}}S^{\eta}dzdt -\rho_S h\int_0^T\int_{\Gamma} \partial_t\bdeta\partial_t\bpsi dz dt+\int_0^T \langle{\mathcal L}_e\bdeta,\bpsi\rangle}
\\ \\
\displaystyle{+\frac{1}{\alpha}\int_0^T \int_{\Gamma}(\partial_t\eta_{\tau^{\eta}} -u_{\tau^{\eta}})\psi_{\tau^{\eta}}S^{\eta}dz dt
=\int_0^T\langle {\bf R},{\bf q}\rangle+\int_{\Omega}J_0{\bf u}_0\cdot{\bf q}(0)+\int_{\Gamma}{\bf v}_0\cdot\bpsi},
\end{array}
\end{equation}
where $J^\eta$ is the determinant of the Jacobian of  ALE mapping, defined in \eqref{Jacobian},
and ${\bf w}^\eta$ is the ALE velocity, defined in \eqref{ALEVelocity}.

We note that Definition~\ref{WeakALE} makes sense even in the case when the ALE mapping ${\bf A}_{\eta}$ is not injective. 
However, in that case the weak formulation~\eqref{WeakALEFormula} is not equivalent to problem \eqref{NSP}-\eqref{ProblemIC} even for smooth solutions.

\section{Main Result}
We are now in a position to state the main result of this work. For this purpose we shall assume that all the moving domains 
$\Omega^\eta(t)$ are contained in a larger domain $\Omega_{\rm max}$. Indeed, 
it will be shown later, see Corollary \ref{ConvEtaC1}, that this will be the case. 
We note that the only reason for this assumption is to 
assume a certain regularity of the source term ${\bf R}$, associated with the non-zero boundary data 
on $\Sigma$. Namely, we will be assuming that ${\bf R}\in L^2(0,T;H^1(\Omega_{\rm max})')$,
where $(H^1)'$ denotes the dual space of $H^1$. 
In the case when only the dynamic pressure boundary data is different from zero, the introduction of the source term ${\bf R}$ 
is not necessary. However, we state the Main Result in  general terms,
and consider $\Omega_{\rm max}$ to be the union of all the moving domains $\Omega^\eta(t)$.

\begin{theorem}{\bf (Main result)} \label{MainResult}
Let all the parameters in the problem be positive (this includes the
fluid and structure densities $\rho_F$ and $\rho_S$, structure thickness $h$, and the slip-condition friction constants $\alpha$ on 
the moving boundary $\Gamma^\eta(t)$,
and $\alpha_i$'s on the rigid boundary $\Gamma_{III}$). 
Moreover, let the source term functional ${\bf R}$, defined in \eqref{SourceTermR}, 
be such that ${\bf R}\in L^2(0,\infty;H^1(\Omega_{\rm max})')$.
 If the initial data 
 ${\bf u}_0\in L^2(\Omega^0)$ and $\bdeta_0\in H^2_0(\Gamma)$ 
 are such that compatibility conditions~\eqref{FluidCC}, \eqref{StructureCC} and \eqref{eta0} are satisfied, then there exists a $T>0$ and a weak solution $({\bf u},\bdeta)$ to problem \eqref{NSP}-\eqref{ProblemIC}
 defined on $(0,T)$,
 such that the following energy estimate is satisfied:
 \begin{equation}\label{EI1}
\begin{array}{c}
\displaystyle{\frac{1}{2}\big (\rho_F\|{\bf u}\|^2_{L^\infty(0,T;L^2(\Omega^\eta(t))}+\rho_S h\|\partial_t\bdeta\|^2_{L^\infty(0,T;L^2(\Gamma))}
+c\|\bdeta\|^2_{L^\infty(0,T;H^2(\Gamma))}\big )+\mu\|{\bf D}({\bf u})\|^2_{L^2(0,T; L^2(\Omega^{\eta}(t)))}}\\ \\
\displaystyle{+\frac{1}{\alpha}\|u_{\tau}-\partial_t\eta_{\tau}\|^2_{L^2(0,T;L^2(\Gamma^\eta(t)))}
+\sum_{i\in III} \frac{1}{\alpha_i} \|u_{\tau}\|^2_{L^2(0,T;L^2(\Gamma_i))}
\leq E_0 +  C \|{\bf R}\|^2_{L^2(0,T;H^{1}(\Omega^\eta(t))')},}
\end{array}
\end{equation}
where $C$ depends only on the initial data and on the parameters in the problem, and $E_0$ is the kinetic and elastic energy of the initial data.
\label{main}
\end{theorem}

The proof of this theorem is based on the following approach. We design a partitioned time-marching scheme
by using the time-discretization via Lie operator spitting. 
We first separate  the fluid and structure subproblems
and then semi-discretize the resulting sub-problems with respect to time.
The time interval $(0,T)$ is subdivided into $N$ sub-intervals $(0,t_1),\cdots,(t_{N-1},t_N)$ each of width $\Delta t$,
and the fluid and structure sub-problems are solved on each sub-interval. 
First, the structure sub-problem is solved on $(t_{i-1},t_i)$ using for the initial data the solution of the fluid sub-problem from 
the previous time step, and then the fluid sub-problem is solved on $(t_{i-1},t_i)$  using
for the initial data the solution of the just calculated structure sub-problem. 
The process is repeated for each 
sub-interval of $(0,T)$. 
This defines an approximation of the solution of the coupled FSI problem on $(0,T)$.
On each sub-interval $(t_{i-1}, t_i)$
the transfer of information between the two sub-problems is achieved via the ``initial data'' at $t_{i-1}$. 
At each time step only one iteration for the fluid sub-problem and one for the structure sub-problem are sufficient
to obtain a stable and convergent algorithm. 
The goal is to show that as the time-discretization step $\Delta t \to 0$, the sequence of approximate solutions, described above, 
converges to a weak solution of the coupled FSI problem. 

A crucial step in this approach is the way how the fluid and structure sub-problems are designed.
In particular, we had to be careful to take into account the well-known problems related to the so called ``added mass effect'' \cite{causin2005added}
to design the fluid sub-problem in such a way that the fluid and structure inertia are kept close together 
via a Robin-type boundary condition for the fluid sub-problem. In contrast with the no-slip boundary condition,
the slip boundary condition is more tricky to deal with because of the lack of continuity in the tangential component of the velocity at the
fluid-structure interface, and so the smoothing of the interface due to the viscous fluid dissipation is no longer transferred 
to the structure in the tangential direction. New  compactness arguments based on the theorem of Simon \cite{Simon}
and on interpolation of classical Sobolev spaces with real exponents $H^s$ (or alternatively Nikolskii spaces) will be used to obtain the existence result.
Details are presented next.

\section{Approximate solutions}\label{Sec:Approx}
We construct approximate solutions to problem \eqref{NSP}-\eqref{ProblemIC} by using the time-discretization via Lie operator splitting.

\subsection{Operator splitting scheme}\label{Sec:Lie}
Let ${\Delta t}=T/N$ be the time-discretization parameter so that the time interval $(0,T)$ is
sub-divided into $N$ sub-intervals of width $\Delta t$. On each sub-interval we split the problem into a fluid and structure
sub-problem, and linearize each sub-problem appropriately. 
Each of the sub-problems will be discretized in time using the Backward Euler scheme. 

To perform the Lie splitting we must rewrite problem \eqref{NSP}-\eqref{ProblemIC} as a first-order system in time 
\begin{align}
\frac{d {\bf X} }{dt} &= A {\bf X}, \ t\in (0,T), \nonumber \\
{\bf X}|_{t=0} &= {\bf X}^0,
\end{align}
where $A$ is an operator on a Hilbert space, such that $A$ can be split into a non-trivial decomposition $A = A_1 + A_2$.
For this purpose introduce the substitution ${\bf v}=\partial_t\bdeta$, and rewrite the structure acceleration in terms of the first-order
derivative of structure velocity. 
The initial approximation of the solution will be the initial data in the problem, namely, 
${\bf u}^0={\bf u}_0$, $\bdeta^0=\bdeta_0$, and ${\bf v}^0={\bf v}_0$.
For every sub-division of $(0,T)$ containing $N\in\N$ sub-intervals, we recursively define the vector of unknown approximate solutions 
\begin{equation}
{\bf X}_{\Delta t}^{n+\frac i 2}=\left (\begin{array}{c} {\bf u}_{\Delta t}^{n+\frac i 2} \\ {\bf v}_{\Delta t}^{n+\frac i 2} \\ \bdeta_{\Delta t}^{n+\frac i 2} \end{array}\right ), n=0,1,\dots,N-1,\,\ i=1,2,
\label{X}
\end{equation}
where $i = 1,2$ denotes the solution of sub-problems defined by $A_1$ or $A_2$, respectively.
The initial condition is given by the initial data in the problem. 
${\bf X}^0= ({\bf u}_0, {\bf v}_0, \bdeta_0)^T.$

A crucial ingredient for the existence proof is that the semi-discretization of the split problem be performed in 
such a way that a semi-discrete version of the energy inequality \eqref{EI1} is preserved 
at every time step. This is associated with successfully dealing with the ``added mass effect'' \cite{causin2005added,BorSun}.
For this purpose we define a semi-discrete version of the total energy and dissipation at time $n\Delta t$ as follows:
\begin{align}\label{DiscreteEnergy}
&E^{n+i/2}_{\Delta t}=\frac{\rho_F}{2}\int_{\Omega}J^n|{\bf u}^{n}_{\Delta t}|^2+\frac{\rho_S h}{2}\big (\|{\bf v}^{n+i/2}_{\Delta t}\|^2_{L^2(\Gamma)}+\langle{\mathcal L}_e\bdeta^{n+i/2}_{\Delta t},\bdeta^{n+i/2}_{\Delta t}\rangle \big ),\; i=0,1;\\
&D^n_{\Delta t}={\Delta t}
\mu\int_{\Omega}  J^n {\bf D}({\bf u}^n_{\Delta t}):{\bf D}({\bf u}^n_{\Delta t})
+\frac{\Delta t}{2}\sum_{i\in III}\int_{\Gamma_i}\frac{1}{\alpha_i}   ({u^n_{\Delta t}})_\tau ({u^n_{\Delta t}})_\tau
+\frac{\Delta t}{\alpha}\|(v^n_{\Delta t})_{\tau}-(u^n_{\Delta t})_{\tau}\|^2_{L^2(\Gamma)}.\;
\label{Discrete dissipation}
\end{align}
Throughout the rest of this section, we keep the time step $\Delta t$ fixed,
and define the semi-discretized fluid and structure sub-problems.
To simplify notation, we will omit the subscript $\Delta t$ and write
$({\bf u}^{n+\frac i 2},{\bf v}^{n+\frac i 2},\bdeta^{n+\frac i 2})$ instead of 
$({\bf u}^{n+\frac i 2}_{\Delta t},{\bf v}^{n+\frac i 2}_{\Delta t},\bdeta^{n+\frac i 2}_{\Delta t})$.
\vskip 0.2in
\noindent
{\bf THE STRUCTURE SUB-PROBLEM} (Differential formulation):
\begin{equation}\label{StrcutureSub1}
\left .
\begin{array}{c}
\displaystyle{\rho_S h \frac{{\bf v}^{n+1/2}-{\bf v}^n}{\Delta t}+{\mathcal L}_e\bdeta^{n+1}=0,}\\ \\
\displaystyle{\frac{\bdeta^{n+1}-\bdeta^n}{\Delta t}={\bf v}^{n+1/2},}
\end{array}\right \}\;{\rm on}\; \Gamma,
\end{equation}
$$
\bdeta^{n+1}(0)=\partial_z\bdeta^{n+1}(0)=\bdeta^{n+1}(L)=\partial_z\bdeta^{n+1}(L)=0.
$$
The  weak formulation is given by: find $({\bf v}^{n+1/2},\bdeta^{n+1})\in H^2_0(\Gamma)^4$ such that
\begin{equation}\label{StructureWeak}
\begin{array}{c}
\displaystyle{\rho_S h \int_{\Gamma}\frac{{\bf v}^{n+1/2}-{\bf v}^n}{\Delta t}\cdot\bpsi+\langle{\mathcal L}_e\bdeta^{n+1},\bpsi\rangle=0,\quad \bpsi\in [H^2_0(\Gamma)]^2,} \\
\displaystyle{\int_{\Gamma} \frac{\bdeta^{n+1}-\bdeta^n}{\Delta t} \cdot  {\bphi}= \int_\Gamma {\bf v}^{n+1/2} \cdot \bphi,} \quad \bphi \in [L^2(\Gamma)]^2.
\end{array}
\end{equation}

This problem is similar to the structure sub-problem for the fluid-structure interaction problem
studied in  \cite{BorSun}, where the {\sl no-slip boundary condition} was considered,
and only the radial displacement of the thin structure was assumed to be different from zero. 
Using the same ideas as in \cite{BorSun} one can show that the following
existence result and energy estimate hold for problem \eqref{StructureWeak}:

\begin{proposition}\label{prop:Energy1}
For each fixed $\Delta t > 0$, 
problem \eqref{StructureWeak} has a unique solution $({\bf v}^{n+\frac 1 2},\bdeta^{n+\frac 1 2})\in [H^2_0(\Gamma)]^2\times [H_0^2(\Gamma)]^2$.
Moreover, the solution of problem \eqref{StructureWeak} satisfies the following discrete energy inequality:
\begin{equation}
\begin{array}{c}
\displaystyle{E_{\Delta t}^{n+\frac 1 2}+\frac 1 2\big (\rho_sh\|{\bf v}^{n+\frac 1 2}-{\bf v}^{n}\|^2+
c \| \bdeta^{n+\frac 1 2}-\bdeta^{n} \|^2_{H^2_0} \big) \le E_{\Delta t}^n,}
\end{array}
\label{StructureEnergyIE}
\end{equation}
where the kinetic energy $E_{\Delta t}^n$ is defined in \eqref{DiscreteEnergy}, and $c$ is the coercivity constant defined in \eqref{coercivity}.
\end{proposition}

\proof
The proof of this proposition is similar to the proof in Propositions 1 and 2 in \cite{BorSun}.
The existence of a unique weak solution follows from the Lax-Milgram Lemma as in \cite{BorSun}. 

The energy inequality is obtained by using ${\bf v}^{n+1/2}$ in place of the structure velocity test function $\bpsi$
in the first term in the first equation in 
\eqref{StructureWeak}, and by using  $\frac{\bdeta^{n+1/2} - \bdeta^n}{\Delta t}$ in place of the structure velocity test function $\bpsi$
in the second term in the first equation in \eqref{StructureWeak}.
After a calculation incorporating the equality $|(a-b) a| = \frac 1 2 (|a|^2 + |a-b|^2 - |b|^2)$, one gets 
\begin{align}
\frac{\rho_S h}{2} \|{\bf v}^{n+1/2}\|^2_{L^2}+\frac{\rho_S h}{2} \|{\bf v}^{n+1/2} - {\bf v}^{n} \|^2_{L^2} +
&\frac{1}{2} \langle {\cal{L}}_e \bdeta^{n+1/2},\bdeta^{n+1/2}\rangle +
\frac{1}{2} \langle {\cal{L}}_e (\bdeta^{n+1/2}-\bdeta^n),(\bdeta^{n+1/2}-\bdeta^n)\rangle 
\nonumber
\\
&= \frac{\rho_S h}{2} \|{\bf v}^{n}\|^2_{L^2} + \frac{1}{2} \langle {\cal{L}}_e \bdeta^{n},\bdeta^{n}\rangle.
\nonumber
\end{align}
We add the term $\frac{\rho_f}{2} \int_\Omega J^n |{\bf u}_{\Delta t}^n |^2$ on both sides of the above equality,
and use the coercivity property of ${\cal{L}}_e$ to obtain the energy inequality \eqref{StructureEnergyIE}.
\qed

{\bf Remark.} We would like to draw the attention of the reader to the fact that the estimate of the term
$\| \bdeta^{n+\frac 1 2}-\bdeta^{n} \|^2_{H^2_0}$ was made possible by the fact that we
are performing the time-discretization via operator splitting, and study semi-discretized problems in time. 
As a result, the approximation of the structure velocity, which was semi-disretized using the backward Euler method, 
gives rise to the term
$\| \bdeta^{n+\frac 1 2}-\bdeta^{n} \|^2_{H^2_0}$, which corresponds to the well-known numerical dissipation term. 
By the iterative application of inequality \eqref{StructureEnergyIE},
the sum with respect to $n$ of the differences  $\| \bdeta^{n+\frac 1 2}-\bdeta^{n} \|^2_{H^2_0}$
 is uniformly bounded by a constant,
which only depends on the initial kinetic energy and the inlet and outlet boundary data, as stated in Proposition~\ref{Stability}
below. 
This estimate gives additional information about the behavior in time of the structure displacement, 
which will be crucial for the compactness arguments established in Section~\ref{Sec:strong}.

\vskip 0.2in

The structure sub-problem updates the position of the elastic boundary $\bdeta^{n+1}$, based on which we can now calculate
 the ALE mapping ${\bf A}^{n+1}$ as the harmonic extension of $\bdeta^{n+1}$, i.e. ${\bf A}^{n+1}={\bf id}+{\bf B}^{n+1}$ where ${\bf B}^{n+1}$ is defined as the solution of the following boundary value problem:
\begin{equation}\label{ALEapp}
\displaystyle{\Delta {\bf B}^{n+1}=0\;{\rm in}\;\Omega,\; {\bf B}^{n+1}_{|\Gamma}=\bdeta^{n+1},\; {\bf B}^{n+1}_{|\Sigma}=0.}
\end{equation}
The corresponding discrete version of the ALE velocity and the Jacobian of the ALE mapping are defined by:
\begin{equation}\label{ALEVelDiscrete}
{\bf w}^{n+1}=\frac{{\bf A}^{n+1}-{\bf A}^n}{\Delta t},\; J^{n+1}=\det \nabla {\bf A}^{n+1}.
\end{equation}

\vskip 0.2in
\noindent
{\bf THE FLUID SUB-PROBLEM} (Differential formulation):
\begin{equation}\label{FluidSub1}
\left.
\begin{array}{rcl}
\rho_F\displaystyle{\frac{{\bf u}^{n+1}-{\bf u}^n}{\Delta t}+\rho_F(({\bf u}^n-{\bf w}^{n+1})\cdot\nabla^{\eta^{n+1}}) {\bf u}^{n+1}}
&=&\displaystyle{\nabla^{\eta^{n+1}}\cdot \bsigma^{\eta^{n+1}}({\bf u}^{n+1},p^{n+1})}+{\bf R}^{n+1}\; \\
{\rm where}\ \displaystyle{{\bf R}^{n+1}}&=&\displaystyle{\frac{1}{\Delta t}\int_{n\Delta t}^{(n+1)\Delta t}{\bf R}}
\end{array}
\right\}\quad {\rm in}\ \Omega
\end{equation}

\begin{equation}
\left. \begin{array}{rcl}
\displaystyle{\frac{{\bf v}^{n+1}-{\bf v}^{n+1/2}}{\Delta t}}&=&\displaystyle{-S^{\eta^{n+1}}\bsigma^{\eta^{n+1}}({\bf u}^{n+1},p^{n+1})\bnu^{\eta^{n+1}}} \; \\
\displaystyle{({\bf u}^{n+1}-{\bf v}^{n+1})\cdot\bnu^{\eta^{n+1}}}&=&\displaystyle{ 0}
\\ 
\displaystyle{\alpha \boldsymbol\sigma^{\eta^{n+1}}({\bf u}^{n+1},p^{n+1})\bnu^{\eta^{n+1}}\cdot\btau^{\eta^{n+1}}}&=&
\displaystyle{({\bf v}^{n+1}-\textbf{u}^{n+1})\cdot\btau^{\eta^{n+1}}}\;
\end{array}
\right\}
\quad {\rm on}\; \Gamma,
\end{equation} 
This system is supplemented with boundary conditions~\eqref{ProblemBC}.

Before we state the corresponding weak formulation, we introduce the following abbreviations to simplify notation:
\begin{equation}\label{Pokrate}
\nabla^{n}:=\nabla^{\eta^n},\; \bsigma^n:=\bsigma^{\eta^n},\; \bnu^n:=\bnu^{\eta^n},\; \btau^n:=\btau^{\eta^n},\; u^n_{\nu}:={\bf u}^{n}\cdot\bnu^n,\;u^n_{\tau}:={\bf u}^{n}\cdot\btau^n.
\end{equation}
We now define the weak solution function space for the fluid sub-problem given in terms of the fluid velocity ${\bf u}$ and its trace ${\bf v}$
on $\Gamma$ as:
$$
{\mathcal V}^n=\{({\bf u},{\bf v})\in H^1(\Omega)^2\times L^2(\Gamma)^2:\nabla^{n}\cdot{\bf u}=0,\; ({\bf u}-{\bf v})\cdot\bnu^n=0,
$$
$$
{\bf u}\cdot\btau=0,\; {\rm on}\ \Gamma_i,\;i\in I,\;{\bf u}=0\ {\rm on}\ \Gamma_i,\;i\in II,\; {\bf u}\cdot\bnu=0,\; {\rm on}\ \Gamma_i,\;i\in III\cup IV \}.
$$
The weak  formulation is defined as follows:
find $({\bf u}^{n+1},{\bf v}^{n+1})\in{\mathcal V}^n$ such that
\begin{equation}\label{WeakFluid}
\begin{array}{c}
\displaystyle{\rho_F\int_{\Omega}J^{n}\frac{{\bf u}^{n+1}-{\bf u}^n}{\Delta t}\cdot{\bf q}
+\frac{\rho_F}{2}\int_{\Omega}\frac{J^{n+1}-J^n}{\Delta t}{\bf u}^{n+1}\cdot{\bf q}
+\frac{\rho_F}{2}\int_{\Omega}J^{n+1}\Big ((({\bf u}^{n}-{\bf w}^{n+1})\cdot\nabla^{n+1}){\bf u}^{n+1}\cdot {\bf q}}\\ \\
\displaystyle{-(({\bf u}^n-{\bf w}^{n+1})\cdot\nabla^{n+1}){\bf q}\cdot {\bf u}^{n+1}\Big)
+2\mu\int_{\Omega} J^{n+1} {\bf D}({\bf u}^{n+1}):{\bf D}({\bf q})
+\sum_{i\in III}\int_{\Gamma_i}\frac{1}{\alpha_i}{u}^{n+1}_\tau q_\tau}
\\ \\
\displaystyle{
+\frac{1}{\alpha}\int_{\Gamma}(u_{\tau}^{n+1}-v_{\tau}^{n+1})q_{\tau^{n+1}}S^{n+1}dz}
\displaystyle{+\rho_S h\int_{\Gamma} \frac{{\bf v}^{n+1}-{\bf v}^{n+1/2}}{\Delta t}\bpsi+\frac{1}{\alpha}\int_{\Gamma}({v}^{n+1}_{\tau}-u^{n+1}_{\tau})\psi_{\tau^{n+1}}S^{n+1}dz}\\
\\
\displaystyle{=\langle {\bf R}^{n+1},{\bf q}\rangle,\quad ({\bf q},\bpsi)\in{\mathcal V}^n.}
\end{array}
\end{equation}
This is obtained by considering \eqref{WeakALEFormula} and by using formula \eqref{J_t} to express
$J^\eta (\nabla^\eta \cdot {\bf w}^\eta) = \partial_t J^\eta$. Furthermore, in the third and fourth integral we replaced $J^n$ with $J^{n+1}$ 
for higher accuracy. This, however, does not influence the existence proof. Either choice works well.

\begin{proposition}
Let $\Delta t>0$, and $J^{n+1}\geq c>0$. Then there exists a unique solution $({\bf u}^{n+1},{\bf v}^{n+1})$ to 
the fluid sub-problem~\eqref{WeakFluid}. Furthermore, the solution satisfies the following semi-discrete energy inequality:
\begin{equation}\label{DiscreteEIFluid}
\begin{array}{c}
\displaystyle{E^{n+1}_{\Delta t}+\frac{\rho_F}{2}\int_{\Omega}J^n\|{\bf u}^{n+1}-{\bf u}^n\|^2_{L^2(\Omega)}+\frac{\rho_S h}{2}\|{\bf v}^{n+1}-{\bf v}^{n+1/2}\|^2_{L^2(\Gamma)}}\\ \\ 
\displaystyle{+  D^{n+1}_{\Delta t} \leq E^{n+1/2}_{\Delta t}
+C \Delta t \|{\bf R}^{n+1}\|^2_{(H^{1}(\Omega))'},}
\end{array}
\end{equation}
where $(H^1(\Omega))'$ is the dual space of $H^1(\Omega)$. 
\label{ExistenceFluidSub}
\end{proposition}
\proof

The existence proof follows  from the Lax-Milgram Lemma and the transformed Korn's inequality stated in Lemma~\ref{Korn},
see \cite{BorSun}.

The semi-discrete energy inequality \eqref{DiscreteEIFluid}
is a consequence of the fact that we have discretized our fluid sub-problem given by \eqref{WeakFluid}
so that the discrete version of the geometric conservation law is satisfied. 
More precisely, by taking ${\bf q}={\bf u}^{n+1}$
in the following two terms in the weak formulation~\eqref{WeakFluid}:
$$
\rho_F\int_{\Omega}J^n\frac{{\bf u}^{n+1}-{\bf u}^n}{\Delta t}\cdot{\bf q}+\frac{\rho_F}{2}\int_{\Omega}\frac{J^{n+1}-J^n}{\Delta t}{\bf u}^{n+1}\cdot{\bf q}
$$
we see that with this kind of discretization a {\em semi-discrete version of the geometric conservation law is exactly satisfied}, i.e. we have:
\begin{equation}\label{Geometric_CL}
\frac{\rho_F}{2}\Big (\int_{\Omega}J^{n+1}|{\bf u}^{n+1}|^2+\int_{\Omega}J^{n}|{\bf u}^{n+1}-{\bf u}^n|^2\Big )=\frac{\rho_F}{2}\int_{\Omega}J^{n}|{\bf u}^{n}|^2.
\end{equation}
Thus, the fluid kinetic energy at time $(n+1)\Delta t$ plus the kinetic energy due to the fluid domain motion, is exactly equal to the 
fluid kinetic energy at time $n \Delta t$.
The two terms on the left hand-side in \eqref{Geometric_CL} appear on the left hand-side in the energy estimate \eqref{DiscreteEIFluid}, 
while the term on the right hand-side in \eqref{Geometric_CL} appears on the right hand-side in the energy estimate \eqref{DiscreteEIFluid}.
\qed

\begin{proposition}{\bf (Uniform semi-discrete energy estimates)}
Let $\Delta t > 0$ and $N=T/\Delta t > 0$. 
Furthermore, let $E_{\Delta t}^{n+\frac 1 2}, E_{\Delta t}^{n+1}$, and $D_{\Delta t}^n$ be the kinetic energy and dissipation 
given by \eqref{DiscreteEnergy} and \eqref{Discrete dissipation}, respectively.
There exists a constant $C>0$ independent of $\Delta t$, which depends only on the parameters in the problem, 
on the kinetic energy of the initial data $E_0$, and on the  norm of the right-hand side $\|{\bf R}\|_{L^2(0,T;H^{1}(\Omega)')}^2$ (i.e. on the boundary data),
such that the following estimates hold:
\vskip 0.1in
\noindent
1. $\displaystyle{E_{\Delta t}^{n+\frac 1 2}\leq C, E_{\Delta t}^{n+1}\leq C}$, for all $ n = 0,...,N-1, $\\
2. $\displaystyle{\sum_{n=1}^ND_{\Delta t}^n\leq C},$\\
3. $\displaystyle{\sum_{n=0}^{N-1}
\Big(
\int_{\Omega}J^n_{\Delta t} } |{\bf u}_{\Delta t}^{n+1}-{\bf u}_{\Delta t}^n|^2+\|{\bf v}_{\Delta t}^{n+1}-{\bf v}_{\Delta t}^{n+\frac 1 2}\|^2_{L^2(\Gamma)}
 +c\|\bdeta_{\Delta t}^{n+1}-\bdeta_{\Delta t}^n\|^2_{H^2(\Gamma)}+\|{\bf v}_{\Delta t}^{n+\frac 1 2}-{\bf v}_{\Delta t}^{n}\|^2_{L^2(\Gamma)}
 \Big)\leq C.$
\vskip 0.1in
In fact, $C = E_0 + \tilde{C} \|{\bf R}\|^2_{L^2(0,T;H^{1}(\Omega)')}$, where $\tilde{C}$ is a constant which depends only on the parameters in the problem, and ${\bf R}$ is the term, defined in \eqref{SourceTermR},  coming from the dynamic pressure data.
\label{Stability}
\end{proposition}

\proof
The proof of Proposition~\ref{Stability} follows directly from the energy estimates \eqref{StructureEnergyIE} and \eqref{DiscreteEIFluid},
and from the H\"{o}lder inequality applied to the source term ${\bf R}$. More precisely, the statements in  Proposition~\ref{Stability}
are obtained after summing the combined energy estimates \eqref{StructureEnergyIE} and \eqref{DiscreteEIFluid} over $n = 0, \cdots, N-1$,
and after  taking into account that 
$$
\Delta t \sum_{n=0}^{N-1} \|{\bf R}^{n+1}\|^2_{(H^{1}(\Omega))'} = \Delta t \sum_{n=0}^{N-1} \| \frac{1}{\Delta t} \int_{n\Delta t}^{{n+1}\Delta t} {\bf R}\|^2 \le
\tilde{C} \| {\bf R} \|^2_{L^2(0,T; (H^{1}(\Omega))')}.
$$
\qed

This is a crucial estimate which will provide uniform boundedness of 
approximating solutions to problem \eqref{NSP}-\eqref{ProblemIC}, constructed using our 
semi-discretized scheme based on Lie splitting. 
However, notice that we have so far only defined the approximate values of our solution at discrete points in time, 
given by $n \Delta t$. We want to define approximate solutions to be defined at all the points in $(0,T)$. 
For this purpose we define approximate solutions  to be the functions which are piece-wise constant 
on each sub-interval $((n-1)\Delta t,n\Delta t],\ n=1\dots $ of $(0,T)$, such that for 
$t\in ((n-1)\Delta t,n\Delta t],\ n=1\dots N,$
\begin{equation}
{\bf u}_{\Delta t}(t,.)={\bf u}_{\Delta t}^n,\ \eta_{\Delta t}(t,.)=\eta_{\Delta t}^n,\ {\bf v}_{\Delta t}(t,.)={\bf v}_{\Delta t}^n,\ v^*_{\Delta t}(t,.)={\bf v}^{n-\frac 1 2}_{\Delta t}.
\label{aproxNS}
\end{equation}
See Figure~\ref{fig:u_N}. We define other approximate quantities in analogous way, i.e.
$$
{\bf A}_{\Delta t}(t,.)={\bf A}_{\Delta t}^n,\ \bnu_{\Delta t}(t,.)=\bnu_{\Delta t}^n,\ \btau_{\Delta t}(t,.)=\btau_{\Delta t}^n,\; {\bf w}_{\Delta t}(t,.)={\bf w}^n_{\Delta t},\ S_{\Delta t}(t,.)=S_{\Delta t}^n,\ J_{\Delta t}(t,.)=J_{\Delta t}^n.
$$
\begin{figure}[ht]
\centering{
\includegraphics[scale=0.45]{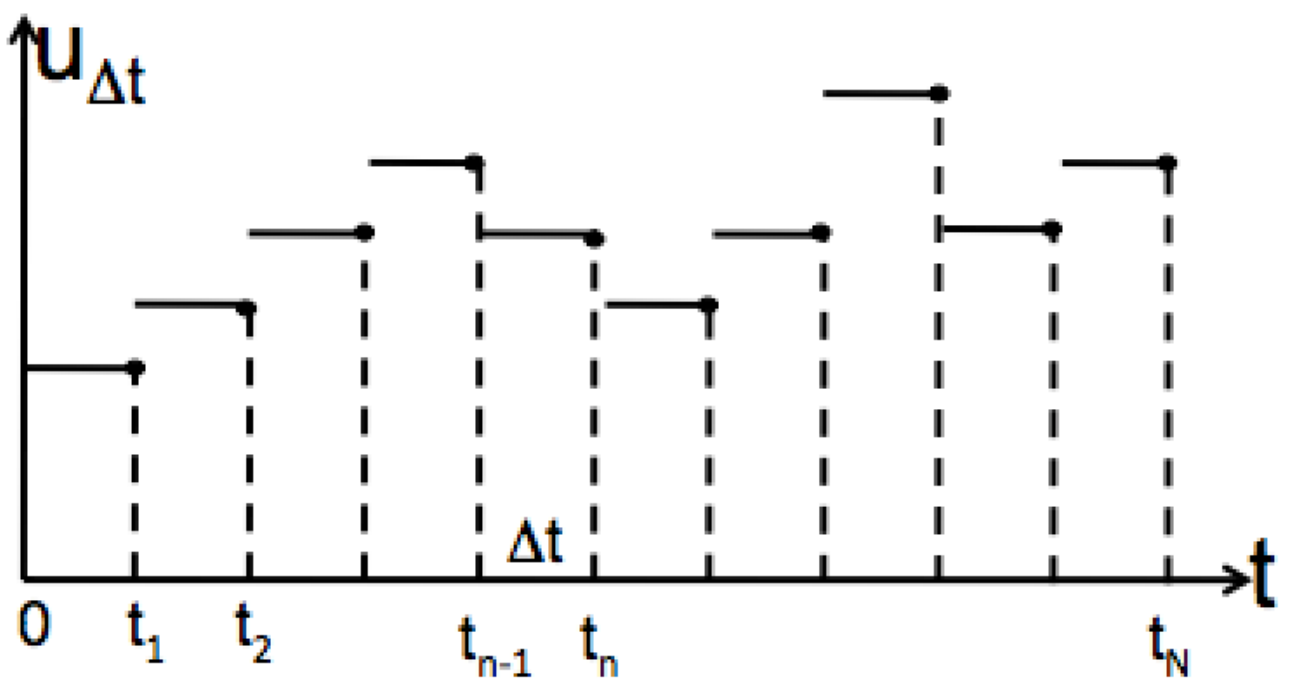}
}
\caption{A sketch of $u_{\Delta t}$.}
\label{fig:u_N}
\end{figure}
Our goal is to show that there exists a subsequence of the sequence of approximating solutions
 defined above, which converges to a weak solution of problem \eqref{NSP}-\eqref{ProblemIC}. 
 
\section{Convergence of approximate solutions}

\subsection{Weak and weak* convergence}\label{Sec:weak}
We now focus on the sequences of approximate solutions 
$(\bdeta_{\Delta t})_{\Delta t}$, $({\bf u}_{\Delta t})_{\Delta t}$, and $({\bf v}_{\Delta t})_{\Delta t}$,
as $\Delta t \to 0$ (or, equivalently, as $N \to \infty$). We first show that these sequences are uniformly bounded,
independently of $\Delta t$,
in the appropriate function norms. The main ingredient in showing the 
uniform estimates 
will be the results of Proposition~\ref{Stability}.
\begin{proposition}\label{EtaBound}
The sequence $(\bdeta_{\Delta t})_{\Delta t}$ is uniformly bounded in $L^{\infty}(0,T;H_0^2(\Gamma))^2$. 
Moreover, there exists a $T$ small enough such that ${\bf A}_{\Delta t}^n$ is an injection, and
$$
J^n_{\Delta t}=\det \nabla {\bf A}_{\Delta t}^n>0,\quad\Delta t>0,n=1,\dots N.
$$
\end{proposition}
\proof
From Proposition~\ref{Stability} we have that $E_{\Delta t}^n \le C$, where $C$ is independent of $\Delta t$. This implies
$$
\|\bdeta_{\Delta t}\|_{L^\infty(0,T;H_0^2(\Gamma))} \le C.
$$
To show that  ${\bf A}_{\Delta t}^n$  is injective we again use Theorem 5.5.-1 from \cite{CiarletBook1}, stated as Theorem~\ref{Ciarlet} in this
manuscript, in the analogous way as in Proposition~\ref{ALEinjectivity}. First, we fix $\Delta t$ and consider the $H^2_0$-norm of the difference between the initial data and the approximate
solution at time $n\Delta t$, to notice that
$$
\|\bdeta^n_{\Delta t}-\bdeta_0\|_{H_0^2(\Gamma)}\leq \| \bdeta^n_{\Delta t}\|_{H_0^2(\Gamma)} + \| \bdeta_0 \|_{H_0^2(\Gamma)} \leq 2C,\; n=1,\dots, N.
$$
Furthermore, we calculate
$$
\|\bdeta^n_{\Delta t}-\bdeta_0\|_{L^2(\Gamma)}\leq \sum_{i=0}^{n-1}\|\bdeta^{i+1}_{\Delta t}-\bdeta^{i}_{\Delta t}\|_{L^2(\Gamma)}
=\Delta t\sum_{i=0}^{n-1}\|{\bf v}^{i+\frac 1 2}_{\Delta t}\|_{L^2(\Gamma)} \le CT,
$$
where we used that $\bdeta^{0}_{\Delta t}=\bdeta_0$, and
Proposition~\ref{Stability}, estimate $E_{\Delta t}^{n+\frac 1 2}\leq C$, where $C$ is independent of $\Delta t$,
to bound
$$
\|\bdeta^n_{\Delta t}-\bdeta_0\|_{L^2(\Gamma)}\leq C n \Delta t\leq CT,\; n=1,\dots, N.
$$
Now, we have uniform bounds for $\|\bdeta^n_{\Delta t}-\bdeta_0\|_{L^2(\Gamma)}$ and $\|\bdeta^n_{\Delta t}-\bdeta_0\|_{H_0^2(\Gamma)}$.
Therefore, we can use the interpolation inequality for Sobolev spaces (see for example \cite{ADA}, Thm. 4.17, p. 79) to get
$$
\|\bdeta^n_{\Delta t}-\bdeta_0\|_{H^{11/6}(\Gamma)}\leq 2CT^{1/12},\; n=1,\dots, N.
$$ 
Now from~\eqref{ALEreg2} and the construction of the ALE mapping we have
$$
\|\nabla{\bf A}_{\Delta t}^n-{\bf I}\|_{C^{0,1/3}(\Omega)}\leq \tilde{C} \|\bdeta^n_{\Delta t}\|_{H^{11/6}(\Gamma)}\leq\|\bdeta_0\|_{H^{11/6}(\Gamma)}+2CT^{1/12},
$$
where ${\bf I}$ is the identity matrix. 

We want to show that the right hand-side is bounded by a constant which is smaller than or equal to the constant $c(\Omega)$
given by \eqref{c_Omega} in Theorem~\ref{Ciarlet} (B), which guarantees injectivity of the ALE mapping. 
Indeed, 
from Proposition~\ref{Stability} we see that $C$ above depends on $T$ through the norms of the inlet and outlet data in such a way that $C$ is an increasing function of $T$.
Therefore by choosing $T$ small, we can make $\|\bdeta^n_{\Delta t}-\bdeta_0\|_{H^{11/6}(\Gamma)}$ arbitrarily small for $n=1,.\dots,N$. 
Furthermore, by using the assumption \eqref{eta0} on the smallness of the initial domain displacement, 
we see that we can choose a $T'' > 0$ small enough, {\em independent of $\Delta t$}, such that  there exists a constant $\tilde{c} > 0$ 
giving
$$
\|\nabla{\bf A}_{\Delta t}^n-{\bf I}\|_{C^{0,1/3}(\Omega)} \leq \tilde{c},
$$
where $\tilde{c}$ is smaller than or equal to the constant $c(\Omega)$ from Theorem~\ref{Ciarlet} (B), which 
implies injectivity of the ALE mapping 
$\nabla{\bf A}_{\Delta t}^n$.

Finally, from condition \eqref{StructureCC}  requiring that the ``initial'' Jacobian of $\bphi^0$ is strictly positive,
we see that there exists a $T' > 0$ small, {\em independent of $\Delta t$}, such that the Jacobian $J_{\Delta t}^n$ is strictly positive. 
By taking 
\begin{equation}\label{Tdescrete}
T = {\rm min}\{T',T''\},
\end{equation}
we obtain the proof of the proposition.
\qed

We note that it is this $T$, given by Proposition~\ref{EtaBound},
that determines the time interval of existence of a weak solution to problem
\eqref{NSP}-\eqref{ProblemIC},
since $T$ given in \eqref{Tdescrete} is independent of $\Delta t$. 
The time given in \eqref{T} is a continuous version of the time given by \eqref{Tdescrete}. 
 
 Next, we show uniform boundedness of the approximating sequences for the fluid and structure velocities.

\begin{proposition}\label{velocity_bounds}
The following statements hold:
\begin{enumerate}
\item $({\bf v}_{\Delta t})_{\Delta t}$, $({\bf v}_{\Delta t}^*)_{\Delta t}$  are uniformly bounded in $L^\infty(0,T;L^2(\Gamma))$,
\item $({\bf u}_{\Delta t})_{\Delta t}$ is uniformly bounded in $L^\infty(0,T;L^2(\Omega))\cap L^2(0,T;H^1(\Omega))$.
\end{enumerate}
\end{proposition}
\proof
The proof follows from Propositions~\ref{Stability} and \ref{EtaBound} and from the uniform Korn's inequality 
stated in Lemma~\ref{Korn}.  
Namely, statement 2 of Proposition~\ref{Stability} implies that ${\bf D}^{\eta_{\Delta t}}({\bf u}_{\Delta t})$ is bounded in $L^2(0,T;L^2(\Omega))$. 
Now, we can use the uniform bound for $\bdeta_{\Delta t}$ and the uniform Korn's inequality from Lemma~\ref{Korn} to finish the proof.
\qed

\begin{lemma}{\bf (Weak and weak* convergence results)} \label{weak_convergence}
There exist subsequences $(\bdeta_{\Delta t})_{\Delta t}, ({\bf v}_{\Delta t})_{\Delta t}, ({\bf v}^*_{\Delta t})_{\Delta t}, $ and $({\bf u}_{\Delta t})_{\Delta t}$, 
and the functions $\eta \in L^{\infty}(0,T;H^2_0(\Gamma))$, 
${\bf v}\in L^{\infty}(0,T;L^2(\Gamma))$, 
${\bf v}^*\in L^{\infty}(0,T;L^2(\Gamma))$, and ${\bf u}\in L^{\infty}(0,T;L^2(\Omega))\cap L^2(0,T;H^1(\Omega))$,
such that
\begin{equation}\label{weakconv}
\begin{array}{rcl}
\bdeta_{\Delta t} &\rightharpoonup & \bdeta \; {\rm weakly*}\; {\rm in}\; L^{\infty}(0,T;H^2_0(\Gamma)),
\\
{\bf v}_{\Delta t} &\rightharpoonup & {\bf v}\; {\rm weakly*}\; {\rm in}\; L^{\infty}(0,T;L^2(\Gamma)),
\\
{\bf v}^*_{\Delta t} &\rightharpoonup & {\bf v}^*\; {\rm weakly*}\; {\rm in}\; L^{\infty}(0,T;L^2(\Gamma)),
\\
{\bf u}_{\Delta t} &\rightharpoonup & {\bf u}\; {\rm weakly*}\; {\rm in}\; L^{\infty}(0,T;L^2(\Omega)),
\\
{\bf u}_{\Delta t} &\rightharpoonup & {\bf u}\; {\rm weakly}\; {\rm in}\; L^{2}(0,T;H^1(\Omega)).
\end{array}
\end{equation}
Furthermore,
\begin{equation}\label{v_star}
{\bf v} = {\bf v}^*.
\end{equation}
\end{lemma}

Since our problem is nonlinear, we need strong convergence of approximating sub-sequences to be able to pass to the limit, and
show that the limiting functions satisfy the weak formulation of the problem. For this purpose we need a compactness result,
which we present next.

\subsection{Compactness}\label{Sec:strong}
Compactness arguments in the case when the boundary condition on $\Gamma(t)$ is the {\sl slip condition} is different from the compactness
argument for the problem in which the boundary condition on $\Gamma(t)$ is {\sl no-slip}. 
In the no-slip case the viscous dissipation of the fluid smooths out the interface providing spatial regularity of the interface velocity 
${\bf v}_{\Delta t}$, which is no longer available in the slip condition case, since the fluid and structure velocities are no longer equal 
at the fluid-structure interface. We will have some help from the viscous fluid dissipation in the estimates for the 
normal component of structure velocity,
but the tangential component no longer ``feels'' fluid dissipation as before. It is because of this that we need to construct a different 
compactness argument, which we present next. 

To investigate compactness (regularization) in time we introduce 
the translation in time by $h$ of a function $f$,
denoted by $T_h$, as:
\begin{equation}\label{shift}
T_h f(t,.)=f(t-h,.),\  h\in\R. 
\end{equation}

The following estimates hold for our approximate solutions as they are shifted in time by  $h$:
\begin{lemma}\label{Translations}
There exists a constant $C>0$, independent of $\Delta t$, such that for every $h>0$ we have:
$$
\|T_h{\bf u}_{\Delta t}-{\bf u}_{\Delta t}\|_{L^2(h,T;L^2(\Omega))}\leq C\sqrt{h},
$$
$$
\|T_h{\bf v}_{\Delta t}-{\bf v}_{\Delta t}\|_{L^2(h,T;L^2(\Gamma))}\leq C\sqrt{h},
$$
$$
\|T_h{\bf v}^*_{\Delta t}-{\bf v}^*_{\Delta t}\|_{L^2(h,T;L^2(\Gamma))}\leq C\sqrt{h},
$$
$$
\|T_h\bdeta_{\Delta t}-\bdeta_{\Delta t}\|_{L^2(h,T;H^2(\Gamma))}\leq C\sqrt{h},\; \Delta t>0.
$$
\end{lemma}
\proof
The proof is analogous to the proof of Theorem 2 in \cite{BorSun}. 
A summary of the main steps is the following. 
We focus on the first statement given in terms of ${\bf u}_{\Delta t}$, while the proofs for the other two are analoguous. 
First, from Proposition~\ref{Stability} we immediately have:
$$
\|T_{\Delta t}{\bf u}_{\Delta t}-{\bf u}_{\Delta t}\|_{L^2(\Delta t,T;L^2(\Omega))}^2=\sum_{n}\|{\bf u}_{\Delta t}^{n+1}-{\bf u}_{\Delta t}^n\|^2_{L^2(\Omega)}\Delta t\leq C\Delta t.
$$
This implies that the first estimate in the above Lemma holds for ``the diagonal'' terms for which the translation is performed 
by $h$ that is exactly equal to $\Delta t$.
However, we would like to prove the statement for an  arbitrary translation by $h > 0$. 
Let us fix $\Delta t>0$ and consider the following two cases: $0<h<\Delta t$ and $0<\Delta t<h$. 
We obtain the desired estimates by calculating the following.
\begin{enumerate}
\item For $0<h<\Delta t$ we have
$$
\|T_{h}{\bf u}_{\Delta t}-{\bf u}_{\Delta t}\|_{L^2(h,T;L^2(\Omega))}^2=\sum_{n}\int_{n\Delta t-h}^{n \Delta t}\|{\bf u}_{\Delta t}^{n+1}-{\bf u}_{\Delta t}^n\|^2_{L^2(\Omega)}
$$
$$
=h\sum_{n}\|{\bf u}_{\Delta t}^{n+1}-{\bf u}_{\Delta t}^n\|^2_{L^2(\Omega)}
\leq Ch.
$$
\item For $0<\Delta t<h$ we write $h=l\Delta t+s$ for some $l\in\N$ and $0\leq s<\Delta t$, and use the triangle inequality to get
$$
\|T_{h}{\bf u}_{\Delta t}-{\bf u}_{\Delta t}\|_{L^2(h,T;L^2(\Omega))}^2\leq\Delta t \sum_{n}\sum_{i=1}^{l+1}\|{\bf u}_{\Delta t}^{j+i-1}-{\bf u}_{\Delta t}^{j+i}\|^2_{L^2\Gamma}\leq C(l+1)\Delta t\leq C h.
$$
\end{enumerate}
\qed

In what follows, it will be useful to introduce a slightly different set of approximate functions for $\mathbf u$, $v$, and $\eta$
by extending the values of those functions at points $n \Delta t$ to the time sub-interval $[(n-1)\Delta t,n\Delta t]$
not in a piece-wise constant fashion as before, but linearly.
Namely, for each fixed $\Delta t$, define $\tilde{\bf u}_N$, $\tilde{\bdeta}_{\Delta t}$ and $\tilde{{\bf v}}_{\Delta t}$ to be {\sl continuous}, {\sl linear} on
each sub-interval $[(n-1)\Delta t,n\Delta t]$, and such that
\begin{equation}\label{tilde}
\tilde{\bf u}_{\Delta t}(n\Delta t,.)={\bf u}_{\Delta t}(n\Delta t,.),\ \tilde{{\bf v}}_{\Delta t}(n\Delta t,.)={{\bf v}}_{\Delta t}(n\Delta t,.),\ \tilde{\bdeta}_{\Delta t}(n\Delta t,.)={\bdeta}_{\Delta t}(n\Delta t,.), 
\end{equation}
where $n=0,\dots N$.
We now observe that
\begin{equation}\label{EtaLinear}
\displaystyle{\tilde{\bdeta}_{\Delta t}(t)=\frac{\bdeta^{n+1}-\bdeta^n}{\Delta t}(t-n\Delta t)+\bdeta^n,\quad t\in [n\Delta t,(n+1)\Delta t),\;n=0,\dots,N-1,}
\end{equation} 
$$\displaystyle{\partial_t\tilde{\bdeta}_{\Delta t}(t)=\frac{\bdeta^{n+1}-\bdeta^n}{\Delta t}= \frac{\bdeta^{n+1/2}-\bdeta^n}{\Delta t}={\bf v}^{n+\frac 1 2}},\ t\in (n\Delta t,(n+1)\Delta t),$$ 
and so, since ${\bf v}^*_{\Delta t}$ was defined in  \eqref{aproxNS} as
a piece-wise constant function defined via $v^*_{\Delta t}(t,\cdot)=v^{n+\frac 1 2}$, for $t\in(n\Delta t,(n+1)\Delta t]$,
we see that
\begin{equation}\label{derivativeeta}
\partial_t\tilde{\bdeta}_{\Delta t}={\bf v}^*_{\Delta t}\ a.e.\ {\rm on}\ (0,T).
\end{equation}
The following Lemma will be crucial for establishing compactness of the approximate sequence of solutions $\tilde\bdeta_{\Delta t}$.
\begin{lemma}\label{TranslationsTilde}
There exists a constant $C>0$, independent of $\Delta t$, such that for every $h>0$ we have:
$$
\|T_h{\tilde\bdeta}_{\Delta t}-{\tilde\bdeta}_{\Delta t}\|_{L^2(h,T;H^2(\Gamma))}\leq C\sqrt{h}.
$$
\end{lemma}
\proof
In the same way as in Lemma~\ref{Translations} we consider two separate cases: $0<h\leq\Delta t$ and $\Delta t<h.$

\noindent
{\bf Case 1: $0<h\leq\Delta t$}. We use~\eqref{EtaLinear} and explicitly calculate the straight lines defining the function
$\tilde{\bdeta}_{\Delta t}$ and its translation in time by $h$ to the right, to obtain:
$$
\int_h^T\|T_h{\tilde\bdeta}_{\Delta t}-{\tilde\bdeta}_{\Delta t}\|_{H^2(\Gamma)}^2=\int_h^{\Delta t}\|\frac{\bdeta^1_{\Delta t}-\bdeta_0}{\Delta t}(-h)\|_{H^2(\Gamma)}^2\ dt
$$
$$+\sum_{n=1}^{N-1}\Big\{\int_{n\Delta t}^{n\Delta t+h} \Big(\frac{1}{\Delta t^2}\|h(\bdeta_{\Delta t}^{n}-\bdeta_{\Delta t}^{n+1})+(n\Delta t-t)(\bdeta_{\Delta t}^{n+1}-2\bdeta_{\Delta t}^n+\bdeta_{\Delta t}^{n-1})\|_{H^2(\Gamma)}^2\Big) dt
$$
$$
+\int^{(n+1)\Delta t}_{n\Delta t+h}\|\frac{\bdeta_{\Delta t}^{n+1}-\bdeta_{\Delta t}^n}{\Delta t}(-h)\|_{H^2(\Gamma)}^2 dt\Big\}\leq C(\frac{h^3}{\Delta t^2}+\frac{h^2}{\Delta t})\sum_{n=0}^{N-1}\|\bdeta_{\Delta t}^{n+1}-\bdeta_{\Delta t}^n\|^2_{H^2(\Gamma)}\leq Ch.
$$
The last inequality follows from Proposition~\ref{Stability} and from the fact that $h\le \Delta t$.

\noindent
{\bf Case 2: $\Delta t<h$}. We notice that $T_h\tilde{\bdeta}_{\Delta t}=\widetilde{T_h\bdeta_{\Delta t}}$ and use the following identity (see e.g. \cite{Tem}, p. 328)
$$
\|\bdeta_{\Delta t}-\tilde{\bdeta}_{\Delta t}\|^2_{L^2(0,T;H^2(0,L))}\leq{\frac{\Delta t}{3}}\sum_{n=0}^{N-1}\|\bdeta^{n+1}_{\Delta t}-\bdeta^{n}_{\Delta t}\|^2_{H^2(0,L)}.
$$
From Lemma~\ref{Translations} we have:
$$
\|T_h{\tilde\bdeta}_{\Delta t}-{\tilde\bdeta}_{\Delta t}\|_{L^2(h,T;H^2(\Gamma))}\leq \|\widetilde{T_h\bdeta_{\Delta t}}-T_h\bdeta_{\Delta t}\|_{L^2(0,T;H^2(0,L))}
$$
$$
+\|T_h\bdeta_{\Delta t}-\bdeta_{\Delta t}\|_{L^2(0,T;H^2(0,L))}
+\|\bdeta_{\Delta t}-\tilde{\bdeta}_{\Delta t}\|_{L^2(0,T;H^2(0,L))}\leq C(\Delta t+h+\Delta t)\leq Ch.
$$
The statement of the Lemma immediately follows by taking the square root on both sides of the inequality.
\qed
\begin{proposition}\label{AdditionalReg}
The following statements hold:
\begin{enumerate}
\item $({\bf u}_{\Delta t})_{\Delta t>0}$ is uniformly bounded in $H^s(0,T;L^2(\Omega))$, $0\leq s<1/2$.
\item $({\tilde\bdeta}_{\Delta t})_{\Delta t>0}$ is uniformly bounded in $H^s(0,T;H^2(\Gamma))\cap H^{s+1}(0,T;L^2(\Gamma))$, $0\leq s<1/2$.
\item $({\bf v}^*_{\Delta t})_{\Delta t>0}$ is uniformly bounded in $H^s(0,T;L^2(\Gamma))\cap L^2(0,T;H^{2s}(\Gamma))$, $0\leq s<1/2$. 
\end{enumerate}
\end{proposition}
\proof
Since we have already proved that ${\bf u}_{\Delta t}$ is uniformly bounded in $L^2(0,T;L^2(\Omega))$, it only remains to prove that the following semi-norm is finite:
$$
\|{\bf u}_{\Delta t}\|_{H^{s}(0,T;L^2(\Omega))}^2=\int_0^T\int_0^T\frac{\|{\bf u}_{\Delta t}(t)-{\bf u}_{\Delta t}(\tau) \|_{L^2(\Omega)}^2}{| t-\tau |^{1+2s}}dt d\tau.
$$
By a simple change of variables $h=t-\tau$ and Lemma~\ref{Translations} we get:
$$
\|{\bf u}_{\Delta t} \|_{H^{s}(0,T;L^2(\Omega))}^2=\int_{-T}^{T}\frac{dh}{|h|^{1+2s}}\int_0^T  \|{\bf u}_{\Delta t}(t-h)-{\bf u}_{\Delta t}(t) \|_{L^2({\Omega})}^2d\tau
\leq\int_{-T}^{T}\frac{|h|dh}{|h|^{1+2s}}.
$$
This integral is finite for $s<1/2$ and therefore we have proved the first statement. 

The first part of the second statement and the first part of the third statement, i.e. $({\tilde\bdeta}_{\Delta t})_{\Delta t>0}$ and 
$({\bf v}^*_{\Delta t})_{\Delta t>0}$ are uniformly bounded in $H^s(0,T;H^2(\Gamma))$ and $H^s(0,T;L^2(\Gamma))$, $0 \le s < 1/2$, respectively, are proved analogously by using Lemma~\ref{Translations} again.

Let us now prove the boundedness of $({\bf v}^*_{\Delta t})_{\Delta t > 0}$ in $L^2(0,T;H^{2s}(\Gamma))$.
First notice that from~\eqref{derivativeeta} and the uniform boundedness of $({\bf v}^*_{\Delta t})_{\Delta t>0}$ 
in $H^s(0,T;L^2(\Gamma))$, $s<1/2$  that we just proved,
we see that $(\partial_t{\tilde \bdeta}_{\Delta t})$ is uniformly bounded in $H^s(0,T;L^2(\Gamma))$, $s<1/2$. Therefore,
$$
({\tilde \bdeta}_{\Delta t})\;{\rm is\; uniformly\; bounded\; in}\; H^{s+1}(0,T;L^2(\Gamma))\cap H^{s}(0,T;H^2(\Gamma)),\; s<1/2.
$$
Now, by the interpolation property (see e.g. \cite{LionsMagenes} Section 1.9.4 p.p. 47)  
we obtain that $({\tilde \bdeta}_{\Delta t})$ is uniformly bounded in $H^1(0,T;H^{2s}(\Gamma))$. 
By using~\eqref{derivativeeta} again we conclude that $({\bf v}^*_{\Delta t})_{\Delta t>0}=(\partial_t\tilde{\bdeta}_{\Delta t})$ 
is uniformly bounded in $L^2(0,T;H^{2s}(\Gamma))$, $s<1/2$.
\qed

Notice that the proof of this Proposition heavily relies on the definition of the new approximate solution sequences \eqref{tilde}
and their properties \eqref{EtaLinear} and \eqref{derivativeeta}. 
In particular, \eqref{derivativeeta} allowed us to obtain information about the regularity properties of 
$({\bf v}^*_{\Delta t})_{\Delta t>0}$ via the regularity of $(\partial_t{\tilde \bdeta}_{\Delta t})$.
It is because of this result that we introduced the new definition of approximate solutions given by \eqref{tilde}.
We shall see below that the limits of approximate sequences as $\Delta t \to 0$
do not depend on the type of extension of the approximate values of the solution at points $n \Delta t$
onto the time sub-interval $(n \Delta t, (n+1) \Delta t)$. Therefore, the introduction of the new approximate sequences
in \eqref{tilde} to obtain additional information about the regularity of the limiting solution is justified. 

We note here a side remark that we could have proved Proposition~\ref{AdditionalReg} 
by using a slightly different but equivalent approach, relying on {\bf Nikolskii spaces}.  
More precisely, from Lemma~\ref{Translations} we can directly conclude that sequences $({\bf u})_{\Delta t}$, ${(\bf v})_{\Delta t}$, $({\bf v})^*_{\Delta t}$, $(\bdeta)_{\Delta t}$ are uniformly bounded in the following Nikolskii spaces $N^{1/2,2}((0,T);L^2(\Omega))$, $N^{1/2,2}((0,T);L^2(\Gamma))$ and $N^{1/2,2}((0,T);H^2(\Gamma))$, respectively (see e.g. \cite{Simon2} for the definition of the Nikolskii spaces). Similarly, Lemma~\ref{TranslationsTilde} gives uniform boundedness of $(\tilde{\bdeta})_{\Delta t}$ in $N^{1/2,2}((0,T);H^2(\Gamma))$. Proposition~\ref{AdditionalReg} then follows  directly from the embeddings of Nikolskii spaces into $H^s$ spaces \cite{Simon2}.

We are now ready to state our main compactness result. It relies on the following compactness theorem by Simon,
stated in \cite{Simon} as Corollary 5. 
\begin{theorem}{\bf (Corollary 5 \cite{Simon})} \label{Simon}
Assume that $X$, $B$, and $Y$ are Banach spaces,
and $X \subset B \subset Y$ with compact embedding $X \subset \subset B$. 
Let $1 \le p \le \infty$ and $1\le r \le \infty$. Let $F$ be bounded in $L^p(0,T; X) \cap W^{s,r} (0,T; Y)$, 
where $s > 0$ if $r \ge p$, and where $s > 1/r - 1/p$ if $r < p$. 
Then $F$ is relatively compact in $L^p(0,T; B)$ (and in $C(0,T;B)$ if $p = \infty$).
\end{theorem}

\begin{theorem}{\bf (Compactness)}\label{compactness}
Sets $\{{\bf u}_{\Delta t}:\Delta t>0\}$ and $\{{\bf v^*}_{\Delta t}:\Delta t>0\}$ are relatively compact in $L^2(0,T;H^{2s}(\Omega))$, $s<1/2$, and $L^2(0,T;H^{2s}(\Gamma))$, respectively.
\end{theorem}
\proof
The proof of this theorem follows from the estimates obtained in Propositions~\ref{velocity_bounds} and~\ref{AdditionalReg},
and by applying Theorem~\ref{Simon} with $p = r = 2$ and $s > 0$.
\qed

We remark that in contrast with the no-slip condition case studied in \cite{BorSun}, where we obtained partial regularity of $\partial_t \eta$
from the trace of the fluid velocity on the interface and the estimates related to the fluid viscous dissipation, here we had to calculate directly all the time-shifts for $\bdeta_{\Delta t}, {\bf u}_{\Delta t}$,
and ${\bf v}_{\Delta t}$, and ${\bf v}^*_{\Delta t}$ to obtain uniform boundedness in the $H^s$ spaces which, combined with the Simon's 
Corollary 5 and interpolation of classical Sobolev spaces $H^s$ with real exponents $s$, provide compactness. 

The compactness result stated in Theorem~\ref{compactness} implies the following strong convergence results.
\begin{corollary}\label{u_convergence}
We have the following strong convergence results as $\Delta t \to 0$:
\begin{enumerate}
\item ${\bf u}_{\Delta t} \to {\bf u} \quad {\rm in} \ L^2(0,T; H^{2s}(\Omega)), \  s < 1/2,$
\item ${\bf v}^*_{\Delta t} \to {\bf v} \quad {\rm in} \ L^2(0,T; H^{2s}(\Gamma)), \  s < 1/2,$
\item ${\bf v}_{\Delta t} \to {\bf v} \quad {\rm in} \ L^2(0,T; H^{2s}(\Gamma)), \  s < 1/2.$
\end{enumerate}
\end{corollary}

To get  strong convergence results for the structure displacements $\bdeta_{\Delta t}$ we proceed in the same way as in \cite{BorSun}. 
Namely, from Propositions~\ref{EtaBound} and \ref{velocity_bounds}, and from $\partial \tilde{\bdeta}_{\Delta t} = {\bf v}^*_{\Delta t}$, we obtain
that $\tilde\bdeta_{\Delta t}$ is uniformly bounded in $L^\infty(0,T;H_0^2(\Gamma))^2\cap W^{1,\infty}(0,T;L^2(\Gamma))^2$. 
From the continuous embedding 
$$
L^\infty(0,T;H_0^2(\Gamma))^2\cap W^{1,\infty}(0,T;L^2(\Gamma))^2 \hookrightarrow C^{0,1-\alpha}([0,T], H^{2\alpha}(\Gamma)),
\  0 < \alpha < 1,
$$
we obtain uniform boundedness of $\tilde\bdeta_{\Delta t}$ in $C^{0,1-\alpha}([0,T], H^{2\alpha}(\Gamma))$. 
Now, to get compactness in space we recall that $H^{2\alpha}$ is continuously embedded into $H^{2\alpha - \epsilon}$.
By the Arzel\`{a}-Ascoli theorem this embedding is compact. 
In fact, by the application of the Arzel\`{a}-Ascoli theorem to the functions in $C^{0,1-\alpha}([0,T], H^{2\alpha}(\Gamma))$
we get the compactness is time as well. More precisely, we obtain the existence of a subsequence, which we denote by
$\tilde\bdeta_{\Delta t}$ again, such that
$$
\tilde\bdeta_{\Delta t} \to \tilde\bdeta \quad {\rm in }\quad C([0,T];H^{2s}(\Gamma)), \ 0 < s < 1.
$$
Since sequences $\tilde\bdeta_{\Delta t}$ and $\bdeta_{\Delta t}$ have the same limit $\tilde\bdeta = \bdeta \in C([0,T];H^{2s}(\Gamma))$,
where $\bdeta$ is the weak* limit discussed in Lemma~\ref{weak_convergence}, 
we obtain 
$$
\tilde\bdeta_{\Delta t} \to \bdeta \quad {\rm in }\quad C([0,T];H^{2s}(\Gamma)), \ 0 < s < 1.
$$
By combining this statement with the continuity in time of $\bdeta$, as was done in Lemma 3 of \cite{BorSun}, 
we obtain the following strong convergence results for the structure:

\begin{theorem}\label{ConvEta}
The following strong convergence results hold as $\Delta t \to 0$:
\begin{enumerate}
\item $\displaystyle{\bdeta_{\Delta t}\to\bdeta\;{\rm in}\; L^{\infty}(0,T;H^{2s}(\Gamma))}$, $s < 1$,
\item $\displaystyle{T_{\Delta t}\bdeta_{\Delta t}\to\bdeta\;{\rm in}\; L^{\infty}(0,T;H^{2s}(\Gamma))}$, $s < 1$.
\end{enumerate}
\end{theorem}
To pass to the limit in the weak formulation, we need uniform convergence of $\bdeta_{\Delta t}$. This is where
the fact that we work in 2D rather than in 3D comes into play. 
Namely, since in our 2D problem $\Gamma$ is a 1D domain, we have that $H^\alpha(\Gamma)$ is embedded into $C^1(\overline\Gamma)$
for $s > 3/2$, and so the first statement of Theorem \ref{ConvEta} implies $\displaystyle{\bdeta_{\Delta t}\to\bdeta\;{\rm in}\; L^{\infty}(0,T;C^1(\overline{\Gamma}))}$. Moreover, we have the following Corollary.

\begin{corollary}\label{ConvEtaC1}
The following uniform convergence results hold as $\Delta t \to 0$:
\begin{enumerate}
\item $\displaystyle{\bdeta_{\Delta t}\to\bdeta\;{\rm in}\; L^{\infty}(0,T;C^1(\overline{\Gamma}))}$,
\item $\displaystyle{T_{\Delta t}\bdeta_{\Delta t}\to\bdeta\;{\rm in}\; L^{\infty}(0,T;C^1(\overline{\Gamma}))}$.
\end{enumerate}
\end{corollary}

The second statement in this corollary can be proved using the same
arguments as those following statement (76) in \cite{BorSun}.

By using this corollary, and the explicit formulas for the normals $\bnu_{\Delta t}$, the tangents $\btau_{\Delta t}$,
and the quantities associated with the ALE mappings ${\bf A}_{\Delta t}$, one can see that the following strong 
convergence results hold:

\begin{corollary}\label{StrongConvergences}
The following strong convergence results hold for the geometric quantities associated with the change of the 
fluid domain $\Omega(t)$:
\begin{enumerate}
\item $\bnu_{\Delta t}\to \bnu^{\eta}\;{\rm in}\; L^{\infty}(0,T;C(\overline{\Gamma}))$,
\item $\btau_{\Delta t}\to \btau^{\eta}\;{\rm in}\; L^{\infty}(0,T;C(\overline{\Gamma}))$,
\item ${\bf w}_{\Delta t}\to {\bf w}^{\eta}\;{\rm in}\; L^{2}(0,T;H^1(\Omega))$,
\item $S_{\Delta t}\to S^{\eta}\;{\rm in}\; L^{\infty}(0,T;C(\overline{\Gamma}))$,
\item $J_{\Delta t}\to J^{\eta}\;{\rm in}\; L^{\infty}(0,T;C(\overline{\Omega}))$,
\item $(\nabla{\bf A}_{\Delta t})^{-1}\to (\nabla{\bf A}_{\eta})^{-1}\;{\rm in}\; L^{\infty}(0,T;C(\overline{\Omega}))$.
\end{enumerate}
\end{corollary}

We remark that the results of Corollary \ref{StrongConvergences} were not necessary in our previous work \cite{BorSun}
because only the normal component of structure displacement was assumed to be non-zero. In the current manuscript both the normal
and tangential structure displacements are considered to be non-zero, which introduces additional complications in tracking
the change in the measure of the interface ``surface'' deformation
that we did not have to deal with before.

\section{The limiting problem}\label{Sec:limiting}
\subsection{Construction of suitable test functions}\label{sec:test_functions}
Now that we have the strong convergence results above, we are ready to show that the limits, as $\Delta t \to 0$, of approximate solutions
satisfy the weak form \eqref{WeakALEFormula} of problem \eqref{NSP}-\eqref{ProblemIC}. Unfortunately, due to the fact that we mapped
our problem defined on the moving domain $\Omega(t)$ onto a fixed, reference domain $\Omega$,
introduces additional difficulties. More precisely, the velocity test functions in the weak formulation of the fluid sub-problem  (\ref{WeakFluid}) 
now depend of $\Delta t$ via their dependence on $\eta^n_{\Delta t}$. This is because of the requirement that
the transformed divergence-free condition $\nabla^{\eta^n_{\Delta t}}\cdot{\bf q} = 0$ must
be satisfied. 
Passing to the limit in the weak formulation of the fluid sub-problem  (\ref{WeakFluid}) when both the test functions and the 
unknown functions depend on $\Delta t$ is tricky, and special care needs to be taken to deal with this issue.

Our strategy is to restrict ourselves to a dense subset, call it ${\mathcal X}^{\eta}(0,T)$, of the space of all test functions ${\mathcal Q}^{\eta}(0,T)$, and for every ${\bf q}\in{\mathcal X}^{\eta}(0,T)$ construct a sequence of test functions for the approximate problems, call them ${\bf q}_{\Delta t}$, such that ${\bf q}_{\Delta t}\to {\bf q}$ in suitable norms. This approach was used in \cite{BorSun} for the FSI problem with the no-slip condition and only radial structural displacements, see also \cite{CDEM,muha2013nonlinear,SunBorMulti}. Since, here the test space is different because of the slip boundary condition, the construction of such test functions is somewhat different.

First let us define the domain which contains all the approximate domains
\begin{equation}\label{Omega_max}
\displaystyle{\Omega_{\rm max}=\bigcup_{\Delta t>0,n\in\N}\Omega^{\eta^n_{\Delta t}}}.
\end{equation}
Notice that Corollary~\ref{ConvEtaC1} implies $\Omega^{\eta}(t)\subset\Omega_{\rm max}$, $t\in [0,T]$, and $\Sigma\subset\partial\Omega_{\rm max}$. We define
\begin{align}
{\mathcal X}_{\rm max} = & \{{\bf r}\in C^1_c([0,T);C^2(\overline{\Omega}_{\rm max}))^2:\nabla\cdot{\bf r}=0,
{\bf r}\cdot\btau=0,\; {\rm on}\ \Gamma_i,\;i\in I,\;{\bf r}=0\ {\rm on}\ \Gamma_i,\;i\in II,\;  
\nonumber \\
& \ \ {\bf r}\cdot\bnu=0,\; {\rm on}\ \Gamma_i,\;i\in III\cup IV \}.
\nonumber 
\\
{\mathcal X}^{\eta}(0,T)= & \{({\bf q},\bpsi):{\bf q}(t,.)={\bf r}(t,.)_{|\Omega^{\eta}(t)}\circ{\bf A}_{\eta}(t),{\bf r}\in{\mathcal X}_{\rm max},\;({\bf r}_{|\Gamma^{\eta}}-\bpsi)\cdot\bnu^{\eta}=0,\;\bpsi\in H^2_0(\Gamma)\}.
\nonumber 
\end{align}
From the construction it is immediate that ${\mathcal X}^{\eta}(0,T)$ is dense in ${\mathcal Q}^{\eta}(0,T)$. 

We now want to construct the test functions ${\bf q}_{\Delta t}$ and ${\bpsi}_{\Delta t}$ for the approximate problems, such that 
${\bf q}_{\Delta t} \to {\bf q}$ and ${\bpsi}_{\Delta t} \to \bpsi$ in a suitable space. 
For this purpose let us fix $({\bf q},\bpsi)\in{\mathcal X}^{\eta}(0,T)$, ${\bf q}(t,.)={\bf r}(t,.)_{|\Omega^{\eta}(t)}\circ{\bf A}_{\eta}(t),\;{\bf r}\in{\mathcal X}_{\rm max}$, and define $({\bf q}_{\Delta t},{\bpsi}_{\Delta t})$ to be piece-wise constant in time so that:
\begin{equation}\label{DefTestF}
\left . 
\begin{array}{c}
{\bf q}_{\Delta t}(t,.)={\bf q}^n_{\Delta t}:={\bf r}(n\Delta t,.)_{|\Omega^{\eta_{\Delta t}}(t)}\circ{\bf A}^n_{\Delta t}(t),
\\ \\
\bpsi_{\Delta t}(t)=\bpsi^n_{\Delta t}:= \bpsi(n\Delta t),
\end{array}\right \},\; t\in ((n-1)\Delta t,n\Delta t].
\end{equation}
Note that $({\bf q}_{\Delta t}(t,.),\bpsi_{\Delta t}(t,.))\in{\mathcal V}^n_{\Delta t}$, $t\in ((n-1)\Delta t,n\Delta t]$. Now, using ideas from \cite{BorSun} we can prove the following lemma.
\begin{lemma}\label{TestFunctionConv}
For every $({\bf q},\bpsi)\in{\mathcal X}^{\eta}(0,T)$ we have
$$
({\bf q}_{\Delta t},\bpsi_{\Delta t})\to ({\bf q},\bpsi)\;{\rm in}\; L^{\infty}(0,T;C^1(\overline{\Omega}))^2\times L^{\infty}(0,T;C^1(\overline{\Gamma}))^2.
$$
\end{lemma}
We will also need information about the convergence of approximations of $\partial_t{\bf q}$, which we define by:
\begin{equation}\label{dq}
d{\bf q}_{\Delta t}(t,.):=\frac{{\bf q}^{n+2}_{\Delta t}-{\bf q}^{n+1}_{\Delta t}}{\Delta t},\quad t\in((n-1)\Delta t,n\Delta t].
\end{equation}
\begin{lemma}\label{ApproxDer}
Let $({\bf q},\bpsi)\in{\mathcal X}^{\eta}(0,T)$, and let $d{\bf q}_{\Delta t}$ be defined by \eqref{dq}.
Then $d{\bf q}_{\Delta t}\to\partial_t{\bf q}$ in $L^2(0,T;L^2(\Omega))$.
\end{lemma}
\proof
By the Mean-Value Theorem we have
\begin{align}
\frac{{\bf q}^{n+2}_{\Delta t}-{\bf q}^{n+1}_{\Delta t}}{\Delta t}&=\frac{1}{\Delta t}\Big ({\bf r}\big ((n+2)\Delta t,{\bf A}^{n+2}_{\Delta t}(r,z)\big )-\big ((n+1)\Delta t,{\bf A}^{n+1}_{\Delta t}(r,z)\big )\Big )
\nonumber
\\
&=\frac{1}{\Delta t}\Big ({\bf r}\big ((n+2)\Delta t,{\bf A}^{n+2}_{\Delta t}(r,z)\big )-{\bf r}\big ((n+1)\Delta t,{\bf A}^{n+2}_{\Delta t}(r,z)\big )+
\nonumber
\\
&\qquad \quad \ {\bf r}\big ((n+1)\Delta t,{\bf A}^{n+2}_{\Delta t}(r,z)\big )-\big ((n+1)\Delta t,{\bf A}^{n+1}_{\Delta t}(r,z)\big )\Big )
\nonumber
\\
&=\partial_t{\bf r}\big ((n+1+\beta)\Delta t,{\bf A}^{n+2}_{\Delta t}(r,z)\big ))+\nabla{\bf r}\big ((n+1)\Delta t,\boldsymbol\zeta\big )\frac{{\bf A}^{n+2}_{\Delta t}(r,z)-{\bf A}^{n+1}_{\Delta t}(r,z)}{\Delta t},
\nonumber
\end{align}
where $\boldsymbol\zeta={\bf A}^{n+1}_{\Delta t}(r,z)+\gamma({\bf A}^{n+2}_{\Delta t}(r,z)-{\bf A}^{n+1}_{\Delta t}(r,z))$, $\beta,\gamma\in [0,1]$. Notice that $\displaystyle{\frac{{\bf A}^{n+2}_{\Delta t}-{\bf A}^{n+1}_{\Delta t}}{\Delta t}={\bf w}^{n+2}_{\Delta t}}$, 
and this term is associated with $T_{-\Delta t}{\bf w}_{\Delta t}$, which converges strongly to ${\bf w}^{\eta}$ in $L^2(0,T;L^2(\Omega))$.
Therefore we have
$$
d{\bf q}_{\Delta t}\to \partial_t {\bf r}+\nabla{\bf r}\cdot{\bf w}^{\eta}=\partial_t{\bf q}\; {\rm in}\; L^2(0,T;L^2(\Omega)).
$$
\qed
\subsection{Approximate equations}
We have so far introduced the weak formulation of the coupled problem at the continuous level, and have split the coupled problem
into the fluid and structure sub-problems. We then semi-discretized the two sub-problems, and introduced the semi-discrete weak formulations
of those approximate fluid and structure sub-problems. 
What needs to be done next is to define the semi-discrete weak formulation of the  {\sl coupled problem} which approximates
the weak formulation of the coupled continuous problem. 
To do that we take the constructed approximate test functions $({\bf q}_{\Delta t},\bpsi_{\Delta t})$, multiply them by $\Delta t$,
and replace the test functions ${\bf q}$ and $\bpsi$ in the weak formulations
for the approximate structure and fluid sub-problems ~\eqref{StructureWeak} and~\eqref{WeakFluid}
with the test functions $\Delta t({\bf q}_{\Delta t},\bpsi_{\Delta t})$.
We add the two weak formulations together, and sum w.r.t. $n=0,\dots,N-1$. 
The approximating solutions $({\bf u}_{\Delta t},\bdeta_{\Delta t})$ satisfy the following variational form of the 
semi-discretized (approximate) coupled problem:

\begin{equation}\label{ApproxEq}
\begin{array}{c}
\displaystyle{\rho_F\int_0^T\int_{\Omega}T_{\Delta t}J_{\Delta t}\partial_t\tilde{\bf u}_{\Delta t}\cdot{\bf q}_{\Delta t}
+\frac{\rho_F}{2}\int_0^T\int_{\Omega}\frac{J_{\Delta t}-T_{\Delta t}J_{\Delta t}}{\Delta t}{\bf u}_{\Delta t}\cdot{\bf q}_{\Delta t}}
\\ \\
+\displaystyle{\frac{\rho_F}{2}\int_0^T\int_{\Omega}J_{\Delta t}\Big (((T_{\Delta t}{\bf u}_{\Delta t}-{\bf w}_{\Delta t})\cdot\nabla^{\eta_{\Delta t}}){\bf u}_{\Delta t}\cdot {\bf q}_{\Delta t}}
\displaystyle{-((T_{\Delta t}{\bf u}_{\Delta t}-{\bf w}_{\Delta t})\cdot\nabla^{\eta_{\Delta t}}){\bf q}_{\Delta t}\cdot {\bf u}_{\Delta t}\Big)}
\\ \\
\displaystyle{+2\mu\int_0^T\int_{\Omega}J_{\Delta t}{\bf D}^{\eta_{\Delta t}}({\bf u}_{\Delta t}):{\bf D}^{\eta_{\Delta t}}({\bf q}_{\Delta t})+\frac{1}{\alpha}\int_0^T\int_{\Gamma}(({\bf u}_{\Delta t}-{\bf v}_{\Delta t})\cdot\btau_{\Delta t})({\bf q}_{\Delta t}\cdot\btau_{\Delta t})S_{\Delta t}dz}
\\ \\
\displaystyle{+\frac{1}{\alpha}\int_0^T\int_{\Gamma}(({\bf v}_{\Delta t}-{\bf u}_{\Delta t})\cdot\btau_{\Delta t})({\bpsi}_{\Delta t}\cdot\btau_{\Delta t})S_{\Delta t}dz }
\displaystyle{+\sum_{i\in III}\int_{\Gamma_i}\frac{1}{\alpha_i}({\bf u}_{\Delta t}\cdot\btau)({\bf q}_{\Delta t}\cdot\btau)}
\\ \\
+\displaystyle{\rho_S h\int_0^T\int_{\Gamma} \partial_t\tilde{\bf v}_{\Delta t}\bpsi_{\Delta t}}
+\langle {\cal L}_e \bdeta_{\Delta t},\bpsi_{\Delta t}\rangle
=\langle {\bf R}_{\Delta t},{\bf q}_{\Delta t}\rangle,\quad ({\bf q},\bpsi)\in{\mathcal X}^\eta(0,T).
\end{array}
\end{equation}

We want to pass to the limit as $\Delta t \to 0$ and show that the limiting functions satisfy the weak formulation of problem \eqref{NSP}-\eqref{ProblemIC}, given 
in \eqref{WeakALEFormula}. Indeed, by using the convergence results for the approximate solutions given by 
Theorem~\ref{compactness}, and Corollaries~\ref{ConvEtaC1} and \ref{StrongConvergences}, and by using the convergence results
for the corresponding test functions, given by Lemma~\ref{TestFunctionConv},  we can pass to the limit directly in all the terms, expect the ones associated with the geometric conservation law of the ALE mapping, i.e. the first two terms in the first line in~\eqref{ApproxEq}.

\subsection{Discrete v.s. continuous geometric conservation law}\label{Sec:GSL}
We show that the terms associated with the semi-discrete approximation of the geometric conservation law associated with our family of ALE mappings
$A_{\Delta t}^n$, 
converge, as $\Delta t \to 0$, to the corresponding terms 
associates with the geometric conservation law satisfied by the ALE mapping $A^\eta$
appearing in the continuous weak formulation \eqref{WeakALEFormula}. More precisely, we have the following result.

\begin{proposition}\label{ConvJDer}
For every $({\bf q},\bpsi)\in {\mathcal X}^{\eta}(0,T)$ the following convergence result holds
$$
\rho_F\int_0^T\int_{\Omega}T_{\Delta t}J_{\Delta t}\partial_t\tilde{\bf u}_{\Delta t}\cdot{\bf q}_{\Delta t}+\frac{\rho_F}{2}\int_0^T\int_{\Omega}\frac{J_{\Delta t}-T_{\Delta t}J_{\Delta t}}{\Delta t}{\bf u}_{\Delta t}\cdot{\bf q}_{\Delta t}
\to
$$
$$
- \rho_F\int_0^T\int_{\Omega}J^{\eta}{\bf u}\cdot\partial_t{\bf q}
-\frac{\rho_F}{2}\int_0^T \int_{\Omega}J^{\eta}(\nabla^{\eta}\cdot{\bf w}^{\eta}){\bf u}\cdot{\bf q}-\int_{\Omega}J_0{\bf u}_0{\bf q}(0,.).
$$
\end{proposition}

\proof

Let us first consider the term that contains the fluid acceleration $\partial_t\tilde{\bf u}_{\Delta t}$. 
We use the definition of approximate solutions and test functions, and the summation by parts formula to obtain
$$
\int_0^T\int_{\Omega}T_{\Delta t}J_{\Delta t}\partial_t\tilde{\bf u}_{\Delta t}\cdot{\bf q}_{\Delta t}=\sum_{n=0}^{N-1}\int_{\Omega}J^n_{\Delta t}\big ({\bf u}_{\Delta t}^{n+1}-{\bf u}^n_{\Delta t}\big )\cdot{\bf q}_{\Delta t}^{n+1}=\int_{\Omega}J^{N}_{\Delta t}{\bf u}^{N}_{\Delta t}\cdot{\bf q}^{N+1}_{\Delta t}
$$
$$
-\int_{\Omega}J^{0}_{\Delta t}{\bf u}^{0}_{\Delta t}\cdot{\bf q}^{1}_{\Delta t}
-\sum_{n=0}^{N-1}\int_{\Omega}{\bf u}^{n+1}_{\Delta t}\cdot\big (J^{n+1}_{\Delta t}{\bf q}^{n+2}_{\Delta t}-J^{n}_{\Delta t}{\bf q}^{n+1}_{\Delta t}\big ).
$$
Notice that by construction we have ${\bf q}_{\Delta t}^n=0$, $n\geq N$. 
By adding and subtracting ${\bf q}^{n+1}_{\Delta t}J^{n+1}_{\Delta t}$ we can write
$$
\sum_{n=0}^{N-1}\int_{\Omega}{\bf u}^{n+1}_{\Delta t}\cdot\big (J^{n+1}_{\Delta t}{\bf q}^{n+2}_{\Delta t}-J^{n}_{\Delta t}{\bf q}^{n+1}_{\Delta t}\big )=\sum_{n=0}^{N-1}\int_{\Omega}{\bf u}^{n+1}_{\Delta t}\cdot{\bf q}^{n+1}_{\Delta t}\big (J^{n+1}_{\Delta t}-J^n_{\Delta t}\big )
+\sum_{n=0}^{N-1}\int_{\Omega}{\bf u}^{n+1}_{\Delta t}\cdot\big ({\bf q}^{n+2}_{\Delta t}-{\bf q}^{n+1}_{\Delta t})J^{n+1}_{\Delta t}.
$$
By plugging this calculation back into the above formula we get
$$
\int_0^T\int_{\Omega}T_{\Delta t}J_{\Delta t}\partial_t\tilde{\bf u}_{\Delta t}\cdot{\bf q}_{\Delta t}=
-\int_{\Omega}J^{0}_{\Delta t}{\bf u}^{0}_{\Delta t}\cdot{\bf q}^{1}_{\Delta t}
-\int_0^T\int_{\Omega} \frac{J_{\Delta t}-T_{\Delta t}J_{\Delta t}}{\Delta t}{\bf u}_{\Delta t}\cdot{\bf q}_{\Delta t}
-\int_0^T\int_{\Omega}T_{\Delta t}J_{\Delta t}{\bf u}_{\Delta t}\cdot d{\bf q}_{\Delta t}.
$$
Therefore, the two terms on the left hand side in the statement of Proposition~\ref{ConvJDer} are equal to
\begin{align}
&\rho_F\int_0^T\int_{\Omega}T_{\Delta t}J_{\Delta t}\partial_t\tilde{\bf u}_{\Delta t}\cdot{\bf q}_{\Delta t}+\frac{\rho_F}{2}\int_0^T\int_{\Omega}\frac{J_{\Delta t}-T_{\Delta t}J_{\Delta t}}{\Delta t}{\bf u}_{\Delta t}\cdot{\bf q}_{\Delta t}
\nonumber
\\
=-&\rho_F\int_0^T \int_{\Omega} T_{\Delta t}J_{\Delta t}{\bf u}_{\Delta t}\cdot d{\bf q}_{\Delta t}-\rho_F\int_{\Omega}J_0{\bf u}_0\cdot{\bf q}^1_{\Delta t}-\frac{\rho_F}{2}\int_0^T\int_{\Omega}\frac{J_{\Delta t}-T_{\Delta t}J_{\Delta t}}{\Delta t}{\bf u}_{\Delta t}\cdot{\bf q}_{\Delta t}.
\label{pre-limit}
\end{align}
We can pass to the limit in the first two terms on the right by using Lemma~\ref{ApproxDer}, Corollary~\ref{StrongConvergences},
and Corollary~\ref{u_convergence}. 
Passing to the limit in the third term on the right is not that straight forward. 
We want to show that 
$$
\int_0^{T}\int_{\Omega}\frac{J_{\Delta t}-T_{\Delta t}J_{\Delta t}}{\Delta t}{\bf u}_{\Delta t}\cdot{\bf q}\to
\int_0^T \int_{\Omega}J^{\eta}(\nabla^{\eta}\cdot{\bf w}^{\eta}){\bf u}\cdot{\bf q}.
$$
The main problem is passing to the limit in the term 
$({J_{\Delta t}-T_{\Delta t}J_{\Delta t}})/{\Delta t}$,
since it is not clear that it converges 
to $J^\eta_t$, which, as stated in \eqref{J_t}, is equal to 
$\partial_t J^\eta = J^\eta (\nabla^\eta \cdot {\bf w}^\eta)$.

The plan is to explicitly calculate 
 $J_{\Delta t}-T_{\Delta t}J_{\Delta t}$ by using the mean value theorem.
 For this purpose we first recall that the integral with respect to time of  $J_{\Delta t}-T_{\Delta t}J_{\Delta t}$
 is equal to the sum over $n=1,\dots,N$ of the differences $J^{n+1}_{\Delta t}-J^{n}_{\Delta t}$ times $\Delta t$.
Now, since 
$$
J^n_{\Delta t}=\det \nabla {\bf A}_{\Delta t}^n>0,\quad\Delta t>0,n=1,\dots N,
$$
we have
$$
J^{n+1}_{\Delta t}-J^{n}_{\Delta t} = \det \nabla {\bf A}_{\Delta t}^{n+1} - \det \nabla {\bf A}_{\Delta t}^n.
$$
We apply the mean value theorem on the determinant function when the difference 
$\nabla {\bf A}_{\Delta t}^{n+1} -  \nabla {\bf A}_{\Delta t}^n$ is small.
The mean value theorem says
\begin{equation}\label{MVT}
\det \nabla {\bf A}_{\Delta t}^{n+1} - \det \nabla {\bf A}_{\Delta t}^n = D(\det)(\nabla {\bf A}_{\Delta t}^{n,\beta}) (\nabla {\bf A}_{\Delta t}^{n+1} -  \nabla {\bf A}_{\Delta t}^n),
\end{equation}
where $ D(\det)(\nabla {\bf A}_{\Delta t}^{n,\beta})$ denotes the derivative of the determinant function evaluated at an intermediate point

\if 1 = 0
associated with the structure 
$$
\bdeta^{n,\beta} = \bdeta^n + \beta(\bdeta^{n+1} - \bdeta^n),\ {\rm for\ some\ } \beta\in[0,1].
$$
\fi

$$
\nabla {\bf A}_{\Delta t}^{n,\beta} := \nabla {\bf A}_{\Delta t}^n + \beta ( \nabla {\bf A}_{\Delta t}^{n+1} - \nabla {\bf A}_{\Delta t}^n), \quad
{\rm for\ some\ } \beta \in [0,1].
$$
The functional $ D(\det)(\nabla {\bf A}_{\Delta t}^{n,\beta})$ acts on the difference $(\nabla {\bf A}_{\Delta t}^{n+1} -  \nabla {\bf A}_{\Delta t}^n)$.

To explicitly calculate the right hand-side of \eqref{MVT}, we use the formula for the derivative of the determinant, evaluated at ${\bf F}$,
acting on ${\bf U}$, given by
$$
D(\det)({\bf F}){\bf U}=(\det {\bf F}){\rm tr}({\bf U} {\bf F}^{-1}).
$$
By using this formula we get
$$
J^{n+1}_{\Delta t}-J^{n}_{\Delta t}=\det\big (\nabla {\bf A}_{\Delta t}^{n,\beta}\big ){\rm tr}\big (({\nabla {\bf A}_{\Delta t}}^{n+1}-{\nabla {\bf A}_{\Delta t}}^n)
({\nabla {\bf A}^{n,\beta}_{\Delta t}})^{-1}\big ),\  {\rm for\  some\ } \beta \in [0,1].
$$ 
Now, from \eqref{nablaeta} we have that the factor on the right hand side containing the trace is equal to
$$
{\rm tr} \left( \nabla \left(\frac{{\bf A}^{n+1}_{\Delta t} - {\bf A}^n_{\Delta t}}{\Delta t}\right)  (\nabla {\bf A}^{n,\beta}_{\Delta t})^{-1} \right) =
\nabla^{n,\beta} \cdot \frac{{\bf A}^{n+1}_{\Delta t} - {\bf A}^n_{\Delta t}}{\Delta t} = \nabla^{n\beta} \cdot {\bf w}^{n+1}_{\Delta t},
$$
where $\nabla^{n\beta} = \nabla^{\eta^n} + \beta\left(\nabla^{\eta^{n+1}} - \nabla^{\eta^n}\right)$,
and by denoting $\det \nabla {\bf A}_{\Delta t}^{n,\beta}$ with $J^{n,\beta}_{\Delta t}$ in the spirit of \eqref{ALEVelDiscrete}, we get 
$$
\frac{J^{n+1}_{\Delta t}-J^{n}_{\Delta t}}{\Delta t}=J^{n,\beta}_{\Delta t} \left(\nabla^{n\beta} \cdot {\bf w}^{n+1}_{\Delta t}\right).
$$
Thus, we have just calculated that
\begin{equation}\label{GeomConv2}
\displaystyle{\int_{\Omega}\frac{J^{n+1}_{\Delta t}-J^n_{\Delta t}}{\Delta t}{\bf u}^{n+1}_{\Delta t}\cdot{\bf q}^{n+1}_{\Delta t}
=\int_{\Omega}J^{n,\beta}\big (\nabla^{\eta^{n,\beta}}\cdot{\bf w}^{n+1}_{\Delta t}\big ){\bf u}^{n+1}_{\Delta t}\cdot{\bf q}^{n+1}_{\Delta t}}.
\end{equation}
By taking the sum over $n=0,\dots,N-1$ in equation~\eqref{GeomConv2} we get
\begin{equation}\label{GeomConv3}
\displaystyle{\int_0^{T-\Delta t}\int_{\Omega}\frac{J_{\Delta t}-T_{\Delta t}J_{\Delta t}}{\Delta t}{\bf u}_{\Delta t}\cdot{\bf q}_{\Delta t}
=\int_0^{T-\Delta t}\int_{\Omega}J^{\beta}_{\Delta t}\big (\nabla^{\eta_{\Delta t}^{\beta}}\cdot{\bf w}_{\Delta t}\big ){\bf u}_{\Delta t} \cdot{\bf q}_{\Delta t}}.
\end{equation}
Now we can pass to the limit as $\Delta t \to 0$ to obtain:
\begin{equation}\label{GeomConv4}
\displaystyle{\int_0^{T}\int_{\Omega}\frac{J_{\Delta t}-T_{\Delta t}J_{\Delta t}}{\Delta t}{\bf u}_{\Delta t}\cdot{\bf q}
\to
\int_0^T \int_{\Omega}J^{\eta}(\nabla^{\eta}\cdot{\bf w}^{\eta}){\bf u}\cdot{\bf q}}.
\end{equation}

Finally, with this conclusion we can pass to the limit in \eqref{pre-limit} to obtain 
$$
\rho_F\int_0^T\int_{\Omega}T_{\Delta t}J_{\Delta t}\partial_t\tilde{\bf u}_{\Delta t}\cdot{\bf q}_{\Delta t}+\frac{\rho_F}{2}\int_0^T\int_{\Omega}\frac{J_{\Delta t}-T_{\Delta t}J_{\Delta t}}{\Delta t}{\bf u}_{\Delta t}\cdot{\bf q}_{\Delta t}
\to
$$
$$
- \rho_F\int_0^T\int_{\Omega}J^{\eta}{\bf u}\cdot\partial_t{\bf q}
-\int_{\Omega}J_0{\bf u}_0{\bf q}(0,.)
-\frac{\rho_F}{2}\int_0^T \int_{\Omega}J^{\eta}(\nabla^{\eta}\cdot{\bf w}^{\eta}){\bf u}\cdot{\bf q},
$$
which is exactly the statement of the Proposition.
\qed

\if 1 = 0

First first we introduce the following notation:
\begin{equation}\label{Notation1}
{\bf F}^{n}_{\Delta t}=\nabla {\bf A}_{\Delta t}^n.
\end{equation}
Therefore we perform the calculation similar to the one in derivation of the weak formulation in the continuous case. We use the following formula for derivative of the determinant (see e.g. \cite{Gur}, p. 23):
$$
D(\det)(A)U=(\det A){\rm tr}(UA^{-1}).
$$
For the moment, let us fix $\Delta t>0$ and for the simplicity of notation omit subscript $\Delta t$ in the following computations.
Let us define ${\bf F}^{n,\beta}={\bf F}^n+\beta({\bf F}^{n+1}-{\bf F}^n)$, $0\leq\beta\leq 1$. Now by the mean value theorem we have:
$$
J^{n+1}-J^{n}=\det\big ({\bf F}^{n,\beta}\big ){\rm tr}\big (({\bf F}^{n+1}-{\bf F}^n)({\bf F}^{n,\beta})^{-1}\big ),
$$
for some $\beta \in [0,1]$. Notice that ${\bf A}^{\eta^n+\beta(\eta^{n+1}-\eta^n)}$ is again a good ALE mapping and therefore we have:
\begin{equation}\label{GeomConv2}
\displaystyle{\frac{1}{2}\int_{\Omega}\frac{J^{n+1}-J^n}{\Delta t}{\bf u}^{n+1}\cdot{\bf q}^{n+1}=\int_{\Omega}J^{n,\beta}\big (\nabla^{\eta^{n,\beta}}\cdot{\bf w}^{n+1}\big )\cdot{\bf q}^{n+1}}
\end{equation}

**********

TODO: $\det\big ({\bf F}^{n,\beta}\big )$ becomes $J^{n,\beta}$, $({\bf F}^{n,\beta})^{-1}$ becomes $\nabla^{\eta^{n,\beta}}$,
and $({\bf F}^{n+1}-{\bf F}^n)$ divided by $\Delta t$ will be related to ${\bf w}^{n+1}$.

***************

Now we can sum equation~\eqref{GeomConv2} for $n=0,\dots,N-1$ an get
\begin{equation}\label{GeomConv3}
\displaystyle{\frac{1}{2}\int_0^{T-\Delta t}\int_{\Omega}\frac{J_{\Delta t}-T_{\Delta t}J_{\Delta t}}{\Delta t}{\bf u}_{\Delta t}\cdot{\bf q}_{\Delta t}=\frac{1}{2}\int_0^{T-\Delta t}\int_{\Omega}J^{\beta}_{\Delta t}\big (\nabla^{\eta_{\Delta t}^{\beta}}\cdot{\bf w}_{\Delta t}\big )\cdot{\bf q}_{\Delta t}}
\end{equation}
Therefore we can pass to the limit, i.e. we have:
\begin{equation}\label{GeomConv4}
\displaystyle{\frac{1}{2}\int_0^{T}\int_{\Omega}\frac{J_{\Delta t}-T_{\Delta t}J_{\Delta t}}{\Delta t}{\bf u}_{\Delta t}\cdot{\bf q}\to\frac{1}{2}\int_{\Omega}J^{\eta}(\nabla^{\eta}\cdot{\bf w}^{\eta}){\bf u}\cdot{\bf q}}
\end{equation}
\qed

\fi 

Therefore, we have shown that in the limit as $\Delta t \to 0$, the approximate solutions constructed in Section~\ref{Sec:Approx}
based on the Lie operator splitting scheme, converge to a weak solution of problem \eqref{NSP}-\eqref{ProblemIC}.
More precisely, we have shown the following result.

\begin{lemma}\label{final_convergence}
There exists a $T > 0$, and a subsequence of approximate solutions $({\bf u}_{\Delta t}, \bdeta_{\Delta t})$, constructed in Section~\ref{Sec:Approx},
such that $({\bf u}_{\Delta t}, \bdeta_{\Delta t})$ converges, as $\Delta t \to 0$, 
to a function $({\bf u},\bdeta) \in {\cal{W}}^\eta(0,T)$, 
which is a weak solution to problem \eqref{NSP}-\eqref{ProblemIC} in the sense of Definition~\ref{WeakALE}. 
The weak form in Definition~\ref{WeakALE} 
 holds for all the test functions $({\bf q},{\bpsi}) \in {\cal{X}}^\eta(0,T)$ which are dense in ${\cal{Q}}^\eta(0,T)$,
 and are obtained as the limits of the test functions $({\bf q}_{\Delta t}, \bdeta_{\Delta t}) \in {\cal{V}}_{\Delta t}$ 
 constructed in Section~\ref{sec:test_functions}.
\end{lemma}

We are now ready to complete the proof of the main existence result, stated in Theorem~\ref{MainResult}.
From Lemma~\ref{final_convergence} we obtain the existence of a weak solution defined on the time interval $(0,T)$, 
where $T>0$ is determined by \eqref{Tdescrete}. To obtain the energy estimate \eqref{EI1} from Theorem~\ref{MainResult} 
we consider discrete energy inequalities stated in points 1. and 2. in Proposition~\ref{Stability} and let $\Delta t \to 0$.
Due to the lower semi-continuity property of norms, we can take the limit in points 1. and 2. in Proposition~\ref{Stability}
to recover the energy estimate \eqref{EI1}.

This concludes the constructive proof to the main existence result, stated in Theorem~\ref{MainResult}.
\vskip 0.1in

\section{Conclusions}
This paper provides a constructive existence proof for a weak solution to a nonlinear moving boundary problem between an incompressible, viscous fluid and 
an elastic shell, with the Navier slip boundary condition holding at the fluid-structure interface. 
Due to different types of boundary conditions holding at each piece of the fluid domain boundary the usual
vorticity formulation, commonly used in a good agreement with the Navier slip boundary condition, 
does not appear helpful for this problem.
The problem is motivated by studying fluid-structure interaction between blood flow and cardiovascular tissue, 
whether natural or bio-artificial, which include cell-seeded tissue constructs, 
which consists of grooves in tissue scaffolds that are lined with cells giving rise to ``rough'' fluid-structure interfaces.
To filter out the small scales of the rough fluid domain boundary, 
effective boundary conditions based on the Navier slip condition have been used 
in various applications, see a review paper by Mikeli\'{c} \cite{MikelicReview} and the references therein. 
The present work is the {\em  first existence result involving the 
Navier slip boundary condition for a fluid-structure interaction problem
with elastic structures}. 
Dealing with the slip condition introduces 
several mathematical difficulties. The main one is associated with the fact that the fluid viscous dissipation can no-longer be used as a regularizing
mechanism for the tangential component of velocity of the fluid-structure interface, as is the case with the no-slip condition, where the regularity of
the fluid-structure interface is directly influenced by the fluid viscosity through the trace of the fluid velocity at the interface. 
As a result, new compactness arguments had to be used in the existence proof, 
which are based on Simon's characterization of compactness in $L^2(0,T;B)$ spaces \cite{Simon},  
and on interpolation of the classical Sobolev spaces with real exponents $H^s$ (or alternatively Nikolskii spaces $N^{s,p} \cite{Simon2})$. Furthermore, to deal with the non-zero longitudinal displacement and keep the behavior of fluid-structure interface ``under control'', 
we had to consider higher-order terms in the structure model
given by the bending rigidity of {\em shells}. 
The linearly elastic membrane model was not tractable. 
Due to the non-zero longitudinal displacement additional nonlinearities appear in the problem
that track the  geometric quantities such as surface measure, the interface tangent and normal, and the Jacobian of the ALE mapping,
 which are now included explicitly in the weak formulation of the problem, and cause various difficulties 
 in the analysis. This is one of the reasons why our existence result is local in time, i.e., it holds for the time interval $(0,T)$ for
 which we can guarantee that the fluid domain will not degenerate in the sense that the ALE mapping ${\bf A}_\eta(t)$ remains injective in time
 as the fluid domain moves, and the Jacobian $J^\eta$ of the ALE mapping remains strictly positive, see Figure~\ref{degeneration}. 
 
 Degeneration of the fluid domain is associated with the ``contact problem'' between structures, as shown in  Figure~\ref{degeneration}. 
It is well known that due to the no-collision paradox associated with the no-slip condition \cite{HilTak2,HilTak,Starovoitov}, contact between two ``smooth'' structures 
is not possible in the case when the Jacobian $J^\eta$ becomes zero, corresponding to the situation in Figure~\ref{degeneration} right.
 This gives rise to various difficulties in the numerical simulation and modeling of problems such as, e.g., heart valve closure,
 if the no-slip boundary condition is used at the fluid-structure interface.
Our analysis presented in this paper is a first step in the direction towards studying contact between elastic structures in flows with slip boundary condition, which promises to shed new light on modeling of various biological phenomena, including the closure of human heart valves.
Further research in this direction is necessary.

\section{Acknowledgements}
Muha's research has been supported in part by the Croatian Science Foundation (Hrvatska Zaklada za Znanost) grant number 9477
and by the US National Science Foundation under grant DMS-1311709.
\v{C}ani\'{c}'s research has been supported by the US National Science Foundation under grants
DMS-1318763, DMS-1311709, DMS-1262385 (joint funding with the National Institutes of Health) and DMS-1109189. 

\vskip 0.2in
\noindent
{\bf Conflict of Interest Statement.} The authors confirm that they have no conflicts of interest indirectly or directly related to the presented research. 
\if 1 = 0
\noindent
{\bf Remark about the main result.} Notice that we proved only the local in time existence of a weak solution to the FSI problem 
with the slip condition holding at the fluid-structure interface. The length of the existence  time interval $(0,T)$ is determined by the time of contact 
between different parts of the structure. Namely, our existence result holds until the structure becomes in contact with another point on the structure,
i.e., until the injectivity of the ALE mapping, defined in \eqref{PositivityJ}, is satisfied. In contrast to the no-slip condition when 
contact is not possible in finite time (assuming certain regularity of the interface and the incompressibility of the fluid), see \cite{Grandmont},
the case with the slip condition allows contact between different parts of the structure, and so it would be non-physical to expect global 
existence of our weak solution for ``general'' data.
\fi

\if 1 = 0
\subsection{Appendix: An alternate splitting scheme}

Even though the presented splitting is well suited for the existence proof, its numerical implementation may have issues with lower accuracy. 
The reason for this is that the structure sub-problem does not incorporate any explicit forcing from the fluid, 
and ``feels'' the fluid only via the initial conditions. 
This problem has already been noted for the classical kinematically coupled scheme, and several modifications have been proposed to ensure better accuracy (see e.g. \cite{MarSun,SunBorMar,Fernandez3}). Here we propose another version of the Lie splitting scheme for Problem~\ref{SlipP},
which uses the fluid tangential stress to load the structure in order to increase the accuracy. 
In can be shown that the proposed modification also gives a stable scheme. 
We present here only the differential form of this modified Lie splitting scheme.

\noindent
{\bf THE STRUCTURE SUB-PROBLEM}:
\begin{equation}\label{StrcutureSubMod}
\left .
\begin{array}{c}
\displaystyle{\rho_S h\frac{{\bf v}^{n+1/2}-{\bf v}^n}{\Delta t}-{\mathcal L}_e\bdeta^{n+1}+\frac{1}{\alpha}({\bf v}^{n+1}\cdot\btau^n)\btau^n=\frac{1}{\alpha}({\bf u}^{n}\cdot\btau^n)\btau^n,}\\ \\
\displaystyle{\frac{\bdeta^{n+1}-\bdeta^n}{\Delta t}={\bf v}^{n+1/2},}
\end{array}\right \}\;{\rm on}\; \Gamma,
\end{equation}
{\bf THE FLUID SUB-PROBLEM}:
\begin{equation}\label{FluidSubMod}
\begin{array}{c}
\rho_F\displaystyle{\frac{{\bf u}^{n+1}-{\bf u}^n}{\Delta t}+\rho_F(({\bf u}^n-{\bf w}^{n+1})\cdot\nabla^{n+1}}) {\bf u}^{n+1}=\nabla^{{n+1}}\cdot \bsigma^{{n+1}}+{\bf R}^{n+1}\;{\rm in}\; \Omega,\\ \\
\displaystyle{\frac{{\bf v}^{n+1}-{\bf v}^{n+1/2}}{\Delta t}\cdot\bnu^{n+1}=-\bsigma^{{n+1}}\bnu^{{n+1}}\cdot\bnu^{n+1} ,\; ({\bf u}^{n+1}-{\bf v}^{n+1})\cdot\bnu^{{n+1}}=0}\;{\rm on}\; \Gamma,
\\ \\
\displaystyle{\Big ({\bf u}^{n+1}+\alpha \sigma^{{n+1}}\bnu^{{n+1}}\cdot\btau^{{n+1}}\Big )=({\bf v}^{n+1/2})\cdot\btau^{{n+1}}}\;{\rm on}\; \Gamma,
\end{array}
\end{equation} 
This scheme is better suited for a numerical implementation of an algorithm producing an approximation of the solution to
the fluid-structure interaction Problem \eqref{FSISlip} with the slip boundary condition at the moving interface. 

*******

TODO: Deal with the source term.

**********

\subsection{Proof overview}
In this section we give a the short overview of the proof and explain the main difficulties and novelties in comparison with the previous work. 
The main idea is to decouple the problem into the fluid and structure sub-problems using the 
Lie (Marchuk-Yanenko) operator-splitting method. In this scheme,
the sub-problems communicate via the initial conditions.
The resulting sub-problems are semi-discretized in time over some time interval $(0,T)$
via the backward Euler scheme. Each of the resulting semi-discretized sub-problems (the fluid and structure sub-problems)
is then linearized, giving rise to a sequence of linear elliptic sub-problems. The trick is to show that these
elliptic sub-problems, which are "coupled" via the initial conditions, give rise to the solutions that converge, 
as the time-discretization step $\Delta t$ approaches zero,
 to a weak solution of the entire coupled nonlinear problem of mixed parabolic-hyperbolic type.
The main steps of the convergence proof are summarized as follows:

TODO: rewrite the remainder of this section ***********

\begin{enumerate}
\item \textit{Formal energy inequality and construction of injective ALE mapping. Korn's inequality.} (Sections~\ref{Sec:EI} and \ref{ALEConstruction}). We begin by making some formal calculations which will be later rigorously justified by the construction of a weak solution. Namely, we formally derive the energy inequality. Assuming that the regularity of a weak solution is determined by the energy inequality we construct the ALE mapping as a harmonic extension of the interface displacement. Moreover, we use the standard Hilbert interpolation and the elliptic regularity on the domains with corners to prove the regularity of the ALE mapping. Finally we prove the injectivity of ALE mapping (Proposition~\ref{ALEinjectivity}) and the uniform Korn's inequality (Lemma~\ref{Korn}). 
\item \textit{Function spaces and definition of a weak solution.} (Section~\ref{Sec:Defweak}) Motivated by the energy inequality the appropriate functions spaces are formulated and the problem is re-formulated (Definition~\ref{WeakALE}) in the fixed reference by using the constructed ALE mapping. The main theorem is stated (Theorem~\ref{main}). 
\item \textit{Splitting scheme and uniform estimates.} (Section~\ref{Sec:Approx}) The problem is semi-discretized and split into the fluid and the structure sub-problems which are linear. The well-posedness of the each sub-problem is proved. Her we prove the semi-discrete version of the energy inequality uniformly in $\Delta t$ (Proposition~\ref{Stability}) which is one of the key ingredients of the proof.
\item \textit{Approximate solutions and weak convergence.} (Section~\ref{Sec:weak}) After obtaining the uniform estimates, we prove that the approximate solutions are well defined on some time-interval which does not depend on $\Delta t$ and that corresponding ALE mapping is injective on the same time-interval (Proposition~\ref{EtaBound}). Then we can prove the weak and weak* convergence results (Lemma~\ref{weak_convergence}) and identify its limits as a candidates for a weak solution.
\item \textit{Compactness and strong convergence.} (Section~\ref{Sec:strong}) The second key ingredient of the proof is to prove the strong convergence properties of the sequence of the approximate solutions. Since the function spaces change with the solution and there is no regularization of the structure velocity due to the fluid dissipation, we use Simon's characterization of compactness in $L^2(0,T;B)$ spaces \cite{Simon} and interpolation of the classical Sobolev spaces with real exponents $H^s$ (or alternatively Nikolskii spaces $N^{s,p} \cite{Simon2})$. 
\item \textit{Geometric conservation law and the limiting problem.} (Section~\ref{Sec:limiting}) In this section we construct the approximate test functions and prove their limiting properties (Lemmas~\ref{TestFunctionConv} and~\ref{ApproxDer}) which enables us to pass to the limit in the approximate equation~\eqref{ApproxEq}. The last important ingredient in the proof is the treatment of the geometric conservation law (Section~\ref{Sec:GSL}).
\end{enumerate} 
The main difficulties and novelties are the following:
\begin{enumerate}
\item Since we allow both, the tangential and the normal displacement, and more general geometry, the construction and analysis of ALE mapping is more involved and one have to prove that the deformation is injective on some time-interval.
\item Because of the more general geometrical setting one must take special care of the discretization of the terms connected to the geometric conservation law.
\item Our analysis include several different types of boundary conditions and a general linear shell/plate model. 
\item {\bf Due to the slip boundary condition, the tangential fluid velocity is not equal to the tangential structure velocity and therefore there is no additional regularization of the structure velocity due to the fluid dissipation.}
\item There are additional non-linearities in a weak formulation, since some geometrical quantities (surface measure, the tangent, the normal, the Jacobian of the ALE mapping) are now included explicitly in the weak formulation of the problem.
\end{enumerate}
\fi

\bibliographystyle{plain}
\bibliography{Bibliography/myrefs}

\begin{thebibliography}{10}

\bibitem{ADA}
Robert~A. Adams.
\newblock {\em Sobolev spaces}.
\newblock Academic Press [A subsidiary of Harcourt Brace Jovanovich,
  Publishers], New York-London, 1975.
\newblock Pure and Applied Mathematics, Vol. 65.

\bibitem{BarGruLasTuff2}
Viorel Barbu, Zoran Gruji{\'c}, Irena Lasiecka, and Amjad Tuffaha.
\newblock Existence of the energy-level weak solutions for a nonlinear
  fluid-structure interaction model.
\newblock In {\em Fluids and waves}, volume 440 of {\em Contemp. Math.}, pages
  55--82. Amer. Math. Soc., Providence, RI, 2007.

\bibitem{BarGruLasTuff}
Viorel Barbu, Zoran Gruji{\'c}, Irena Lasiecka, and Amjad Tuffaha.
\newblock Smoothness of weak solutions to a nonlinear fluid-structure
  interaction model.
\newblock {\em Indiana Univ. Math. J.}, 57(3):1173--1207, 2008.

\bibitem{BdV1}
Hugo Beir{\~a}o~da Veiga.
\newblock On the existence of strong solutions to a coupled fluid-structure
  evolution problem.
\newblock {\em J. Math. Fluid Mech.}, 6(1):21--52, 2004.

\bibitem{FSIforBio}
Tomas Bodnar, Giovanni~P. Galdi, and Sarka Necasova, editors.
\newblock {\em Fluid-Structure Interaction and Biomedical Applications}.
\newblock Birkh{\"a}user/Springer, Basel, 2014.

\bibitem{Bucur2010}
Dorin Bucur, Eduard Feireisl, and {\v{S}}{\'a}rka Ne{\v{c}}asov{\'a}.
\newblock Boundary behavior of viscous fluids: influence of wall roughness and
  friction-driven boundary conditions.
\newblock {\em Arch. Ration. Mech. Anal.}, 197(1):117--138, 2010.

\bibitem{MarSun}
Martina Bukac, Suncica Canic, Roland Glowinski, Josip Tambaca, and Annalisa
  Quaini.
\newblock Fluid-structure interaction in blood flow capturing non-zero
  longitudinal structure displacement.
\newblock {\em Journal of Computational Physics}, 235(0):515 -- 541, 2013.

\bibitem{Martina_Biot}
Martina Buka{\v{c}}, Ivan Yotov, and Paolo Zunino.
\newblock An operator splitting approach for the interaction between a fluid
  and a multilayered poroelastic structure.
\newblock {\em Numer. Methods Partial Differential Equations},
  31(4):1054--1100, 2015.

\bibitem{causin2005added}
Paola Causin, Jean-Fr{\'e}d{\'e}ric Gerbeau, and Fabio Nobile.
\newblock Added-mass effect in the design of partitioned algorithms for
  fluid-structure problems.
\newblock {\em Comput. Methods Appl. Mech. Eng.}, 194(42-44):4506--4527, 2005.

\bibitem{CDEM}
Antonin Chambolle, Beno{\^{\i}}t Desjardins, Maria~J. Esteban, and C{\'e}line
  Grandmont.
\newblock Existence of weak solutions for the unsteady interaction of a viscous
  fluid with an elastic plate.
\newblock {\em J. Math. Fluid Mech.}, 7(3):368--404, 2005.

\bibitem{NecasovaSlip}
Nikolai~Vasilievich Chemetov and {\v S}{\'{a}}rka Ne{\v c}asov{\'{a}}.
\newblock The motion of the rigid body in viscous fluid including collisions.
  global solvability result.
\newblock Preprint.

\bibitem{CiarletBook1}
Philippe~G. Ciarlet.
\newblock {\em Mathematical elasticity. {V}ol. {I}}, volume~20 of {\em Studies
  in Mathematics and its Applications}.
\newblock North-Holland Publishing Co., Amsterdam, 1988.
\newblock Three-dimensional elasticity.

\bibitem{CSS1}
Daniel Coutand and Steve Shkoller.
\newblock Motion of an elastic solid inside an incompressible viscous fluid.
\newblock {\em Arch. Ration. Mech. Anal.}, 176(1):25--102, 2005.

\bibitem{donea1983arbitrary}
Jean Don{\'e}a.
\newblock A {T}aylor-{G}alerkin method for convective transport problems.
\newblock In {\em Numerical methods in laminar and turbulent flow ({S}eattle,
  {W}ash., 1983)}, pages 941--950. Pineridge, Swansea, 1983.

\bibitem{Doyle}
Matthew~G. Doyle, Stavros Tavoularis, and Yves Bourgault.
\newblock Application of parallel processing to the simulation of heart
  mechanics.
\newblock In {\em Proceedings of the 23rd International Conference on High
  Performance Computing Systems and Applications}, HPCS'09, pages 30--47,
  Berlin, Heidelberg, 2010. Springer-Verlag.

\bibitem{Gunzburger}
Qiang Du, Max~D. Gunzburger, L.~Steven Hou, and Jeehyun Lee.
\newblock Analysis of a linear fluid-structure interaction problem.
\newblock {\em Discrete Contin. Dyn. Syst.}, 9(3):633--650, 2003.

\bibitem{GaldiHandbook}
Giovanni~P. Galdi.
\newblock On the motion of a rigid body in a viscous liquid: a mathematical
  analysis with applications.
\newblock In {\em Handbook of mathematical fluid dynamics, {V}ol. {I}}, pages
  653--791. North-Holland, Amsterdam, 2002.

\bibitem{GVHil}
David G{\'e}rard-Varet and Matthieu Hillairet.
\newblock Regularity issues in the problem of fluid structure interaction.
\newblock {\em Arch. Ration. Mech. Anal.}, 195(2):375--407, 2010.

\bibitem{GVHil3}
David G{\'e}rard-Varet and Matthieu Hillairet.
\newblock Existence of weak solutions up to collision for viscous fluid-solid
  systems with slip.
\newblock {\em Comm. Pure Appl. Math.}, 67(12):2022--2075, 2014.

\bibitem{GVHil2}
David G{\'e}rard-Varet, Matthieu Hillairet, and Chao Wang.
\newblock The influence of boundary conditions on the contact problem in a 3{D}
  {N}avier-{S}tokes flow.
\newblock {\em J. Math. Pures Appl. (9)}, 103(1):1--38, 2015.

\bibitem{Masmoudi2010}
David G{\'e}rard-Varet and Nader Masmoudi.
\newblock Relevance of the slip condition for fluid flows near an irregular
  boundary.
\newblock {\em Comm. Math. Phys.}, 295(1):99--137, 2010.

\bibitem{CG}
C{\'e}line Grandmont.
\newblock Existence of weak solutions for the unsteady interaction of a viscous
  fluid with an elastic plate.
\newblock {\em SIAM J. Math. Anal.}, 40(2):716--737, 2008.

\bibitem{Grisvard2}
Pierre Grisvard.
\newblock {\em Elliptic problems in nonsmooth domains}, volume~24 of {\em
  Monographs and Studies in Mathematics}.
\newblock Pitman (Advanced Publishing Program), Boston, MA, 1985.

\bibitem{GioSun}
Giovanna Guidoboni, Roland Glowinski, Nicola Cavallini, and Suncica Canic.
\newblock Stable loosely-coupled-type algorithm for fluid-structure interaction
  in blood flow.
\newblock {\em J. Comput. Phys.}, 228(18):6916--6937, 2009.

\bibitem{Gur}
Morton~E. Gurtin.
\newblock {\em An introduction to continuum mechanics}, volume 158 of {\em
  Mathematics in Science and Engineering}.
\newblock Academic Press Inc. [Harcourt Brace Jovanovich Publishers], New York,
  1981.

\bibitem{HilTak2}
Matthieu Hillairet.
\newblock Lack of collision between solid bodies in a 2{D} incompressible
  viscous flow.
\newblock {\em Comm. Partial Differential Equations}, 32(7-9):1345--1371, 2007.

\bibitem{HilTak}
Matthieu Hillairet and Tak{\'e}o Takahashi.
\newblock Collisions in three-dimensional fluid structure interaction problems.
\newblock {\em SIAM J. Math. Anal.}, 40(6):2451--2477, 2009.

\bibitem{Lukacova}
Anna Hundertmark-Zau{\v{s}}kov{\'a}, M{\'a}ria
  Luk{\'a}{\v{c}}ov{\'a}-Medvi{\soft{d}}ov{\'a}, and Gabriela
  Rusn{\'a}kov{\'a}.
\newblock Fluid-structure interaction for shear-dependent non-{N}ewtonian
  fluids.
\newblock In {\em Topics in mathematical modeling and analysis}, volume~7 of
  {\em Jind\u rich Ne\u cas Cent. Math. Model. Lect. Notes}, pages 109--158.
  Matfyzpress, Prague, 2012.

\bibitem{IgnatovaKukavica}
Mihaela Ignatova, Igor Kukavica, Irena Lasiecka, and Amjad Tuffaha.
\newblock On well-posedness for a free boundary fluid-structure model.
\newblock {\em J. Math. Phys.}, 53(11):115624, 13, 2012.

\bibitem{ignatova2014well}
Mihaela Ignatova, Igor Kukavica, Irena Lasiecka, and Amjad Tuffaha.
\newblock On well-posedness and small data global existence for an interface
  damped free boundary fluid--structure model.
\newblock {\em Nonlinearity}, 27(3):467, 2014.

\bibitem{JagerMikelic2001}
Willi J{\"a}ger and Andro Mikeli{\'c}.
\newblock On the roughness-induced effective boundary conditions for an
  incompressible viscous flow.
\newblock {\em J. Differential Equations}, 170(1):96--122, 2001.

\bibitem{Kuk}
Igor Kukavica and Amjad Tuffaha.
\newblock Solutions to a fluid-structure interaction free boundary problem.
\newblock {\em Discrete Contin. Dyn. Syst.}, 32(4):1355--1389, 2012.

\bibitem{Kuk2}
Igor Kukavica and Amjad Tuffaha.
\newblock Solutions to a free boundary problem of fluid-structure interaction.
\newblock {\em Indiana Univ. Math. J.}, 61:1817--1859, 2012.

\bibitem{KukavicaNSLame}
Igor Kukavica and Amjad Tuffaha.
\newblock Well-posedness for the compressible {N}avier-{S}tokes-{L}am\'e system
  with a free interface.
\newblock {\em Nonlinearity}, 25(11):3111--3137, 2012.

\bibitem{KukavicaTuffahaZiane}
Igor Kukavica, Amjad Tuffaha, and Mohammed Ziane.
\newblock Strong solutions for a fluid structure interaction system.
\newblock {\em Adv. Differential Equations}, 15(3-4):231--254, 2010.

\bibitem{LenRuz}
Daniel Lengeler and Michael R{\u u}{\v{z}}i{\v{c}}ka.
\newblock Weak solutions for an incompressible newtonian fluid interacting with
  a koiter type shell.
\newblock {\em Archive for Rational Mechanics and Analysis}, 211(1):205--255,
  2014.

\bibitem{Leq11}
Julien Lequeurre.
\newblock Existence of strong solutions to a fluid-structure system.
\newblock {\em SIAM J. Math. Anal.}, 43(1):389--410, 2011.

\bibitem{Leq13}
Julien Lequeurre.
\newblock Existence of {S}trong {S}olutions for a {S}ystem {C}oupling the
  {N}avier--{S}tokes {E}quations and a {D}amped {W}ave {E}quation.
\newblock {\em J. Math. Fluid Mech.}, 15(2):249--271, 2013.

\bibitem{LionsMagenes}
Jacques-Louis Lions and Enrico Magenes.
\newblock {\em Non-homogeneous boundary value problems and applications. {V}ol.
  {I}}.
\newblock Springer-Verlag, New York, 1972.
\newblock Translated from the French by P. Kenneth, Die Grundlehren der
  mathematischen Wissenschaften, Band 181.

\bibitem{LukacovaCMAME}
M{\'a}ria Luk{\'a}{\v{c}}ov{\'a}-Medvid'ov{\'a}, Gabriela Rusn{\'a}kov{\'a},
  and Anna Hundertmark-Zau{\v{s}}kov{\'a}.
\newblock Kinematic splitting algorithm for fluid-structure interaction in
  hemodynamics.
\newblock {\em Comput. Methods Appl. Mech. Engrg.}, 265:83--106, 2013.

\bibitem{MikelicReview}
Andro Mikeli{\'c}.
\newblock Rough boundaries and wall laws.
\newblock In {\em Qualitative properties of solutions to partial differential
  equations}, volume~5 of {\em Jind\u rich Ne\u cas Cent. Math. Model. Lect.
  Notes}, pages 103--134. Matfyzpress, Prague, 2009.

\bibitem{BorSun}
Boris Muha and Sun{\v c}ica {\v C}ani{\'c}.
\newblock Existence of a {W}eak {S}olution to a {N}onlinear
  {F}luid--{S}tructure {I}nteraction {P}roblem {M}odeling the {F}low of an
  {I}ncompressible, {V}iscous {F}luid in a {C}ylinder with {D}eformable
  {W}alls.
\newblock {\em Arch. Ration. Mech. Anal.}, 207(3):919--968, 2013.

\bibitem{muha2013nonlinear}
Boris Muha and Suncica Canic.
\newblock A nonlinear, 3d fluid-structure interaction problem driven by the
  time-dependent dynamic pressure data: a constructive existence proof.
\newblock {\em Communications in Information and Systems}, 13(3):357--397,
  2013.

\bibitem{SunBorMulti}
Boris Muha and Sun{\v{c}}ica {\v{C}}ani{\'c}.
\newblock Existence of a solution to a fluid--multi-layered-structure
  interaction problem.
\newblock {\em J. Differential Equations}, 256(2):658--706, 2014.

\bibitem{BorSunNonLinear}
Boris Muha and Sun{\v{c}}ica {\v{C}}ani{\'c}.
\newblock Fluid-structure interaction between an incompressible, viscous 3{D}
  fluid and an elastic shell with nonlinear {K}oiter membrane energy.
\newblock {\em Interfaces Free Bound.}, 17(4):465--495, 2015.

\bibitem{MT3}
Boris Muha and Zvonimir Tutek.
\newblock Note on evolutionary free piston problem for {S}tokes equations with
  slip boundary conditions.
\newblock {\em Commun. Pure Appl. Anal.}, 13(4):1629--1639, 2014.

\bibitem{NePen}
Ji{\v{r}}{\'{\i}} Neustupa and Patrick Penel.
\newblock A weak solvability of the {N}avier-{S}tokes equation with {N}avier's
  boundary condition around a ball striking the wall.
\newblock In {\em Advances in mathematical fluid mechanics}, pages 385--407.
  Springer, Berlin, 2010.

\bibitem{PlanasSueur}
Gabriela Planas and Franck Sueur.
\newblock On the ``viscous incompressible fluid + rigid body'' system with
  {N}avier conditions.
\newblock {\em Ann. Inst. H. Poincar\'e Anal. Non Lin\'eaire}, 31(1):55--80,
  2014.

\bibitem{Berlyand}
Mykhailo Potomkin, Vitaliy Gyrya, Igor Aranson, and Leonid Berlyand.
\newblock Collision of microswimmers in a viscous fluid.
\newblock {\em Physical Review E}, 87(5):053005, 2013.

\bibitem{Quarteroni2000}
Alfio Quarteroni, Massimiliano Tuveri, and Alessandro Veneziani.
\newblock Computational vascular fluid dynamics: problems, models and methods.
\newblock {\em Computing and Visualization in Science}, 2:163--197, 2000.
\newblock 10.1007/s007910050039.

\bibitem{raymond2013fluid}
Jean-Pierre Raymond and Muthusamy Vanninathan.
\newblock A fluid-structure model coupling the {N}avier-{S}tokes equations and
  the {L}am\'e system.
\newblock {\em J. Math. Pures Appl. (9)}, 102(3):546--596, 2014.

\bibitem{Starovoitov}
Jorge~Alonso San~Mart{\'{\i}}n, Victor Starovoitov, and Marius Tucsnak.
\newblock Global weak solutions for the two-dimensional motion of several rigid
  bodies in an incompressible viscous fluid.
\newblock {\em Arch. Ration. Mech. Anal.}, 161(2):113--147, 2002.

\bibitem{Simon}
Jacques Simon.
\newblock Compact sets in the space {$L^p(0,T;B)$}.
\newblock {\em Ann. Mat. Pura Appl. (4)}, 146:65--96, 1987.

\bibitem{Simon2}
Jacques Simon.
\newblock Sobolev, besov and nikolskii fractional spaces: Imbeddings and
  comparisons for vector valued spaces on an interval.
\newblock {\em Annali di Matematica Pura ed Applicata}, 157(1):117--148, 1990.

\bibitem{Tem}
Roger Temam.
\newblock {\em Navier-{S}tokes equations. {T}heory and numerical analysis}.
\newblock North-Holland Publishing Co., Amsterdam, 1977.
\newblock Studies in Mathematics and its Applications, Vol. 2.

\bibitem{Velcic}
Igor Vel{\v{c}}i{\'c}.
\newblock Nonlinear weakly curved rod by {$\Gamma$}-convergence.
\newblock {\em J. Elasticity}, 108(2):125--150, 2012.

\bibitem{WangFluidSolid}
Chao Wang.
\newblock Strong solutions for the fluid-solid systems in a 2-{D} domain.
\newblock {\em Asymptot. Anal.}, 89(3-4):263--306, 2014.

\bibitem{Zorlutuna}
Pinar Zorlutuna, Hasirci Hasirci, and Vasıf Hasirci.
\newblock Nanopatterned collagen tubes for vascular tissue engineering.
\newblock {\em Journal of tissue engineering and regenerative medicine},
  2(6):373--377, 2008.

\end{thebibliography}
\end{document}